\documentclass[12pt]{amsart}
\usepackage{amsmath, amsthm, amssymb,cite,enumitem}
\usepackage{fullpage}
\usepackage{color}
\usepackage{hyperref}
\usepackage[english]{babel}
\usepackage{framed}
\usepackage[utf8]{inputenc}
\usepackage{amsmath}
\usepackage{amsfonts}
\usepackage{amssymb}
\usepackage{calligra}
\usepackage{calrsfs}
\usepackage{yfonts}
\usepackage[mathscr]{euscript}

\usepackage{mathtools}

\usepackage{tikz}
\usepackage{pgfplots}
\pgfplotsset{compat=1.16}
\usetikzlibrary{intersections, pgfplots.fillbetween, patterns}

\newcommand{\enint}{E^{\text{in}}}
\newcommand{\enext}{E^{\text{ex}, \alpha}}

\newcommand{\enexh}{E^{\text{ex}, h}}
\newcommand{\enexhq}{E^{\text{ex}, h}_{\text{quasi}}}

\newcommand{\equasi}{E^{\text{in}}_{\text{quasi}}}
\newcommand{\exquasi}{E^{\text{ex},\alpha}_{\text{quasi}}}
\newcommand{\dint}{\mathcal{D}^{\text{in}}}
\newcommand{\dext}{\mathcal{D}^{\text{ex}}}

\newcommand{\CT}{\mathcal{C}_{[2, s]}}

\newcommand{\Sext}{\Sigma^{\text{ex}}}

\newcommand{\Xext}{X^{\text{ex}, \alpha}}

\newcommand{\scal}{\mathcal{S}}
\newcommand{\ec}{E^{\text{c,in}}}
\newcommand{\ecin}{E^{\text{c,in}}}
\newcommand{\exec}{E^{\text{c,ex}}}
\newcommand{\ecex}{E^{\text{c,ex}}}

\newcommand{\hin}{\mathcal{H}^{\text{in}}}
\newcommand{\hex}{\mathcal{H}^{\text{ex}}}
\newcommand{\op}{\overline{\partial}}

\newcommand{\hcal}{\mathcal{H}}

\newtheorem{theorem}{Theorem}
\newtheorem{lemma}{Lemma}
\newtheorem{proposition}[lemma]{Proposition}
\newtheorem{corollary}[lemma]{Corollary}

\numberwithin{lemma}{section}

\newcommand{\q}[1]{\mathcal{#1}}

\newcommand{\ep}{\epsilon}
\newcommand{\m}[1]{\mathbb{#1}}

\numberwithin{equation}{section}

\newcommand{\Wf}{\mathbf{W}}
\newcommand{\Ff}{\mathbf{F}}

\newcommand{\Wt}{\mathtt{W}}
\newcommand{\ut}{\mathtt{u}}

\newcommand{\rd}{\partial}

\begin{document}
	
	\title{Global well-posedness for a system of quasilinear wave equations on a product space}

\begin{abstract}
	We consider a system of quasilinear wave equations on the product space $\m R^{1+3}\times \m S^1$, which we want to see as a toy model for Einstein equations with additional compact dimensions. We show global existence for small and regular initial data with polynomial decay at infinity. The method combines energy estimates on hyperboloids inside the light cone and weighted energy estimates outside the light cone.
\end{abstract}

	\author{C\'ecile Huneau}
	\address{\'Ecole Polytechnique \& CNRS}
	\email{cecile.huneau@polytechnique.edu}
	
	\author{Annalaura Stingo}
	\address{\'Ecole Polytechnique}
	\email{annalaura.stingo@gmail.com, annalaura.stingo@polytechnique.edu}

	\maketitle
	
	\section{Introduction}
	
	In the present article we address the problem of global existence of small solutions to a certain class of quasilinear systems of wave equations on the product space $\mathbb{R}^{1+3}\times \m S^1$. The system under consideration has the following form
	\begin{equation}
		\label{syteme}
		\begin{cases}
			&  \Box_{x,y} u + u\partial_y^2 u = \sum_{1\le i,j\le 2} \textbf{N}_1(w_i,w_j)\\
			& \Box_{x,y} v+u\partial_y^2 v = \sum_{1\le i,j\le 2} \textbf{N}_2(w_i,w_j)
		\end{cases}\qquad (t,x,y)\in  \m R^{1+3}\times \m S^1
	\end{equation}
	with initial conditions set at time $t_0=2$
	\begin{equation} \label{data}
		(u,v)(2, x, y)=(\phi_0, \psi_0)(x,y), \qquad (\partial_t u, \partial_t v)(2,x,y)=(\phi_1,\psi_1)(x,y).
	\end{equation}
	In the above system $\Box_{x,y} = -\partial^2_t + \Delta_x + \partial_y^2$ denotes the D'Alembertian operator in the $(t,x,y)$ variables where $t\in \m R$ is the time coordinate, $x=(x_1, x_2, x_3)\in\mathbb{R}^3$ are the Cartesian coordinates and $y\in\m S^1$ is the periodic coordinate.
	The nonlinearities $\textbf{N}_1(\cdot, \cdot), \textbf{N}_2(\cdot, \cdot)$ are linear combination of the following quadratic null forms:
	\begin{equation}\label{null_forms}
		\begin{aligned}
			& Q_0(\phi,\psi)= \partial_t\phi\, \partial_t\psi - \nabla_x \phi \cdot \nabla_x\psi  \\
			& Q_{ij}(\phi,\psi)= \partial_i\phi\, \partial_j\psi - \partial_i\phi \,\partial_i\psi \qquad 1\le i<j\le 3\\
			& Q_{0i}(\phi,\psi) = \partial_t\phi \, \partial_i\psi -\partial_i \phi \, \partial_t \psi \qquad 1\le i \le 3
		\end{aligned}
	\end{equation}
	where $\partial_j$ denotes the derivative $\partial_{x_j}$ in the $j$th direction for $j=\overline{1,3}$, and $w_1, w_2=\{u,v\}$. 
	
	The main result we present in this paper asserts the global existence of solutions to \eqref{syteme} when the initial data are small and localized real functions. Our result also extends to semilinear interactions of the form $\partial_\alpha w_i \cdot\partial_y w_j$ with $\partial_\alpha$ being any of the derivatives in the $(t,x,y)$-variables.
	
	\subsection{Motivation and a brief history}
	Our interest in studying nonlinear wave equations on product spaces comes from the theory of supergravity (SUGRA) in physics and more precisely from the Kaluza-Klein theory, which represents
	the classical approach to the unification of general relativity with electromagnetism and more generally with gauge fields 
	(see  the orignal works of Kaluza \cite{Kaluza21} and Klein \cite{Klein26}). 
	The Kaluza-Klein approach considers general relativity in $3+1+d$ dimensions with space-time factorizing as
	\[
	\mathcal{M}^{(3+1+d)} = \mathbb{R}^{1+3}\times K
	\]
	where $K$ is a compact $d$-manifold referred to as \textit{internal space}. 
	In the simplest case $d=1$ dimensional gravity is compactified on a circle ($K=\mathbb{S}^1)$ to obtain at low energies a coupled Einstein-Maxwell-Scalar system in 3+1 space-time dimensions.
	Kaluza-Klein space-times $\m R^{1+3}\times \m S^1$ have been studied by Witten in his influential work \cite{WITTEN}, where he proves instability at the semiclassical level but provides heuristic arguments for classical stability.
	The first result proving the classical stability of the Kaluza-Klein theory is obtained by Wyatt in \cite{Wyatt18}, where only perturbations depending on the non-compact coordinates are considered, using tools developed by Lindblad and Rodninaski \cite{LR10}. More general spacetimes with supersymmetric compactifications $\mathcal{M}=\mathbb{R}^{1+n}\times K$ have recently been studied by Andersson, Blue, Wyatt, and Yau in \cite{ABWY}. The space-times $\mathcal{M}$ are equipped with the metric
	\[
	\hat{g} = \eta_{\mathbb{R}^{1+n}} + k
	\]
	where $\eta_{\mathbb{R}^{1+n}} $ is the Minkowski metric in $\mathbb{R}^{1+n}$ and $k$ is such that $(K,k)$ is a compact Ricci-flat Riemannian manifold having a cover that admits a spin structure and a nonzero parallel spinor. A global stability result is proved in \cite{ABWY} under the assumption $n\ge 9$ and for Cauchy data that are Schwarzschild near infinity, but it is conjectured that these conditions can be relaxed and that space-times with a supersymmetric compactification and $n=3$ are nonlinearly stable.
	
	In both the aforementioned works \cite{Wyatt18} and \cite{ABWY}, as well as in many other works concerning the global stability problem for Einstein equations, the use of the so-called \textit{wave-coordinates}  allows one to write the Einstein equations as a system of quasilinear wave equations on the metric $g=(g^{\alpha \beta})_{\alpha \beta}$
	\begin{equation}\label{waveccord}
		g^{\alpha \beta}\partial_{\alpha }\partial_{\beta} g_{\mu\nu}=P_{\mu \nu}(\partial g,\partial g) + G_{\mu\nu}(g)(\partial g, \partial g)
	\end{equation}
	where the sum is taken over repeated indices, $P_{\mu \nu}$ are quadratic forms and $G_{\mu\nu}(g)(\partial g, \partial g)$ contain cubic terms.
	In the present paper we focus on a toy model for the Einstein equations on $\m R^{1+3}\times \m S^1$ that only keeps a selection of terms from equation \eqref{waveccord}, precisely the semilinear terms with null structure and the quasilinear term $g^{yy}\partial^2_y$ where $y$ is the periodic coordinate on $\m S^1$. Our goal here is in fact to study the global well-posedness for quasilinear wave equations on the product space $\m R^{1+3}\times \m S^1$ without having to make use of the full structure of the Einstein equations. The unknown $u$ in system \eqref{syteme} plays the role of the coefficient $g_{yy}$ and the unknown $v$ encodes any other metric coefficient $g_{\mu\nu}$.
	
	We mention here that wave equations on product spaces also appear in other contexts, for instance when studying the propagation of waves along infinite homogeneous waveguides (see \cite{LeskyRacke, MSS05, MS2008, ettinger}).
	
	\smallskip
	One key observation when studying the small data global well-posedness problem for wave equations on product spaces $\m R^{1+3}\times \m T^d$ is that they are closely related to infinite systems coupling wave and Klein-Gordon equations on the flat space $\m R^{1+3}$. In fact, if $W=W(t,x,y)$ is solution to 
	\[
	\Box_{x,y}W = F\, \qquad (t,x,y)\in \m R^{1+3}\times \m T^d
	\]
	for some source term $F$, its Fourier modes in the periodic direction $\{W_k(t,x)\}_{k\in \m Z}$ solve the following equations
	\[
	(-\partial^2_t + \Delta_x)W_k - |k|^2 W_k = \int_{\m T^d} e^{-iy\cdot k} F dy, \qquad (t,x)\in \m R^{1+3}.
	\]
	In particular, the zero-mode $W_0$ is solution to a wave equation and any other non-zero mode $W_k$ solves a Klein-Gordon equation of mass $|k|$, all on the flat space $\m R^{1+3}$.
	
	The global well-posedness for systems coupling a finite number of wave and Klein-Gordon equations on the flat 3+1 space-time with small data have largely been studied. We cite the initial results by Georgiev \cite{Georgiev:system} and Katayama \cite{katayama:null_condition}, followed by LeFloch and Ma  \cite{LeFloch_Ma, LeFloch-Ma:global_nl_stability}, Wang \cite{Q.Wang, Q.Wang:E-KG} and Ionescu and Pausader \cite{IoPa, IP2} who study such systems as a model for the full Einstein-Klein-Gordon equations. In \cite{LeFloch_Ma, Q.Wang} global well-posedness is proved for compactly supported initial data and quadratic quasilinear nonlinearities that satisfy some suitable conditions, including the \textit{null condition} of Klainerman \cite{klainerman:null_condition} for self-interactions between the wave components of the solution. An idea used in these works is that of employing hyperbolic coordinates
	in the forward light cone; this was first introduced by Klainerman \cite{klainerman:global_existence} for Klein-Gordon equations and Tataru in the wave context \cite{Tataru2002}, and later reintroduced by LeFloch and Ma in \cite{LeFloch_Ma} under the name of \textit{hyperboloidal foliation method}.
	In \cite{IoPa} global regularity and scattering is proved in the case of small smooth initial data that decay at a suitable rate at infinity and nonlinearities that do not verify the null condition but present a particular resonant structure. We also cite the work by Dong and Wyatt \cite{DW20}, who prove global well-posedness for a quadratic semilinear interaction in which there are no derivatives on the massless wave component. Other related results are \cite{Bachelot88, OTT_KGZ, T96_KGZ, T03_DP, T03_MH, klainerman_wang_yang, DLFW} and see \cite{ma:2D_semilinear, Ma2020, ma:2D_tools, ma:2D_quasilinear, stingo_WKG, IS2019, DW20_2d, ma:1D_semilinear} for results about wave-Klein-Gordon systems in lower dimensions.
	
	\smallskip
	Our goal in this paper is to prove the global stability for \eqref{syteme} in the case where the initial data are not compactly supported but only have a mild polynomial decay at infinity. Our approach makes use of the vector field method by Klainerman \cite{CK93} and follows LeFloch and Ma \cite{LeFloch-Ma:global_nl_stability} in that a big portion of the estimates recovered for the solution in the interior of the cone $\{t=r+1\}\times \m S^1$ are estimates on hyperboloids. The main difference with \cite{LeFloch_Ma} is that the interior estimates need to be coupled with exterior estimates in the region outside the cone. Those are weighted energy estimates that we have to propagate in time.

	\subsection{The main result}
	In order to describe the initial data we consider for our problem we introduce the energy space $\hcal^0$ endowed with the norm
	\[
	\left\|(u[t],v[t])\right\|^2_{\mathcal{H}^0} := \|u\|^2_{H^1_{xy}} + \|u_t\|^2_{L^2} + \|v\|^2_{H^1_{xy}}+ \|v_t\|^2_{L^2}
	\]
	and the higher order energy spaces $\mathcal{H}^n$ for $n\ge 1$ endowed with the norm
	\[
	\left\|(u[t],v[t])\right\|^2_{\hcal^n} := \sum_{|\alpha| \le  n} \left\|(\partial^\alpha_{xy} u[t],\partial^\alpha_{xy} v[t])\right\|^2_{\hcal^0}
	\]
	where we use the following notation for the Cauchy data in \eqref{data} at time $t$:
	\[
	(u[t],v[t]) := (u(t), u_t(t), v(t), v_t(t)).
	\]
	The global well-posedness result that is the object of this paper is proved under some decay assumptions on the initial data.
	A premilinary version of our main theorem states the following
	\begin{theorem}\label{main}
		Assume the initial data $(u[2], v[2])$ for \eqref{syteme} satisfy
		\[
		\sum_{k=0}^5 \left\| \langle x\rangle^{\le k + {\frac{\alpha}{2}}}\partial^k_{xy}(u[2], v[2])\right\|_{\hcal^0} \le \ep \ll 1,
		\]
		for some positive fixed $\alpha$ and $\langle x \rangle = \sqrt{1+|x|^2}$.
		Then the system \eqref{syteme} is globally well-posed in the space $\hcal^5$.
	\end{theorem}
	We remark here that our choice to set the initial data at time $t_0=2$ over the conventional $t_0=0$ is more convenient for our computations and comes at no expense as the system \eqref{syteme} is invariant under time translations.
	
	\subsection{The wave-Klein-Gordon structure} The Cauchy problem \eqref{syteme}-\eqref{data} can be written in a more compact form as a vector equation for the unknown $W = (u,v)^T$
	\begin{equation}\label{eq_W}
		\Box_{x,y}W +u \partial^2_y W = \mathbf{N}(W,W)
	\end{equation}
	with data
	\begin{equation}\label{data_W}
		W|_{t=2} = \Phi_0, \qquad \partial_t W|_{t=2} = \Phi_1
	\end{equation}
	where $\Phi_0 = (\phi_0, \psi_0)^T$ and $\Phi_1 = (\phi_1, \psi_1)^T$ and
	\[
	\mathbf{N}(W,W) = \sum_{1\le i,j\le 2}
	\begin{pmatrix}
		\textbf{N}_1(w_i,w_j) \\
		\textbf{N}_2(w_i,w_j)
	\end{pmatrix}.
	\]
	The projection of $W$ onto the periodic direction $y$ reveals the nature of the equation \eqref{eq_W} as a system coupling one (vector) wave equation with an infinite sequence of (vector) Klein-Gordon equations of mass $|k|$ with $k\in\mathbb{Z}^*$. If we denote by $W_k = (u_k, v_k)^T$ the projection of $W$ onto the $k$-th frequency
	\[
	W_k(t,x) = \int_{\m S^1} e^{-iky} W(t,x,y)\, dy, \qquad k\in\mathbb{Z}
	\]
	we see that the functions $\{W_k\}_k$ satisfy the following coupled system
	\[
	\begin{cases}
		(-\partial^2_t + \Delta_x) W_k - |k|^2(1+u_0) W_k= \int_{\m S^1} e^{-iky} \mathbf{N}(W,W)\, dy - \int_{\m S^1} e^{-iky} (u-u_0) \partial^2_y W\, dy \\
		k\in\mathbb{Z}.
	\end{cases}
	\]
	The zero mode $W_0$ is solution to a wave equation while any other non-zero mode $W_k$ is solution to a nonlinear Klein-Gordon equation of mass $|k|$. For our analysis this distinction will be fundamental and we will often work throughout the paper with the following decomposition of $W$
	\begin{equation} \label{dec}
		W = W_0 + \mathtt{W}, \qquad   \mathtt{W}(t,x,y) = \sum_{k\ne 0} e^{iky}W_k(t,x),
	\end{equation}
	so that the equation \eqref{eq_W} is equivalent to the following system
	\begin{equation}\label{syst_freq}
		\begin{cases}
			(-\partial^2_t + \Delta_x) W_0 = \int_{\m S^1}\mathbf{N}(W,W)\, dy + \int_{\m S^1} \partial_y \mathtt{u}\, \partial_y \mathtt{W}\, dy\\[5pt]
			(-\partial^2_t + \Delta_x) \Wt+ (1+ u)\,  \partial^2_y \mathtt{W} = \mathbf{N}(W,W) - \int_{\m S^1}\mathbf{N}(W,W)\, dy - \int_{\m S^1} \partial_y \mathtt{u}\, \partial_y \mathtt{W}\, dy.
		\end{cases}
	\end{equation}
	We are now able to state a more precise version of the main theorem.
	\begin{theorem}
		Assume that for some positive fixed $\alpha$ the initial data for \eqref{eq_W} satisfy
		\[
		\sum_{k=0}^5\left\|x^{\le k+\frac{\alpha}{2}}\partial^k_{xy} W[2] \right\|_{\mathcal{H}^0}\le \epsilon \ll 1.
		\]
		Then the solution $W$ to \eqref{eq_W}-\eqref{data_W} exists globally in time in $\mathcal{H}^5$ and the two components of the solution $W_0=\int_{\m S^1}W\, dy$ and $\mathtt{W}=W-W_0$ satisfy the following pointwise bounds
		\[
		|\partial^j_{tx} W_0(t,x)|\lesssim \epsilon \langle t+|x|\rangle^{-1}\langle t-|x|\rangle^{-\frac{1}{2}}, \quad j=\overline{1,3}
		\]
		\[
		\left\|\partial^j_y \partial^k_{tx} \mathtt{W}(t,x,\cdot)\right\|_{L^2_y(\m S^1)}\lesssim \epsilon \langle t+|x|\rangle^{-\frac{3}{2}}, \quad j=\overline{0,1},\,  k=\overline{0,1},
		\]
		\[
		\left\|\partial^j_y \partial^2_{tx} \mathtt{W}(t,x,\cdot)\right\|_{L^2_y(\m S^1)}\lesssim \epsilon \langle t+|x|\rangle^{-1}\langle t-|x|\rangle^{-\frac{
				1}{2}}, \quad j=\overline{0,1}.
		\]
	\end{theorem}
	\subsection{Vector Fields}
	
	In order to describe the global bounds and decay properties of the solution $W=(u,v)^T$ to \eqref{eq_W}-\eqref{data_W} we need to introduce the family of Killing vector fields associated to our problem. Those are the vector fields that exactly commute with $\Box_{x,y}$:
	\begin{align}
		&\partial_t, \partial_1, \partial_2, \partial_3, \partial_y \label{der}\\
		&\Omega_{ij} = x_j \partial_i - x_i \partial_j \qquad 1\le i<j\le 3\label{rotation}\\
		&\Omega_{0i} = t\partial_i + x_i \partial_t \qquad i=\overline{1,3.}\label{boost}
	\end{align}
	The expressions in \eqref{der} correspond to the translations in the coordinate directions; \eqref{rotation} correspond to the Euclidean rotations in the $x$ coordinates; \eqref{boost} are the hyperbolic rotations, also called boosts. 
	We also introduce the conformal scaling vector field
	\begin{equation}
		\scal  = t\partial_t + x\cdot \nabla_x
	\end{equation}
	which is not Killing for \eqref{eq_W} but will appear later in the analysis of the problem.
	We refer to \eqref{rotation} and \eqref{boost} as \emph{Klainerman vector fields} and generally denote them by $Z$:
	\[
	Z :=\{\Omega_{ij}, \Omega_{0i}\}.
	\]
	We denote the full set of admissible vector fields for $\Box_{x,y}$ as
	\begin{equation}\label{family_Z}
		\mathcal{Z}:=\{\partial_t, \partial_1, \partial_2, \partial_3, \partial_y, \Omega_{ij}, \Omega_{0i}\}
	\end{equation}
	and for any multi-index $\gamma=(\alpha, \beta)$ we denote
	\[
	\mathcal{Z}^\gamma = \partial^\alpha Z^\beta.
	\]
	For any two non-negative integers $k,n$ with $k\le n$ we say that the multi-index $\gamma = (\alpha, \beta)$ is of type $(n,k)$ if
	$|\gamma| = |\alpha|+|\beta| \le n$ and $|\beta|\le k$, in other words if there are at most $k$ Klainerman vector fields among the $|\gamma|$ admissible vector fields in the product $\mathcal{Z}^\gamma$.

	\subsection{The null structure}
	The nonlinearities we consider in this work are linear combination of the classical quadratic null forms \eqref{null_forms}. An important feature of the null forms is that they are combination of three types of products and can be expressed schematically as follows
	\begin{equation}\label{rewrite}
		\mathbf{N}(\phi, \psi) =\overline{\partial} \phi \cdot \partial\psi + \partial\phi \cdot \overline{\partial}\psi + \frac{t-r}{t} \partial \phi\cdot \partial\psi
	\end{equation}
	where $\op_j = t^{-1}\Omega_{0j}$ are the rescaled hyperbolic rotations, or also as
	\begin{equation}\label{rewrite1}
		\mathbf{N}(\phi, \psi) =\mathcal{T} \phi \cdot \partial\psi + \partial\phi \cdot \mathcal{T}\psi 
	\end{equation}
	where $\mathcal{T}_j=\partial_j + \frac{x_j}{r}\partial_t$ for $j=\overline{1,3}$ are the vector fields tangent to the cones $\{t-r=const\}$. The $\mathcal{T}$ vector fields are related to the boosts in general via the following relation
	\[
	\mathcal{T}_j = \frac{1}{t}\Omega_{0j} + \frac{(t-r)}{t}\frac{x_j}{r} \partial_t.
	\]
	We will often make use of the two representations \eqref{rewrite} and \eqref{rewrite1} to recover suitable pointwise and energy estimates for the solution as they allow to recover additional decay for the interactions $W_0\times W_0$.
	
	\subsection{The interior and exterior region}
	The proof of our main theorem is based on the combination of a classical local existence result with a bootstrap argument.
	We will perform such argument separately in the two regions in which we decompose our space-time:
	\[
	\text{interior region}\qquad	\dint := \left\{(t,x):t\ge 2 \text{ and } |x|< t-1 \right\} \times \m S^1
	\]
	\[
	\text{exterior region}\qquad		\dext := \left\{(t,x): t\ge 2 \text{ and } |x|\ge t-1 \right\} \times \m S^1.
	\]
	In order to describe our bootstrap assumptions we first introduce some notations.
	Given any hyperboloid $\q H_s$ in $\m R^{1+3}\times \m S^1$ we denote by $\hin_s$ (resp. $\hex_s$) the branch of $\q H_s$ contained in the interior region $\dint$ (resp. in the exterior region $\dext$):
	\[
	\q H_s =\left\{(t, x): s^2 = t^2-|x|^2\right\}\times \m S^1
	\]
	\[
	\begin{aligned}
		\hin_s &: =\left\{(t,x, y)\in \q H_s :  |x|<(s^2-1)/2\right\}, \\
		\hex_s & :=\left\{(t,x, y)\in \q H_s :  |x|\ge (s^2-1)/2\right\}.
	\end{aligned}
	\]
	Moreover we denote by $\hin_{[2,s]}$ the hyperbolic interior region above $\hin_2$ and below $\hin_s$, and by $\hex_{[2,s]}$ the portion of the exterior region below $\hex_s$, for any $s\ge 2$ (see figure \ref{fig:hin})
	\[
	\hin_{[2,s]} :=\left\{(t,x, y) \in \dint: 2\le t^2-|x|^2\le s^2\right\}
	\]
	\[
	\hex_{[2, s]} :=\left\{(t,x, y)\in \dext: t^2 - |x|^2\le s^2\right\}.
	\]
	
	In the interior region the bootstrap assumptions will be energy bounds on the truncated hyperboloids $\hin_s$ for $s\ge 2$ and pointwise bounds on the $Z$ derivative of the zero mode of the solution. The local wellposedness theory for this problem assures the existence and smallness of the solution $W=(u,v)^T$ to \eqref{eq_W}-\eqref{data_W} up to the interior hyperboloid $\hin_2$, hence our goal will be to propagate the bootstrap assumptions in the hyperbolic interior region above $\hin_2$
	\[
	\hin_{[2,\infty)} :=\left\{(t,x, y) \in \dint: 2\le t^2-|x|^2\right\}.
	\]
	
	In the exterior region	the bootstrap assumptions will instead be weighted energy bounds on the constant time slices $\Sext_t$ which foliate $\dext$ for $t\ge 2$,
	\begin{align*}
		\Sext_t &:= \{x\in \m R^3 : |x|\ge t-1\} \times \m S^1.
	\end{align*} 
	
	\begin{figure}[h]
\begin{center}
\begin{tikzpicture}[scale=1.9]

\draw[->] (0,-0.3) -- (0,3);
\draw[->] (-0.3,0) -- (3.1,0);
\node[below] at (3,0) {\small $r$};
\node[left] at (0,2.9) {\small $t$};

\draw[red,thick, name path=A] [domain=0:2] plot(\x, {((1.44)^2+(0.8)*(\x)^2)^(0.5)});
\node[red, above, left] at (1.8, 2.2) {\tiny $\hin_s$};

\draw[dashed] [domain=2:2.7] plot(\x, {((1.44)^2+(0.8)*(\x)^2)^(0.5)});

\draw[red, thick] [domain = 0.45: 2] plot(\x, {\x +0.3});

\draw[red,thick, name path=B] [domain=0:0.45] plot(\x, {((0.6)^2+(\x)^2)^(0.5)});

\node[red, right] at (0.1,1.1) {\tiny $\hin_{[2,s]}$};

\draw[dashed] [domain=0:2.5] plot(\x,\x+0.3);
\node[above, left] at (2.5,2.5+0.3) {\tiny $t=r+1$};

\draw[dashed] [domain=0:2.8] plot(\x,\x);
\node[above, right] at (2.8,2.8) {\tiny $t=r$};

\tikzfillbetween[of=A and B]{red!40,nearly transparent}

\end{tikzpicture}
\caption{Vertical section of the region $\hin_{[2, s]}$ projected onto $\m R^{1+3}$}
\label{fig:hin}
\end{center}
\end{figure}
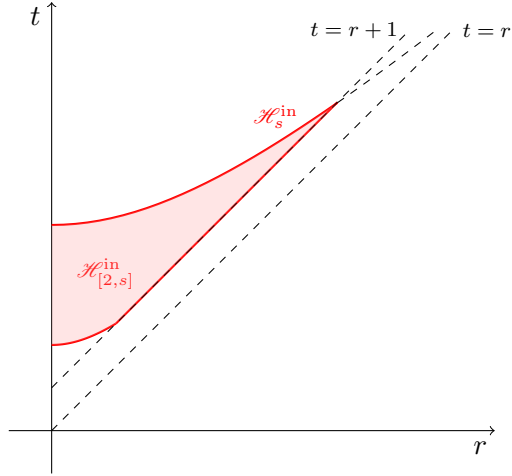
	
	We warn the reader that throughout the paper we will work with functions defined on the product space $\m R^{1+3}\times \m S^1$ as well as with functions not depending on the $y$ variable and defined on the flat space $\m R^{1+3}$. With the purpose of keeping notations as light as possible, a region in $\m R^{1+3}\times \m S^1$ and its projection onto $\m R^{1+3}$ will have the same name.
	
	\subsection{The energy functionals} \label{Energy_functionals}
	In the interior region we aim to propagate a-priori energy bounds on the solution on the truncated hyperboloids $\hin_s$. 
	The interior energy functional on $\hin_s$ associated to the linear counterpart of equation \eqref{eq_W} is
	\begin{equation}\label{enint}
		\begin{aligned}
			\enint(s, W) &: =  \iint_{\hin_s} \left(\frac{s}{t}\right)^2|\partial_tW|^2 + |\op W|^2 + |\partial_yW|^2  dxdy \\
			& = \iint_{\hin_s} \left(\frac{s}{t}\right)^2|\partial_xW|^2 + |\scal W|^2 + t^{-2}\sum_{1\le i<j\le 3}|\Omega_{ij}W|^2 + |\partial_yW|^2\, dxdy
		\end{aligned}
	\end{equation}
	where $|\op W|^2 $ stands for $\sum_{i=1}^3 |\op_i W|^2$.
	The above functional can be decomposed  as follows
	\[
	\enint(t, W) = \enint(t, W_0) + \enint(t, \Wt)
	\]
	where $\enint(t, \Wt)$ is defined by replacing $W$ with $\Wt$ in \eqref{enint} and
	\[
	\begin{aligned}
		\enint(s, W_0) &: =  \int_{\hin_s} \left(\frac{s}{t}\right)^2|\partial_tW_0|^2 + |\op W_0|^2  dx \\
		& = \int_{\hin_s} \left(\frac{s}{t}\right)^2|\partial_xW_0|^2 + |\scal W_0|^2 + t^{-2}\sum_{1\le i<j\le 3}|\Omega_{ij}W_0|^2 \, dx.
	\end{aligned}
	\]
	We also introduce and work with the conformal energy functional on truncated hyperboloids associated to the linear flat wave equation on $\m R^{1+3}$. This is the functional defined as follows
	\begin{equation}\label{conf_en}
		\ec(s, W_0):=\int_{\hin_s} \frac{1}{t^2}\left|K W_0 + 2t W_0\right|^2 + \frac{s^2}{t^2}\sum_{i=0}^3|\Omega_{0i}W_0|^2\, dx,
	\end{equation}
	where $K=(t^2+r^2)\partial_t + 2rt\partial_r$ is the Morawetz multiplier.
	
	\medskip
	In the exterior region we will describe the evolution of equation \eqref{eq_W} by means of the following weighted energy functional
	\begin{equation} \label{enext}
		\enext(t, W) = \iint_{\Sext_t} (2+r-t)^{\alpha+1} \left[ |\partial_t W|^2 + |\nabla_x W|^2 + |\partial_y W|^2\right]\, dxdy
	\end{equation}
	and of stronger norm $\Xext_T$
	\begin{equation} \label{Xext}
		\left\|W\right\|^2_{\Xext_T} :=\sup_{t\in [2,T]}\enext(t,W) +(1+\alpha) \int_2^T \iint_{\Sext_t} (2+r-t)^\alpha \left( |\mathcal{T}W|^2  + |\partial_y W|^2  \right) \, dxdydt
	\end{equation}
	where $|\mathcal{T}W|^2 $ stands for $\sum_{i=1}^3 |\mathcal{T}_iW|^2$.
	In the above energy functionals we have $T>2$ and $\alpha>0$.
	The bootstrap assumptions in this region will be energy bounds on the $\Xext_T$ norm of the solution, which not only controls the weighted energy of the solution but also the weighted $L^2$ space-time norm of the \textit{good derivatives}: the tangential derivatives $\q T$ to the cones $\{t-r=const\}$ and the derivative along the periodic direction $\partial_y$.
	
	As a result of the global energy bounds that will be proved to hold in the exterior region (see Section \ref{Sec:Exterior}) we also obtain a control on the energy of $W$ on the exterior hyperboloids
	\begin{equation}\label{ext_en_hyp}
		\begin{aligned}
			\enexh(s, W)& =   \iint_{\hex_s} \left(\frac{s}{t}\right)^2|\partial_tW|^2 + |\op W|^2 + |\partial_yW|^2  dxdy \\
			& = \iint_{\hex_s} \left(\frac{s}{t}\right)^2|\partial_xW|^2 + |\scal W|^2 + t^{-2}\sum_{1\le i<j\le 3}|\Omega_{ij}W|^2 + |\partial_yW|^2\, dxdy
		\end{aligned}
	\end{equation}
	as well as on the exterior conformal energy of $W_0$ on constant time slices $\Sext_t$
	\begin{equation}\label{conf_enext}
		\ecex(t, W_0):=\int_{\Sext_t} \left|\scal W_0 + 2W_0\right|^2 + \sum_{i=0}^3 |\Omega_{0i}W_0|^2\, dx.
	\end{equation}
	\subsection{The quasilinear energies}
	
	Equation \eqref{eq_W} is quasilinear and in order to propagate both the interior and exterior a-priori energy bounds for the solution we need to consider a cubic modification of the energies introduced in the previous subsection. Such \textit{quasilinear energies} are defined as follows
	\[
	\equasi(s, W):=\enint(s, W) + \iint_{\hin_s} u|\partial_yW|^2\, dxdy
	\]
	\[
	\exquasi(t, W):=\enext(t, W) + \iint_{\Sext_t}u|\partial_yW|^2\, dxdy
	\]
	\[
	\enexhq(s, W):=\enexh(s, W) + \iint_{\hex_s}u|\partial_yW|^2\, dxdy.
	\]
	We also introduce the quasilinear modification of the stronger norm $\Xext_T$ \small
	\[
	\|W\|^2_{\Xext_{\text{quasi}, T}} : =\sup_{t\in[2, T]}\exquasi(t, W) + (1+\alpha)\int_2^T \iint_{\Sext_t} (2+r-t)^\alpha \left( |\mathcal{T}W|^2  +(1+u) |\partial_y W|^2  \right) \, dxdydt.
	\]\normalsize
	We immediately observe that under smallness assumptions on $u$, e.g. $|u|\le 1/10$,
	our starting energies are equivalent to their corresponding quasilinear counterparts, i.e. for any $E = \{\enint, \enext, \enexh\}$ 
	\[
	\frac{9}{10}E(t,W) \le E_{\text{quasi}}(t, W)\le \frac{11}{10}E(t, W) .
	\]
	The same holds true for the stronger norms $\Xext_T$ and $\Xext_{\text{quasi}, T}$.
	
	\subsection{Higher order norms}
	We use the vector fields introduced before to define the higher order counterparts of the energy functionals and of the stronger exterior norm $\Xext_T$
	\[
	E_{n}(s, W) := \sum_{|\gamma|\le n}E(s, \mathcal{Z}^\gamma W), \qquad E = \{\enint, \enext, \enexh\}
	\]
	\[
	\|W\|_{X^{n, \alpha}_T} :=\sum_{|\gamma|\le n} \|\mathcal{Z}^\gamma W\|_{\Xext_T}.
	\]
	The higher order energies of $W_0$ and $\Wt$ are defined analogously. We observe that the above higher order energies control the high Sobolev regularity of the solution in the interior and exterior region respectively and also keep track of the $Z$ vector fields applied to the solution in addition to usual derivatives. In the interior region it will be important to keep track of the precise number of Klainerman vector fields acting on the $\Wt$ component of the solution and to that purpose we also introduce the following energy
	\[
	\enint_{n,k}(s, \Wt) := \sum_{\q I_{n,k}}E(s, \q Z^\gamma \Wt),
	\]
	where $\q I_{n,k}$ denotes the set of indices of type $(n,k)$.
	We finally introduce the higher order counterparts of the conformal energy functionals \eqref{conf_en} and \eqref{conf_enext} in order to control the conformal energies of pure products of Klainerman vector fields acting on $W_0$
	\[
	\ec_{n}(s, W_0):=\sum_{|\beta|\le n}\ec(s, Z^\beta W_0), \qquad \ecex_{n}(t, W_0):= \sum_{|\beta|\le n}\ecex(s, Z^\beta W_0).
	\]

	\subsection{Notations} Throughout the rest of the paper we will use the following notations:
	\begin{itemize}
		\item[\tiny{$\bullet$}] $r=|x|$
		\item[\tiny{$\bullet$}] $\nabla = (\partial_1, \partial_2, \partial_3)$ denotes spatial derivatives
		\item[\tiny{$\bullet$}]$\partial = \{\partial_t, \partial_j, \partial_y\}$ denotes space-time derivatives
		\item[\tiny{$\bullet$}] $\partial_{tx}=\{\partial_t, \partial_j\}$ denotes flat space-time derivatives
		\item[\tiny{$\bullet$}] $\op_j = t^{-1}\Omega_{0j}$ and $\q T_j = \partial_j + \frac{x_j}{r}\partial_t$ for $j=\overline{1,3}$
		\item[\tiny{$\bullet$}] $\q Z^{\le n} = \sum_{|\gamma|\le n}\q Z^\gamma$ 
	\end{itemize}

	\subsection*{Acknowledgments} This material is based upon work supported by the National Science Foundation under Grant No. DMS-1929284 while the second author was in residence at the Institute for Computational and Experimental Research in Mathematics in Providence, RI, during the \textit{Hamiltonian Methods in Dispersive and Wave Evolution Equations} program. The second author was also supported by an AMS Simons travel grant. The first author was supported by the ANR-19-CE40-0004.
	
	\section{Overview of the proof} 
	
	The proof of our main theorem is based on the combination of a classical local well-posedness result for equation \eqref{eq_W} with a bootstrap argument. We will perform this argument separately in the interior region $\dint$ and the exterior region $\dext$ in which we divide the space-time $\m R^{1+3}\times \m S^1$.
	
	\smallskip
	The bootstrap assumptions in the exterior region $\dext$ are uniform-in-time energy bounds on the higher order stronger norm $X^{5, \alpha}_{T_0}$ of the solution $W$, for any fixed $\alpha>0$:
	\begin{equation}\label{boot_energy}
		\|W\|^2_{X^{5, \alpha}_{T_0}}\le 2C^2_0\ep^2.
	\end{equation}
	Such bound in particular implies the existence of some function $l\in L^1([2, T_0])$ such that
	\begin{equation}\label{l(t)}
		\iint_{\Sext_t} (2+r-t)^\alpha \left( |\mathcal{T} \q Z^{\le 5} W|^2  + |\partial_y \q Z^{\le 5} W|^2  \right) \, dxdy \le 2C^2_0\ep^2 l(t).
	\end{equation}
	
	The result we want to prove in the exterior region is the following
	\begin{proposition}\label{prop:boot_ext}
		There exists a constant $C_0>0$ sufficiently large and a constant $\epsilon_0>0$ sufficiently small such that, for every $0<\epsilon<\epsilon_0$ if $W=(u,v)^T$ is a solution to \eqref{eq_W}-\eqref{data_W} in an interval $[2,T_0]$ and satisfies the energy bounds \eqref{boot_energy} then actually
		\[
		\left\| W\right\|^2_{X^{5, \alpha}_{T_0}}\le C_0^2\ep^2.
		\]
	\end{proposition}
	In the above proposition the time $T_0$ is arbitrary, therefore the solution $W$ exists globally in $\dext$ and satisfies the energy bound \eqref{boot_energy} for all times $T_0>2$. In particular, one has that
	\begin{equation}\label{global_strongnorm}
		\|\q Z^{\le 5}W \|_{\Xext_\infty} := \lim_{T_0\rightarrow\infty}\|\q Z^{\le 5}W\|_{\Xext_{T_0}} \le 2C_0^2\ep^2.
	\end{equation}

	\medskip
	The bootstrap assumptions in the interior region are higher order energy bounds on hyperboloids for $W_0$ and $\Wt$ and pointwise bounds on the $Z$ derivative of $W_0$
	\begin{equation}\label{booten1}
		\enint_{5}(s, W_0) \le 2A^2\ep^2
	\end{equation}
	\begin{equation}\label{booten2}
		\enint_{5,k}(s, \Wt)\le 2A^2\ep^2 s^{2\delta_k}, \qquad k=\overline{0,5}
	\end{equation}
	\begin{equation} \label{bootZ}
		|ZW_0(t,x)|\le 2B \ep t^{-1}s^\sigma
	\end{equation}
	for all $s\in [2, s_0]$, $(t,x)\in \hin_{[2, s_0]}$, and $s_0>2$ fixed. Here $\delta_0=0$, the parameters $\sigma, \delta_k$s are fixed small universal constants satisfying $0<\sigma\ll \delta_{k}\ll \delta_{k+1}$ for $k=\overline{1,4}$, and $A$ and $B$ are large universal constants which we will improve as a part of the conclusion of the proof.
	The result we want to prove in this region requires the global exterior energy bounds \eqref{global_strongnorm} and can be stated as follows
	\begin{proposition}
		There exist two constants $A, B>0$ sufficiently large, $0<\ep_0, \sigma, \delta_k\ll 1$ sufficiently small with $\delta_0=0$ and $\sigma\ll \delta_{k}\ll \delta_{k+1}$ for $k=\overline{1,4}$ such that, for every $0<\ep<\ep_0$ if $W=(u,v)^T$ is a solution to \eqref{eq_W}-\eqref{data_W} in the region $\hin_{[2, s_0]}\cup \dext$ and satisfies the global exterior energy bounds \eqref{global_strongnorm} as well as the interior bounds \eqref{booten1}- \eqref{bootZ} for all $s\in [2, s_0]$, then it actually satisfies the enhanced interior bounds
		\[
		\enint_{5}(s, W_0) \le A^2\ep^2
		\]
		\[
		\enint_{5,k}(s, \Wt)\le A^2\ep^2 s^{2\delta_k}, \qquad k=\overline{0,5}
		\]
		\[
		|ZW_0(t,x)|\le B \ep t^{-1}s^\sigma
		\]
		for all $s\in [2, s_0]$ and all $(t,x)\in \hin_{[2, s_0]}$.
	\end{proposition}
	In the above proposition the hyperbolic time $s_0$ is arbitrary which implies the global existence of the solution in the interior region $\dint$.
	The reason why we distinguish between the energies of $W_0$ and of $\Wt$ and do not propagate uniform-in-time energy bounds for the latter is related to some slow decaying semilinear terms that only appear in the equations satisfied by the differentiated function $\q Z^\gamma \Wt$. One of such terms is for instance $Z^\beta u_0\cdot\partial^2_y \Wt$, which appears when $\gamma$ is a multi-index of type $(k,k)$, i.e. when $\q Z^\gamma = Z^\beta$ is a pure product of Klainerman vector fields. The $L^2_{xy}$ norm of such product can only be controlled in terms of the (square root of the) conformal energy $\ec_k(s, W_0)$, but even assuming the following sharp bounds
	\[
	\ec_k(s, W_0)\lesssim \ep^2 s \qquad \text{and}\qquad \sup_{\hin_s}|\partial^2_y\Wt|\lesssim \ep^2 s^{-\frac{3}{2}}
	\]
	we are not able to recover a better $L^2$ bound than the following one, which is at the limit of integrability and prevents us from obtaining uniform in time bounds for the higher order energies of $\Wt$
	\[
	\left\| Z^\beta u_0\cdot\partial^2_y\Wt \right\|_{L^2_{xy}(\hin_s)}\le \ec_{k}(s, W_0)^\frac{1}{2}\|\partial^2_y\Wt\|_{L^\infty(\hin_s)}\lesssim \ep^2 s^{-1}.
	\]
	These problematic terms do not appear in the equation for $\q Z^\gamma W_0$, for which we can instead easily prove uniform in time energy bounds using the null structure.
	
	\medskip
	In both the interior and exterior region the booststrap argument is performed in two classical steps. The first step consists in recovering pointwise bounds for the solution. A first set of estimates is obtained from the a-priori energy bounds using Klainerman-Sobolev inequalities on hyperboloids in the interior region and weighted Sobolev embeddings on constant time slices in the exterior region. The pointwise bounds obtained for $\Wt$ in the interior region with this approach are not optimal due to the slow growth in time assumed in \eqref{booten2}, therefore more suitable bounds need to be recovered by directly analyzing the equation satisfied by $\Wt$. 
	Since Klainerman-Sobolev inequalities only yield pointwise bounds for derivatives of $W_0$, we will also need to recover $L^\infty-L^\infty$ estimates for $W_0$ using the equation it satisfies.
	
	The second step of the bootstrap argument consists in writing a higher order (interior and exterior respectively) energy inequality for the solution and in using the pointwise bounds previously obtained to perturbatively estimate the cubic terms appearing in the right hand side of such inequality. As appears from the example briefly mentioned above, we also need to recover higher order conformal energy bounds for $W_0$ from the a-priori energy bounds. This is especially important in order to propagate the interior energy bounds \eqref{booten1} and \eqref{booten2}.
	
	\smallskip
	The paper is structured as follows. In Section \ref{Sec:Lin} we derive the energy inequalities for the linearized equation, both in the exterior and the interior region. We also derive the conformal energy inequalities. In Section \ref{Sec:Exterior} we prove the global existence of the solution in the exterior region. In Section \ref{PointwiseEst} we recover the pointwise estimates for $\Wt$ and $W_0$. Finally in Section \ref{Sec: VF} we improve the energy estimates in the interior region, concluding the proof of Theorem \ref{main}.

	\section{The Linearized Equation}\label{Sec:Lin}
	
	The purpose of this section is to write the energy inequality and the conformal energy inequality for the linearized equation associated to the equation \eqref{eq_W} in both the interior and exterior regions $\dint$ and $\dext$. We will look at the following linear inhomogeneous equation
	\begin{equation}\label{linearized}
		(-\partial^2_t + \Delta_x) \Wf + (1+u) \partial^2_y \Wf = \Ff, \qquad (t,x,y)\in \m R\times \m R^3\times \m S^1,
	\end{equation}
	where $u$ is assumed to be a sufficiently small function, e.g. $|u|\le 1/10$.
	
	We start by proving a weighted energy inequality on the exterior constant time slices $\Sext_t$ for general linear inhomogeneous equations of the form \eqref{linearized}. In the following proposition the lifespan $T_0$ is arbitrary, $\dext_{T_0}$ denotes the portion of exterior region in the time strip $[2, T_0]$ and $\q C_{[2, T_0]}$ is its null boundary
	\[
	\dext_{T_0}=\left\{(t,x,y)\in \dext: 2\le t\le T_0\right\},
	\] 
	\[
	\q C_{[2, T_0]} = \left\{(t, x): 2\le t\le T_0, \, r=t-1 \right\}\times \m S^1.
	\]
	
	\begin{proposition}\label{prop: en_ineq_ext}
		Let $\Wf$ be a solution to the equation \eqref{linearized}, $l\in L^1([2, T_0])$ and suppose that $u$ is a function satisfying the following pointwise bounds in the exterior region $\dext_{T_0}$
		\begin{align}
			\label{estu1}	 \|u\|_{L^\infty(\dext_{T_0})}& \le 1/10 \\
			\label{estu2}		\left\|(2+r-t)^\frac{1}{2}\partial_y u(t, x, \cdot)\right\|_{L^\infty(\m S^1)}& \lesssim \ep \sqrt{l(t)} \\
			\label{estut}
			\left\|(2+r-t)\partial u(t, x, \cdot)\right\|_{L^\infty(\m S^1)}& \lesssim \ep.
		\end{align}
		For any fixed $\alpha>0$ the following inequality holds true
		\begin{multline}\label{ineq_ex_lin}
			\|\Wf\|_{\Xext_{T_0}}^2 + \iint_{\q C_{[2, T_0]}} (2+r-t)^{\alpha+1}\left( |\mathcal{T}\Wf|^2 +|\partial_y\Wf|^2\right)\, d\sigma dydt \\
			\lesssim \enext(2, \Wf)  
			+ \left\|(2+r-t)^{\frac{\alpha+1}{2}}\Ff\right\|_{L^1_tL^2_{xy}(\dext_{T_0})} \|\Wf\|_{\Xext_{T_0}}.
		\end{multline}
		where $d\sigma$ is the area element of the sphere $\m S^2$.
	\end{proposition}
	We remark that the result of proposition \ref{prop: en_ineq_ext} can be proved for any positive and increasing weight $\omega=\omega(r-t)$ only depending on the distance from the cone $\{t=r\}$ if the hypothesis on the function $u$ are changed appropriately.
	\proof
	From the smallness assumption on $u$ it will be enough to prove the inequality in the statement with $\enext(t, \Wf)$ and $\|\Wf\|_{\Xext_{T_0}}$ replaced by $\exquasi(t, \Wf)$ and $ \|\Wf\|_{\Xext_{\text{quasi}, T_0}}$ respectively.
	A simple computation shows that for any positive weight $\omega = \omega(r-t)$ one has
	\begin{equation}\label{weighted_div}
		\begin{aligned}
			& \omega(r-t) \partial_t \Wf \left[\Box_{x,y} \Wf + u \, \partial^2_y \Wf \right] 
			=-\frac{1}{2}\rd_t \left[\omega(r-t)\left((\partial_t\Wf)^2+|\nabla_x \Wf|^2+(1+u)(\partial_y \Wf)^2\right)\right]\\ 
			& +\text{div}_x\left(\omega(r-t)\partial_t \Wf \,\nabla_x \Wf\right )+\partial_y\left(\omega(r-t)(1+u)\partial_t \Wf \partial_y \Wf\right)\\
			&-\frac{1}{2}\omega'(r-t)\left( |\mathcal{T}\Wf|^2+(1+u)(\partial_y \Wf)^2  \right)-\omega(r-t)\left(\partial_y u \, \partial_t \Wf \partial_y \Wf-\frac{1}{2}\partial_t u (\partial_y \Wf)^2\right) .
		\end{aligned}
	\end{equation}
	We consider the case $\omega(z)=(2+z)^{\alpha+1}$, integrate the above equality over the exterior region $\dext_{T_0}$ and use the Stokes' theorem. We obtain that
	\begin{equation}\label{eneq}
		\begin{aligned}
			&\|\Wf\|^2_{\Xext_{\text{quasi}, T_0}} + \iint_{\q C_{[2, T_0]}} (2+r-t)^{\alpha+1}\left[ |\mathcal{T}\Wf|^2 + (1+u)|\partial_y\Wf|^2\right]\, d\sigma dydt \le \exquasi(2, \Wf) \\
			&\hspace{2cm}+ 2\iint_{\dext_{T_0}} (2+r-t)^{\alpha+1} \left( \Ff\, \partial_t \Wf + \partial_yu\, \partial_t \Wf\partial_y\Wf -\frac{1}{2} \partial_tu \, (\partial_y\Wf)^2\right)\, dxdydt.
		\end{aligned}
	\end{equation}
	The smallness of $u$ assures that the integral in the above left hand side is equivalent to the one in the left hand side of inequality \eqref{ineq_ex_lin}.
	From the assumption \eqref{estu2} and the Cauchy-Schwarz inequality we see that
	\[
	\begin{aligned}
		& \iint_{\dext_{T_0}} (2+r-t)^{\alpha+1}  |\partial_y u \partial_t\Wf \partial_y\Wf | dxdydt \lesssim \left\|(2+r-t)^{\frac{\alpha+1}{2}}  \partial_y u \cdot\partial_y\Wf\right\|_{L^1_tL^2_{xy}(\dext_{T_0})}\|\Wf\|_{\Xext_{T_0}} \\
		& \hspace{1.5cm}\lesssim \left\|(2+r-t)^\frac{1}{2}\partial_y u\right\|_{L^2_t L^\infty_{xy}(\dext_{T_0})}\left\|(2+r-t)^\frac{\alpha}{2}\partial_y \Wf\right\|_{L^2_{txy}(\dext_{T_0})}\|\Wf\|_{\Xext_{T_0}}\\
		& \hspace{1.5cm} \lesssim \ep \|\sqrt{l}\|_{L^2_t}\|\Wf\|^2_{\Xext_{T_0}} \lesssim \ep\|\Wf\|^2_{\Xext_{T_0}} 
	\end{aligned}
	\]
	and from \eqref{estut}
	$$ \iint_{\dext_{T_0}} (2+r-t)^{\alpha+1}  |\partial_t u (\partial_y\Wf)^2 | dxdydt \lesssim \ep \iint_{\dext_{T_0}} (2+r-t)^{\alpha}  (\partial_y\Wf)^2 dxdydt \lesssim  \ep\|\Wf\|^2_{\Xext_{T_0}} .$$
	If $\epsilon\ll 1$ is sufficiently small the above two integrals can be hence absorbed in the left hand side of \eqref{eneq}. Finally, from the Cauchy-Schwarz inequality we also have that
	\[
	\iint_{\dext_{T_0}} (2+r-t)^{\alpha+1} |\Ff| |\partial_t \Wf| dxdydt \le \left\|(2+r-t)^{\frac{\alpha+1}{2}}\Ff\right\|_{L^1_tL^2_{xy}(\dext_{T_0})} \|\Wf\|_{\Xext_{T_0}}.
	\]
	\endproof
	
	If we assume that the function $u$ satisfies the pointwise bounds \eqref{estu1}-\eqref{estut} in the whole exterior region $\dext$ we can also recover an energy inequality on the exterior truncated hyperboloids $\hex_s$. In the following proposition $\q C_{[2, s]}$ denotes the lateral boundary of $\hex_{[2,s]}$
	\[
	\q C_{[2, s]} = \left\{(t,x): 2\le t\le (s^2+1)/2, \,  r=t-1\right\}\times \m S^1
	\]
	and $\|\cdot\|_{\Xext_\infty}$ is the stronger norm over the time interval $[2,\infty)$
	\[
	\|\Wf\|_{\Xext_\infty} := \lim_{T\rightarrow\infty}\|\Wf\|_{\Xext_T}. 
	\]
	\begin{proposition}\label{prop:ext_hyp}
		Assume that $\Wf$ is solution to the linear inhomogeneous equation \eqref{linearized} with $u$ satisfying the decay properties \eqref{estu1}-\eqref{estut} in the whole exterior region $\dext$. 
		Then
		\begin{align*}
			&	\enexh(s, \Wf) + \iint_{\CT}  |\mathcal{T}\Wf|^2 + |\partial_y\Wf|^2 \, d\sigma dy dt \\
			& \hspace{4cm}\lesssim E^{\text{ex},0}(2, \Wf) + \ep  \|\Wf\|^2_{X^{\text{ex},0}_\infty} + \|\Ff\|_{L^1_tL^2_{xy}(\hex_{[2, s]})} \|\Wf\|_{X^{\text{ex},0}_\infty}
		\end{align*}
		for any $s\ge 2$.
	\end{proposition}
	\proof
	From the smallness assumption on $u$ it will be enough to prove the the statement with $\enexh(t, \Wf)$ and $E^{\text{ex},0}(2, \Wf)$ replaced by $\enexhq(t, \Wf)$ and $E^{\text{ex},0}_\text{quasi}(2, \Wf)$ respectively.
	We integrate the relation \eqref{weighted_div} in the case where $\omega\equiv 1$ over the region $\hex_{[2, s]}$ which we foliate by the constant time slices $\Sigma^s_t$ for $t\ge 2$ (see picture \ref{fig:hext}), where 
	\[
	\Sigma^s_t =\left\{x\in \m R^3: r\ge \max(t-1,  \sqrt{t^2-s^2}) \right\}\times \m S^1.
	\] 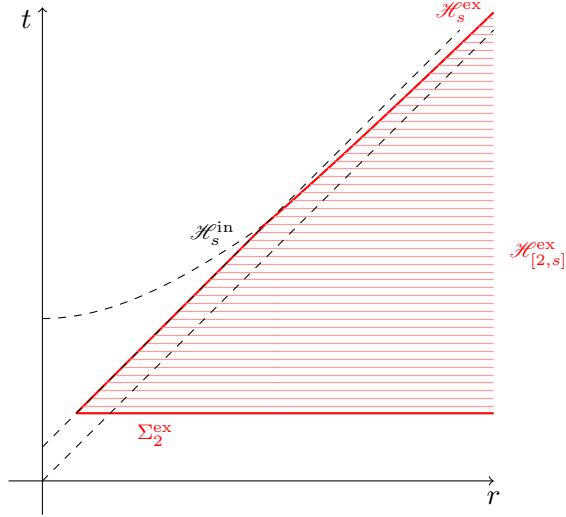
\begin{figure}[h]
\begin{center}
\begin{tikzpicture}[scale=1.5]

\draw[->] (0,-0.3) -- (0,4.2);
\draw[->] (-0.3,0) -- (4,0);
\node[left, below] at (4,0) {\small $r$};
\node[left] at (0,4.1) {\small $t$};

\draw[dashed] [domain=0:2] plot(\x, {((1.44)^2+(0.8)*(\x)^2)^(0.5)});
\node[above, left] at (1.8, 2.2) {\tiny $\hin_s$};

\draw[red,thick, name path=A] [domain=2:4] plot(\x, {((1.13)^2+(\x)^2)^(0.5)});
\node[red, left] at (4, 4.16) {\tiny $\hex_s$};

\draw[red, thick, name path = B]  (0.3,0.6) -- (4,0.6);
\node[red, thick, below] at (1, 0.6) {\tiny $\Sext_2$};

\draw[red, thick] [domain=0.3:2] plot(\x, \x+0.3);

\tikzfillbetween[of=A and B]{pattern = horizontal lines, pattern color = red!40}

\node[red, thick] at (4.4,2) {\tiny $\hex_{[2, s]}$};

\draw[dashed] [domain=0:3.7] plot(\x,\x+0.3);

\draw[dashed] [domain=0:4] plot(\x,\x);


\end{tikzpicture}
\caption{Vertical section of the region $\hex_{[2, s]}$ and its foliation projected onto $\m R^{1+3}$}
\label{fig:hext}
\end{center}
\end{figure}
	\noindent We obtain that
	\[
	\begin{aligned}
		& \enexhq(s, \Wf) + \iint_{\CT}  |\mathcal{T}\Wf|^2 + (1+u)|\partial_y\Wf|^2\, d\sigma dy dt \le E^{\text{ex},0}_{\text{quasi}}(2, \Wf)  \\
		&+ \iint_{\hex_{[2, s]}} \partial_yu\, \partial_t\Wf \partial_y\Wf - \frac{1}{2}\partial_tu \, (\partial_y\Wf)^2 \, dxdydt + \iint_{\hex_{[2, s]}} \Ff \partial_t \Wf\, dxdydt.
	\end{aligned}
	\]
	From the assumptions on the derivatives of $u$, the Cauchy-Schwarz inequality and the definition of the norm $X^{\text{ex},0}_\infty$ we immediately see that 
	\[
	\begin{aligned}
		\iint_{\hex_{[2, s]}} \partial u\, \partial_y\Wf\,  \partial \Wf\, dxdydt &\le \int_2^\infty \|\partial u\|_{L^\infty_{xy}(\Sigma^s_t)}\|\partial_y\Wf\|_{L^2_{xy}(\Sigma^s_t)}\|\partial \Wf\|_{L^2_{xy}(\Sigma^s_t)}\, dt \\
		& \lesssim \|\Wf\|_{X^{\text{ex},0}_\infty}\int_2^\infty \ep \sqrt{l(t)}\|\partial_y\Wf\|_{L^2_{xy}(\Sigma^s_t)}\, dt\lesssim \ep \|\Wf\|^2_{X^{\text{ex},0}_\infty}.
	\end{aligned}
	\]
	Furthermore
	\[
	\iint_{\hex_{[2, s]}} \Ff \partial_t \Wf\, dxdydt \le \|\Ff\|_{L^1_tL^2_{xy}(\hex_{[2,s]})}\|\Wf\|_{X^{\text{ex},0}_\infty}.
	\]
	\endproof

	We now prove the interior energy inequality for \eqref{linearized}. In the following proposition the hyperbolic lifespan $s_0$ is arbitrary and $\CT$ will denote the later boundary of the hyperbolic region $\hin_{[2, s]}$
	\[
	\CT = \left\{ (t,x) : t=r+1\text{ and }  3/2\le t \le (s^2-1)/2\right\}\times \m S^1.
	\]
	
	\begin{proposition}\label{Prop:en_lin}
		Let $\Wf$ be a solution to the equation \eqref{linearized} and suppose that $u$ is a function that satisfies the following bounds in the hyperbolic region $\hin_{[2, s_0]}$
		\begin{align}
			\label{boot_u}\|u\|_{L^\infty(\hin_{[2, s_0]})}&\le 1/10 \\
			\label{boot_u0}|\partial_t u_0(t,x)|&\lesssim \ep t^{-\frac{1}{2}}s^{-1}\\
			\label{boot_ut}\|\partial \ut(t,x, \cdot)\|_{L^\infty(\m S^1)}&\lesssim \ep t^{-\frac{3}{2}+\delta}
		\end{align}
		for some small $\delta>0$, where $u_0=\int_{\m S^1}u(t,x,y) dy$ and $\ut = u-u_0$. Then
		\begin{equation}\label{lin_energy}
			\enint(s, \Wf)  \lesssim \enint(2, \Wf) + \iint_{\CT} |\mathcal{T} \Wf|^2 + |\partial_y\Wf|^2\,  d\sigma dydt  + \int_2^{s}\|\Ff\|_{L^2_{xy}(\hin_\tau)}\enint(\tau, \Wf)^\frac{1}{2}\, d\tau
		\end{equation}
		for all $s\in [2, s_0]$,
		where $d\sigma$ is the surface element on the sphere $\m S^2$.
	\end{proposition}
	\proof
	We integrate the equality \eqref{weighted_div} in the case where $\omega\equiv 1$ over the region $\hin_{[2, s]}$ (see picture \ref{fig:hin}) for any $s\in [2, s_0]$ and use the Stokes' theorem together with system \eqref{linearized}. We foliate $\hin_{[2, s]}$ by hyperboloids $\hin_\tau$ with $\tau\in [2, s_0]$ and obtain the following
	\begin{equation}\label{Alinhac}
		\begin{aligned}
			& \equasi(s, \Wf) \le \equasi(2, \Wf) + \iint_{\CT}  |\mathcal{T}\Wf|^2 + (1+u)|\partial_y\Wf|^2\, d\sigma dy dt \\
			&+ \int_2^{s}\hspace{-5pt} \int_{\hin_\tau}  (\tau/t) \partial_yu\, \partial_t\Wf \partial_y\Wf - \frac{1}{2} (\tau/t) \partial_tu \, (\partial_y\Wf)^2 \, dxdyd\tau +  \int_2^{s}\hspace{-5pt} \int_{\hin_\tau}  (\tau/t) \Ff \partial_t \Wf\, dxdyd\tau.
		\end{aligned}
	\end{equation}
	The integral on the null boundary $\q C_{[2,s]}$ in the above right hand side is bounded by its counterpart in the right hand side of \eqref{lin_energy} thanks to the smallness assumption \eqref{boot_u}.
	From the assumption \eqref{boot_ut} and the fact that $\tau\le t$ we derive that
	\begin{multline*}
		\left| \int_2^{s}\hspace{-5pt} \int_{\hin_\tau} (\tau/t)  \partial_yu\, \partial_t\Wf \partial_y\Wf - \frac{1}{2} (\tau/t) \partial_t\ut \, (\partial_y\Wf)^2   \, dxdyd\tau\right|   \\
		\lesssim \int_{2}^{s}  \left\|\partial \ut\right\|_{L^\infty(\hin_\tau)}\enint(\tau, \Wf) \, d\tau\lesssim \epsilon \int_{2}^{s} \tau^{-\frac{3}{2} + \delta}  \enint(\tau, \Wf) \, d\tau 
	\end{multline*}
	while from \eqref{boot_u0} we have
	\[
	\begin{aligned}
		\left| \int_2^{s}\hspace{-5pt} \int_{\hin_\tau} \frac{1}{2}(\tau/t) \partial_tu_0 \, (\partial_y\Wf)^2   \, dxdyd\tau\right| & \lesssim \epsilon \int_2^{s} \tau^{-\frac{3}{2}} \enint(\tau, \Wf)\, d\tau.
	\end{aligned}
	\]
	Finally, the Cauchy-Schwarz inequality yields
	\[
	\left|\int_2^{s}\hspace{-5pt} \int_{\hin_\tau} (\tau/t) \Ff \partial_t \Wf \, dxdyd\tau\right| \le \int_2^{s_0} \|\Ff\|_{L^2_{xy}(\hin_\tau)}\enint(\tau, \Wf)^\frac{1}{2}\, d\tau.
	\]
	\endproof
	
	As detailed in Section \ref{Sec: VF}, it will be important to distinguish between the two components $W_0$ and $\Wt$ of the solution $W$ to \eqref{eq_W}, in particular to show that the energies associated to the zero mode $W_0$ are uniformly bounded in time. We will make use of the following classical result on the energy on interior truncated hyperboloids for linear inhomogeneous wave equations on the flat space $\mathbb{R}^{1+3}$. 
	\begin{proposition}\label{prop:linear_wave0}
		Let $\Wf_0$ be a solution to the following linear inhomogeneous wave equation
		\begin{equation}\label{linear_wave0}
			(-\partial^2_t + \Delta_x)\Wf_0 = \Ff_0, \qquad (t,x)\in \m R\times \m R^3.
		\end{equation}
		For all $s\in [2, s_0]$ we have the following energy inequality 
		\[
		\enint(s, \Wf_0) \le \enint(2,  \Wf_0) + \int_{\CT} |\mathcal{T}\Wf_0|^2\, d\sigma dt + \int_2^{s}   \left\|\Ff_0\right\|_{L^2_x(\hin_\tau)}\enint(\tau, \Wf_0)^{1/2}\, d\tau.
		\]
	\end{proposition}
	
	We also state below the interior and exterior conformal energy inequalities for linear inhomogeneous wave equations on $ \m R^{1+3}$. We will need to have a control on the higher order conformal energies of the zero-mode $W_0$ of our solution $W$ in order to recover pointwise bounds for $W_0$ and $ZW_0$ later in the paper.
	
	\begin{proposition}\label{Prop:Conf_En}
		Let $\Wf_0$ be a solution to \eqref{linear_wave0}. We have the following inequality
		\[
		\begin{aligned}
			\ecin(s, \Wf_0) & \le  \ecin(2, \Wf_0) + \int_2^s \left\| s\Ff_0\right\|_{L^2(\hin_\tau)}\ecin(\tau, \Wf_0)^\frac{1}{2}d \tau \\
			& + \int_{\mathcal{C}_{[2, s]}}\left|
			(t+r)(\partial_t \Wf_0+\partial_r \Wf_0) + 2\Wf_0\right|^2+(t-r)^2(|\nabla \Wf_0|^2-(\partial_r \Wf_0)^2)\, d\sigma dt ,
		\end{aligned}
		\]
		where $d\sigma$ is the surface element on the sphere $\m S^2$.
	\end{proposition}
	\proof
	Let $K=(t^2+r^2)\partial_t  +2rt\partial_r$ denote the Morawetz multiplier. Then
	\begin{equation}\label{conformal}
		\begin{aligned}
			& \left(K\Wf_0 + 2t\Wf_0\right)\Box_x\Wf_0\\
			&  = -\partial_t\left[\frac{1}{2}(t^2+r^2)\left(|\partial_t \Wf_0|^2 + |\nabla_x\Wf_0|^2\right) + 2rt\partial_t\Wf_0 \partial_r\Wf_0 + 2t\Wf_0 \partial_t \Wf_0 - \Wf_0^2\right] \\
			& + \text{div}_x\left[(t^2+r^2)\partial_t \Wf_0 \nabla_x\Wf_0 + 2rt \partial_r\Wf_0 \nabla_x \Wf_0 + tx\left(|\partial_t\Wf_0|^2 - |\nabla_x\Wf_0|^2) + 2t\Wf_0 \nabla_x\Wf_0\right)\right]
		\end{aligned}
	\end{equation}
	and the result of the statement follows from the integration of the above equality over the interior hyperbolic region $\hin_{[2,s]}$ and from the Stokes' theorem.
	\endproof

	\begin{proposition}\label{prop:conf_ext}
		Let $\Wf_0$ be a solution to the equation \eqref{linear_wave0}. Then
		\[
		\begin{aligned}
			&\exec(T, \Wf_0) + \int_{\q C_{[2, T]}}\left| (t+r)(\partial_t \Wf_0+\partial_r \Wf_0) + 2\Wf_0\right|^2+(t-r)^2(|\nabla \Wf_0|^2-(\partial_r \Wf_0)^2)\, d\sigma dt\\
			& \hspace{4cm}\le \exec(2, \Wf_0) + \int_2^T \left\|(t+r)\Ff_0 \right\|_{L^2_x(\Sext_t)}\exec(t, \Wf_0)^\frac{1}{2} dt,
		\end{aligned}
		\]
		where $d\sigma$ is the surface element on the sphere $\m S^2$.
	\end{proposition}
	\proof
	The result of the proposition follows from the integration of equality \eqref{conformal} over the region $\dext_T$ combined with Stokes' theorem and the following equality
	\[
	K\Wf_0 + 2t\Wf_0 = t \left(\scal \Wf_0 + 2\Wf_0\right) + r\Omega_{0r}\Wf_0.
	\]
	\endproof
	
{
		In the interior region we will recover pointwise bounds on $ZW_0$ from the higher conformal energies of $W_0$ via Klainerman-Sobolev inequalities on hyperboloids (see lemma \ref{Lem:KS}). This will require a control for the conformal energy of the solution on a portion of the hyperboloid $\q H_s$ in the exterior region $\dext$, that in turn will be obtained from a control of the conformal energy on the flat hypersurfaces $\Sigma^s_t$ defined below. We hence state the following modification of the conformal energy inequality.}
	
{
		\begin{proposition}\label{prop: twisted_conformal}
			For any $s\ge 2$, $s<T_1<T_2$ and any $t\in [T_1, T_2]$ let
			\[
			\Sigma^s_t:=\{x\in \m R^3: |x|\ge \sqrt{t^2-s^2}\}
			\]
			and
			\[
			\q E^{\text{c,s}}(t, \Wf_0) :=\int_{\Sigma^s_t}  \left|\scal W_0 + 2W_0\right|^2 + \sum_{i=0}^3 |\Omega_{0i}W_0|^2\, dx.
			\]
			Then
			\[
			\begin{aligned}
				& \q E^{\text{c,s}}(T_2, \Wf_0) + \int_{\q H_s\cap [T_1, T_2]} \frac{1}{t^2}|K\Wf_0+2t\Wf_0|^2 + \frac{s^2}{t^2}\sum_{j=0}^3|\Omega_{0j}\Wf_0|^2\, dx \\
				&\hspace{1cm} \le  \q E^{\text{c,s}}(T_1, \Wf_0) + \int_{T_1}^{T_2}\left\|(t+r)\Ff_0\right\|_{L^2_x(\Sigma^s_t)}\, \q E^{\text{c,s}}(t, \Wf_0)^\frac{1}{2} dt
			\end{aligned}
			\]
		\end{proposition}
		\proof
		The result of the statement follows by integrating \eqref{conformal} over the region between $\Sigma^s_{T_2}$, $\Sigma_{T_1}^s$ and $\q H_s\cap [T_1, T_2]$, which can be foliated by the hypersurfaces $\Sigma^s_t$ for $t\in [T_1, T_2]$.
		\endproof}

	\section{Global existence in the exterior region}\label{Sec:Exterior}
	
	The main goal of this section is to prove the global existence of solutions $(u,v)$ to system \eqref{syteme}, or equivalently of solutions $W$ to \eqref{eq_W}, in the exterior region $\dext$ under the a-priori energy assumption \eqref{boot_energy}.
	This result will follow immediately from the proof of proposition \ref{prop:boot_ext}, which is organized in two steps. First we recover sharp pointwise bounds for the solution $W$ from the energy bounds \eqref{boot_energy} by means of weighted Sobolev inequalities.
	Then we compute the equation satisfied by the differentiated variable $\mathcal{Z}^\gamma W$, compare it to the inhomogeneous linear equation \eqref{linearized} and estimate perturbatively the source terms to finally propagate \eqref{boot_energy}.

	As a result of proving global energy bounds in the exterior region we also obtain bounds for the higher order conformal energy $\exec_{5}(t, W)$ for all $t\ge 2$ and a uniform-in-time control of the higher order energy on exterior hyperboloids $\enexh_5(s, W)$ for all $s\ge 2$.
	
	This section is organized as follows: in subsections \ref{Weighted_Sobolev} and \ref{Hardy} we prove some weighted Sobolev and Hardy inequalities; in subsection \ref{Pointwise_Bounds} we recover pointwise bounds for the solution from the a-priori energy estimate \eqref{boot_energy}; finally, subsection \ref{Propagation of the exterior energy bounds} is dedicated to the proof of  proposition \ref{prop: en_ineq_ext}.

	\subsection{Weighted Sobolev inequalities}\label{Weighted_Sobolev}

	\begin{lemma}\label{lemma_wsi}
		Let $\beta\in\mathbb{R}$. For any sufficiently smooth function $w$ we have
		\begin{multline}\label{sobolev}
			\sup_{\Sext_t}\, (2+r-t)^\beta r^2|w(t,x,y)|^2  \\
			\lesssim\iint_{\Sext_t} (2+r-t)^{\beta+1}(\partial_r \mathcal{Z}^{\le 2} w)^2
			+(2+r-t)^{\beta-1}(\mathcal{Z}^{\le 2} w)^2 \,	dxdy.
		\end{multline}
	\end{lemma}
	\begin{proof} Let $(r,\sigma)$ be the spherical coordinates in $\m R^3$, $r=|x|$ and $\sigma = x/|x|\in \m S^2$.
		We begin by observing that  the Sobolev embedding $H^2(\m S^{2}\times \m S^1) \subset L^\infty(\m S^{2}\times \m S^1)$ implies 
		\[\sup_{ \m S^{2}\times \m S^1}|w(t,r,\sigma,y)|^2 \leq \sum_{0\le l+k\leq 2} \int|\nabla_\sigma^l \partial_y^k w(t,r,\sigma,y)|^2d\sigma dy.
		\]
		We then remark that for any function $v$ and $(t,x,y)\in \Sext_t$
		\[
		\begin{aligned}
			& \partial_r \Big[ (2+r-t)^\beta r^2\, v(t,x,y)^2 \Big] \\
			&\hspace{10pt} = 2 (2+r-t)^\beta r^2\,  v\partial_r v     +  \beta(2+r-t)^{\beta-1}r^2 v^2 + 2(2+r-t)^\beta r v^2\\
			& \hspace{10pt}\ge 2 (2+r-t)^\beta r^2 v\partial_r v     +  \beta(2+r-t)^{\beta-1}r^2 v^2 ,
		\end{aligned}
		\]
		so if $v$ is compactly supported in $x$ we can write
		\begin{equation}\label{ineq_ks}
			\begin{aligned}
				&(2+r-t)^\beta r^2\, v(t,x,y)^2 = -\int_r^\infty \partial_\rho \big[(2+\rho-t)^\beta \rho^2\, v(t,x,y)^2\big]d\rho\\
				& \lesssim_\beta \int_r^{\infty}   (2+\rho-t)^\beta |v\partial_\rho v| \rho^2d\rho + \int_r^\infty (2+\rho-t)^{\beta-1}v^2 \rho^2d\rho \\
				&	\lesssim_\beta \int_{r}^\infty  (2+\rho-t)^{\beta+1} (\partial_\rho v)^2\, \rho^2 d\rho + \int_r^\infty (2+\rho-t)^{\beta-1}v^2 \rho^2d\rho .
			\end{aligned}
		\end{equation}
		By	 replacing $v$ with $\nabla_\sigma^l \partial_y^k w(t,r,\sigma,y)$ for $k+l\le 2$ we obtain \eqref{sobolev} in the case where $w$ is compactly supported. 
		In the general case where $w$ is not compactly supported we consider a cut-off function $\chi\in C^\infty_0(\mathbb{R})$ and apply the inequality \eqref{sobolev} to $\chi(\ep r )w$ for any $\ep>0$
		\[
		\begin{aligned}
			&\sup_{\Sext_t}\, (2+r-t)^\beta r^2|\chi(\ep r)w|^2 \\
			&\hspace{20pt} \lesssim \sum_{k+l\leq 2} \iint_{\Sext_t} (2+r-t)^{\beta+1} (\partial_r \nabla_\sigma^k \partial_y^l w)^2
			+(2+r-t)^{\beta-1}(\nabla_\sigma^k \partial_y^l w)^2 \,	dxdy \\
			& \hspace{20pt} +  \sum_{k+l\leq 2} \iint_{\Sext_t} (2+r-t)^{\beta+1} \ep^2|\chi'(\ep r)|^2(\nabla_\sigma^k \partial_y^l w)^2 \, dxdy.
		\end{aligned}
		\]
		On the intersection of $\Sext_t$ with the support of $\chi'(\ep r)$ we have that $(2+r-t)^2 \ep^2 \lesssim 1$ so
		\[
		\iint_{\Sext_t} (2+r-t)^{\beta+1} \ep^2|\chi'(\ep r)|^2(\nabla_\sigma^k \partial_y^l w)^2 \, dxdy \lesssim \iint_{\Sext_t} (2+r-t)^{\beta-1}(\nabla_\sigma^k \partial_y^l w)^2 \,	dxdy.
		\]
		By	 letting $\epsilon\rightarrow0$ we derive \eqref{sobolev} also in the case of non compactly supported $w$.	
	\end{proof}
	
	Slight modifications of the above proof yield the following three results.
	
	\begin{lemma}\label{lemma_wsi2}
		Let $\beta \in \m R$. For a sufficiently regular function $w$ we have
		\begin{equation}\label{sobolev_bis}
			\sup_{\Sext_t}\, (2+r-t)^\beta r^2|w(t,x,y)|^2 \lesssim  \iint_{\Sext_t} (2+r-t)^{\beta}\Big((\partial_r \mathcal{Z}^{\le2}  w)^2
			+(\mathcal{Z}^{\le2} w)^2\Big)dxdy.
		\end{equation}
	\end{lemma}
	\begin{proof}
		It follows by estimating $v$ and $\partial_\rho v$ in \eqref{ineq_ks} with the same weight and using the fact that $2+r-t\ge 1$ on $\Sext_t$.
	\end{proof}
	
	\begin{lemma}\label{lemma_L2L2}
		Let $\beta\in\mathbb{R}$. For any sufficiently regular function $w$ we have
		\begin{equation}\label{wks2}
			\begin{split}
				& \sup_{\Sext_t}\, (2+r-t)^\beta r^2\|w(t,r,\cdot)\|^2_{L^2(\m S^2 \times \m S^1)} \lesssim\iint_{\Sext_t} (2+r-t)^{\beta+1}(\partial_r w)^2
				+(2+r-t)^{\beta-1} w^2 \,	dxdy.
			\end{split}
		\end{equation}
	\end{lemma}
	\proof
	The inequality \eqref{wks2} follows by replacing $v$ with the $L^2(\m S^2\times\m S^1)$ norm of $w$ in both left and right hand sides of inequality \eqref{ineq_ks}. 
	\endproof
	
	\begin{lemma}
		Let $\beta\in\mathbb{R}$. For any sufficiently regular function $w$ we have
		\begin{equation}\label{wks4_bis}
			\sup_{\Sext_t}\,  (2+r-t)^\beta r^2 \left\| w(t,r,\cdot)\right\|^2_{L^4(\m S^2 \times \m S^1)} \lesssim\iint_{\Sext_t} (2+r-t)^{\beta}(\partial^{\le 1}_r \mathcal{Z}^{\le 1} w)^2\,	dxdy.
		\end{equation}
	\end{lemma}
	\proof
	The inequality \eqref{wks4_bis} follows by estimating $v$ and $\partial_\rho v$ in \eqref{ineq_ks} with the same weight, then applying the inequality with $v$ replaced by the $L^4(\m S^2\times\m S^1)$ norm of $w$ and finally using the Sobolev injection $H^1(\m S^2\times \m S^1)\subset L^4(\m S^2\times \m S^1)$. 
	\endproof

	\subsection{Weighted Hardy inequalities}\label{Hardy}

	\begin{lemma}\label{lemma_har}
		Let $\beta>-1$. For any compactly supported function $w$ we have
		\[
		\iint_{\Sext_t} (2+r-t )^\beta w^2 dxdy \lesssim \iint_{\Sext_t}(2+r-t)^{\beta+2} (\partial_r w)^2 dxdy.
		\]
	\end{lemma}
	\begin{proof}
		A simple computation shows that for any $\beta\in\mathbb{R}$ and $(t,x,y)\in \Sext_t$
		$$\partial_r \left[ r^2 (2+r-t)^{\beta+1} \right] =2r(2+r-t)^{\beta+1}+(\beta+1)r^2(2+r-t)^\beta
		\geq (\beta+1)r^2(2+r-t)^\beta.$$
		Consequently, if $\beta>-1$ and $w$ is a compactly supported function 
		\begin{align*}
			&\int_{r\geq t-1} (2+r-t)^\beta w^2 dx = \int_{t-1}^\infty\int_{\m S^2} (2+r-t)^\beta r^2   w^2 drd\sigma \\
			& \hspace{10pt} \leq \frac{1}{\beta+1}\int_{t-1}^\infty\int_{\m S^2}   \partial_r (r^2 (2+r-t)^{\beta+1} ) w^2 drd\sigma \\
			& \hspace{10pt}	= -\frac{2}{\beta+1}\int_{r \geq t-1}   (2+r-t)^{\beta+1}w\partial_r w \ r^2  drd\sigma -\frac{1}{\beta+1} \int_{\m S^{2}\times \m S^1} w^2 r^2 d\sigma\Big|_{r=t-1} \\
			& \hspace{10pt}	\lesssim  \left(\int_{r \geq t-1}   (2+r-t)^{\beta}w^2 dx\right)^\frac{1}{2} \left(\int_{r \geq t-1}   (2+r-t)^{\beta+2}(\partial_r w)^2 dx\right)^\frac{1}{2}.
		\end{align*}
		The result of the lemma follows then after further integration on $\m S^1$.
	\end{proof}
	
	\begin{corollary}\label{corhar}
		Let $\beta>-1$ and $Z$ be any Klainerman vector field. For a sufficiently regular function $w$ we have
		\[
		\iint_{\Sext_t} (2+r-t )^\beta (Zw)^2 dxdy \lesssim \iint_{\Sext_t}(2+r-t)^{\beta+2}(\partial Z^{\le 1} w)^2 dxdy.
		\]
	\end{corollary}
	\begin{proof}
		For any $\ep >0$ and any fixed cut-off function $\chi$ we apply lemma \ref{lemma_har} to $\chi(\ep r)Zw$ and use the fact that $|Zw|\lesssim r|\partial u|$ on $\Sext_t$. Since $r\ep$ is bounded on the support of $\chi(\ep r)$ we have 
		\begin{align*}
			&\iint_{\Sext_t} (2+r-t )^\beta (\chi(\ep r)Zw)^2 dxdy \\
			& \lesssim \iint_{\Sext_t}(2+r-t)^{\beta+2} |\chi(\ep r)|^2(\partial_r Zw)^2\, dxdy
			+\iint_{\Sext_t}(2+r-t)^{\beta+2} \ep^2 |\chi'(\ep r)|^2(Zw)^2\, dxdy\\
			& \lesssim  \iint_{\Sext_t}(2+r-t)^{\beta+2}(\partial_r Zw)^2 dxdy
			+\iint_{\Sext_t}(2+r-t)^{\beta+2} (\partial w)^2 dxdy
		\end{align*}
		and obtain the result after letting $\ep \to 0$.
	\end{proof}

	\subsection{Pointwise Bounds} \label{Pointwise_Bounds}
	Using the lemmas introduced in the previous sections we can recover sharp pointwise bounds for the solution $W$ from the a-priori energy bounds \eqref{boot_energy}.
	
	\begin{proposition}\label{prop:ext_pointbound}
		Assume that the solution $W=(u,v)^T$ to \eqref{eq_W} satisfies the a-priori energy bounds \eqref{boot_energy} for some fixed time $T_0$ and some fixed $\alpha>0$. There exists $l \in L^1([2, T_0])$ such that we have the following pointwise estimates in $\dext_{T_0}$
		\begin{align}
			\label{wext}& \sup_{\m S^1} |W|\lesssim \ep \\
			\label{dwext}&
			\sup_{\m S^1}|\partial\mathcal{Z}^{\le 2} W|\lesssim C_0\ep r^{-1}(2+r-t)^{-\frac{\alpha+1}{2}} \\
			\label{dbarwext}&\sup_{\m S^1} |\mathcal{T}\mathcal{Z}^{\le 2} W| + |\partial_y \mathcal{Z}^{\le 2} W| + |\mathcal{Z}^{\le 2}\Wt| \lesssim C_0 \ep r^{-1} \sqrt{l(t)} (2+r-t)^{-\frac{\alpha}{2}} \\
			\label{zwext}&\sup_{\m S^1} |Z\mathcal{Z}^{\le 2} W|\lesssim C_0 \ep r^{-1}(2+r-t)^{-\frac{\alpha}{2}}.
		\end{align}
	\end{proposition}
	\begin{proof}
		The bounds \eqref{dwext} and \eqref{dbarwext} follow immediately from lemma \ref{lemma_wsi2} with $\beta=\alpha+1$ and $\beta=\alpha$ respectively and from \eqref{l(t)}.
		The pointwise bound \eqref{zwext} follows instead applying lemma \ref{lemma_wsi} with $\beta = \alpha$  and corollary \ref{corhar} to write that
		$$\iint_{\Sext_t}(2+r-t)^{\alpha-1}(Z \q Z^{\leq 4} W)^2 dxdy \lesssim \iint_{\Sext_t} (2+r-t)^{\alpha+1}(\partial \q Z^{\leq 5} W)^2 dxdy.$$ 
		Finally, the bound \eqref{wext} on $W$ is obtained from the integration of \eqref{dwext} along the direction $\partial_q = \partial_r-\partial_t$ until the initial time slice $t_0=2$ and from the smallness assumption on the initial data.
	\end{proof}	
	
	A trivial consequence of the decomposition \eqref{dec} and the bounds \eqref{dbarwext}-\eqref{zwext} that will be useful later in section \ref{PointwiseEst} is the following estimate for the zero-mode $W_0$ of the solution
	\begin{equation}\label{Z0_ext}
		\sup_{\m S^1}\left| Z \q Z^{\le 1}W_0\right|\lesssim C_0\ep r^{-1}(2+r-t)^{-\frac{\alpha}{2}}.
	\end{equation}

	\subsection{Propagation of the exterior energy bounds} \label{Propagation of the exterior energy bounds}
	In order to prove proposition \ref{prop:boot_ext} we start by considering any multi-index $\gamma$ with $|\gamma| =n \le 5$ and compare the system satisfied by the differentiated function $W^\gamma = (u^\gamma,v^\gamma)^T$ - which is a short hand notation for $\mathcal{Z}^\gamma W = (\mathcal{Z}^\gamma u,\mathcal{Z}^\gamma v)$ - to the inhomogeneous linear equation \eqref{linearized}. Here the variable that plays the role of the linear variable $\Wf$ is $W^\gamma$. 
	We remind the reader that all vector fields $\mathcal{Z}$ in the family \eqref{family_Z} are related to the geometry of the problem and are the generators of the Lorentz transformations of the Minkowski space $\mathbb{R}^{1+3}$. In particular, they preserve the structure of system \eqref{syteme} and equivalently of the equation \eqref{eq_W}.
	
	The equation satisfied by $W^\gamma$ is obtained by commuting $\q Z^\gamma$ to the equation \eqref{eq_W} 
	\begin{equation}\label{systeme_vf}
		(-\partial^2_t+\Delta_x)W^\gamma + (1+u)\partial^2_y W^\gamma = \Ff^\gamma
	\end{equation}
	where the inhomogeneous term $\Ff^\gamma$ is given by
	\begin{equation}\label{nonlin_vf}
		\Ff^\gamma = - \delta_\gamma\sum_{\substack{|\gamma_1|+|\gamma_2|  =|\gamma| \\  |\gamma_2|<|\gamma|}} u^{\gamma_1} \partial^2_y W^{\gamma_2} + \sum_{|\gamma_1|+|\gamma_2|  \le |\gamma| } \textbf{N}(W^{\gamma_1}, W^{\gamma_2}) 
	\end{equation}
	where $\delta_\gamma=0$ if $|\gamma|=0$, 1 otherwise.
	
	\begin{lemma}
		The nonlinear term $\mathbf{N}(\cdot, \cdot)$ in the right hand side of \eqref{nonlin_vf} is a two-vector of new linear combinations of the quadratic null forms introduced in \eqref{null_forms} that arise from the commutation of the Klainerman vector fields with $\mathbf{N}_1$ and $\mathbf{N}_2$. 
	\end{lemma}
	\proof
	The proof follows by computing the commutator terms and iterating such formulas. 
	For the vector fields $\Omega_{ij}$, $1\le i<j\le 3$, we have 
	\[
	[\Omega_{ij}, \partial_0] = 0, \qquad [\Omega_{ij}, \partial_k] = 
	\begin{cases}
		\partial_j, \qquad & \text{if } k=i\\
		- \partial_i, \qquad & \text{if } k=j \\
		0, \qquad & \text{otherwise} 
	\end{cases}
	\]
	which implies that
	\[
	\Omega_{ij}Q_0(\phi, \psi) = Q_0(\Omega_{ij}\phi, \psi) + Q_0(\phi, \Omega_{ij}\psi)
	\]
	\[
	\Omega_{ij}Q_{0k}(\phi, \psi)  = Q_{0k}(\Omega_{ij}\phi, \psi)  + Q_{0k}(\phi, \Omega_{ij}\psi)  +
	\begin{cases}
		Q_{0j}(\phi, \psi), \quad & \text{if } k=i \\
		Q_{0i}(\phi, \psi), \quad & \text{if } k=j
	\end{cases}
	\]
	and for $1\le h <k\le 3$
	\[
	\Omega_{ij} Q_{hk}(\phi, \psi) =  Q_{hk}(\Omega_{ij}\phi, \psi) +  Q_{hk}(\phi, \Omega_{ij}\psi) + 
	\begin{cases}
		0, \quad &\text{if } h=i, j=k\\
		Q_{jk}(\phi, \psi), \quad & \text{if } h=i, \, j<k \\
		-Q_{kj}(\phi, \psi), \quad &  \text{if } h=i, \, j>k \\
		Q_{hj}(\phi, \psi), \quad & \text{if } h< j , \, k=i \\
		-Q_{ik}(\phi, \psi), \quad & \text{if } h=j,\,  k>i \\
		Q_{ih}(\phi, \psi), \quad & \text{if } k=j, \, h>i \\
		-Q_{hi}(\phi, \psi), \quad & \text{if } k=j, h<i
	\end{cases}
	\]
	For the vector fields $\Omega_{0i}$, $i=\overline{1,3}$, we have
	\[
	[\Omega_{0i}, \partial_0] = -\partial_i, \qquad [\Omega_{0i}, \partial_j] = \begin{cases}
		\partial_0, \quad & \text{if } k=i, \\
		0, \quad & \text{otherwise}
	\end{cases}
	\]
	which implies
	\[
	\Omega_{0i}Q_0(\phi, \psi) = Q_0(\Omega_{0i}\phi, \psi) + Q_0(\phi, \Omega_{0i}\psi)
	\]
	\[
	\Omega_{0i}Q_{0j}(\phi,\psi)= Q_{0j}(\Omega_{0i}\phi,\psi) + Q_{0j}(\phi,\Omega_{0i}\psi) + 
	\begin{cases}
		-Q_{ij}(\phi,\psi), \quad & \text{if } i<j \\
		Q_{ji}(\phi, \psi), \quad & \text{if } i>j \\
		0, \quad & \text{otherwise}
	\end{cases}
	\]
	and for $1\le j<k\le 3$
	\[
	\Omega_{0i}Q_{jk}(\phi, \psi) = Q_{jk}(\Omega_{0i}\phi, \psi) + Q_{jk}(\phi, \Omega_{0i}\psi) + 
	\begin{cases}
		Q_{0k}(\phi, \psi), \quad & \text{if } j=i \\
		-Q_{0j}(\phi, \psi), \quad &  \text{if } k=i \\
		0, \quad & \text{otherwise}
	\end{cases}
	\]
	\endproof
	We have seen in proposition \ref{prop:est_enex_hyp} that under the a-priori energy assumption \eqref{boot_energy} the solution $W$ satisfies the pointwise bounds \eqref{wext} and \eqref{dwext}. The hypothesis of proposition \ref{prop: en_ineq_ext} are then fulfilled and we have that
	\begin{multline}\label{ineq_enext}
		\left\|W^\gamma \right\|_{\Xext_{T_0}}^2 + \iint_{\mathcal{C}_{[2, T_0]}} |\mathcal{T}W^\gamma|^2 + |\partial_yW^\gamma|^2 \, d\sigma dy dt   \\
		\lesssim \enext(2, W^\gamma)  
		+ \left\|(2+r-t)^{\frac{\alpha+1}{2}}\Ff^{\gamma}\right\|_{L^1_tL^2_{xy}(\dext_{T_0})} \|W^{\gamma}\|_{\Xext_{T_0}}.
	\end{multline}

	\proof[Proof of proposition \ref{prop:boot_ext}]
	It is enough for our purpose to estimate the weighted norm of the source term $\Ff^\gamma$ and prove that for every $|\gamma|\le 5$ 
	\begin{equation}\label{est_Fgamma}
		\left\| (2+r-t)^\frac{\alpha+1}{2}\Ff^\gamma\right\|_{L^2_{xy}(\Sext_t)}\lesssim C_0^2\ep^2 (t-1)^{-\frac{\alpha+1}{2}}\sqrt{l(t)},
	\end{equation}
	where $l\in L^1([2, T_0])$.
	In fact, if we plug \eqref{est_Fgamma} and a-priori energy bound \eqref{boot_energy} into \eqref{ineq_enext} we obtain that there exists some universal positive constant $C$ so that
	\[
	\left\|W \right\|_{X^{5, \alpha}_{T_0}}^2 + \sum_{|\gamma|\le 5} \iint_{\mathcal{C}_{[2, T_0]}} |\mathcal{T}W^\gamma|^2 + |\partial_yW^\gamma|^2 \, d\sigma dy dt   
	\le C\enext_5(2, W)  
	+ 2CC_0^3\ep^3.
	\]
	For any fixed constant $K>1$ (e.g. $K=2$) we can then choose $C_0>0$ sufficiently large such that
	\[
	\enext_{5}(2, W) \le \frac{C^2_0\ep^2}{CK}
	\]
	and  $\ep_0>0$ sufficiently small so that $2CC^2_0\ep<1/K$ for $\ep\leq \ep_0$ to finally obtain
	\begin{equation}\label{en_ext1}
		\left\|W \right\|_{X^{5, \alpha}_{T_0}}^2 + \sum_{|\gamma|\le 5}\iint_{\mathcal{C}_{[2, T_0]}}\left( |\mathcal{T}\mathcal{Z}^{\gamma}W|^2 + |\partial_y \mathcal{Z}^{\gamma}W|^2\right)\, d\sigma dy dt \le \frac{2C_0\ep^2}{K}.
	\end{equation}
	
	We estimate the different contributions to $\Ff^\gamma$ separately. In all the estimates that follow we will use the a-priori energy bound \eqref{boot_energy} and the fact that $r>t-1$ in the exterior region.
	
	\smallskip
	\textit{1. The null terms:} we use here the null form representation via the formula \eqref{rewrite1}
	\[
	\mathbf{N}(W^{\gamma_1}, W^{\gamma_2}) \sim \mathcal{T}W^{\gamma_1} \cdot\partial W^{\gamma_2} + \partial W^{\gamma_1}\cdot\mathcal{T}W^{\gamma_2}.
	\]
	The products in the above right hand side are equivalent given the range of $\gamma_1$ and $\gamma_2$ and we only focus on the analysis of the first one. We distinguish between the different values $\gamma_1$ and $\gamma_2$ can take and remind the reader that $|\gamma_1|+|\gamma_2|\le |\gamma|=n\le 5$.
	
	\smallskip
	\textit{a. The case $|\gamma_1| = 0$}: here $\gamma_2=\gamma$ and we immediately obtain from \eqref{dbarwext} that
	\[
	\left\| (2+r-t)^\frac{\alpha+1}{2} \mathcal{T}W\cdot \partial W^{\gamma}\right\|_{L^2_{xy}}\le \left\| \mathcal{T}W \right\|_{L^\infty_{xy}}\|W^\gamma\|_{\Xext_{T_0}}\lesssim C^2_0\ep^2 (t-1)^{-1}\sqrt{l(t)}.
	\]
	
	\smallskip
	\textit{b. The case $|\gamma_2|=0$}: here $\gamma_1=\gamma$ and we obtain from \eqref{dwext} that
	\[
	\begin{aligned}
		\left\| (2+r-t)^\frac{\alpha+1}{2} \mathcal{T} W^{\gamma}\cdot \partial W\right\|_{L^2_{xy}} &
		\le \left\| (2+r-t)^\frac{1}{2}\partial W \right\|_{L^\infty_{xy}} \left\|  (2+r-t)^\frac{\alpha}{2} \mathcal{T}W^{\gamma} \right\|_{L^2_{xy}} \\
		& \lesssim C_0^2\ep^2 (t-1)^{-1}\sqrt{l(t)}.
	\end{aligned}
	\]
	
	\smallskip
	\textit{c. The case $|\gamma_1|, |\gamma_2|>0$}: we use spherical polar coordinates and Cauchy-Schwarz inequality to bound the weighted $L^2_{xy}(\Sext_t)$ norm of $\mathcal{T}W^{\gamma_1}\cdot \partial W^{\gamma_2}$ as follows
	\[
	\begin{aligned}
		\left\| (2+r-t)^{\frac{\alpha+1}{2}}\mathcal{T}W^{\gamma_1} \cdot\partial W^{\gamma_2}\right\|_{L^2_{xy}}^2 & =  \int_{t-1}^{\infty}\iint_{\m S^2\times \m S^1} (2+r-t)^{\alpha+1}|\mathcal{T}W^{\gamma_1}|^2 \cdot |\partial W^{\gamma_2}|^2\, r^2drd\sigma dy \\
		&\hspace{-10pt}  \lesssim \int_{t-1}^\infty (2+r-t)^{\alpha+1} \|\mathcal{T}W^{\gamma_1}\|^2_{L^{4}(\m S^2\times \m S^1)}  \|\partial W^{\gamma_2}\|^2_{L^{4}(\m S^2\times \m S^1)}\, r^2dr.
	\end{aligned}
	\]
	In the case where $|\gamma_1|\le n-2$ we apply the inequality \eqref{wks4_bis} to $\mathcal{T}W^{\gamma_1}$ with $\beta=\alpha$ and the Sobolev's injection $H^1(\m S^2\times \m S^1)\subset L^4(\m S^2\times \m S^1)$ to $\partial W^{\gamma_2}$. We derive that
	\[
	\begin{aligned}
		& \int_{t-1}^\infty (2+r-t)^{\alpha+1} \|\mathcal{T}W^{\gamma_1}\|^2_{L^{4}(\m S^2\times \m S^1)}  \|\partial W^{\gamma_2}\|^2_{L^{4}(\m S^2\times \m S^1)}\, r^2dr\\
		& \lesssim \left\|(2+r-t)^\frac{\alpha}{2}\mathcal{T}\mathcal{Z}^{\le n}W \right\|^2_{L^2_{xy}(\Sext_t)} \int_{t-1}^\infty (2+r-t) r^{-2}  \left\|\partial \mathcal{Z}^{\le n}W\right\|^2_{L^2(\m S^2\times \m S^1)} r^2 dr\\
		& \lesssim (t-1)^{-2}\left\|(2+r-t)^\frac{\alpha}{2}\mathcal{T}\mathcal{Z}^{\le n}W \right\|^2_{L^2_{xy}(\Sext_t)}\left\|(2+r-t)^\frac{\alpha+1}{2}\partial \mathcal{Z}^{\le n}W \right\|^2_{L^2_{xy}(\Sext_t)} .
	\end{aligned}
	\]
	In the remaining case where $|\gamma_1|=n-1$ and $|\gamma_2|=1$ we apply the inequality \eqref{wks4_bis} to $\partial W^{\gamma_2}$ with $\beta=\alpha+1$ and the injection $H^1(\m S^2\times \m S^1)\subset L^4(\m S^2\times \m S^1)$ to $\mathcal{T} W^{\gamma_1}$. We get
	\[
	\begin{aligned}
		& \int_{t-1}^\infty (2+r-t)^{\alpha+1} \|\mathcal{T} W^{\gamma_1}\|^2_{L^{4}(\m S^2\times \m S^1)}  \|\partial W^{\gamma_2}\|^2_{L^{4}(\m S^2\times \m S^1)}\, r^2dr\\
		& \hspace{3cm} \lesssim  \left\| (2+r-t)^{\frac{\alpha+1}{2}}\partial \mathcal{Z}^{\le n}W\right\|^2_{L^2_{xy}(\Sext_t)} \iint_{\Sext_t} r^{-2}|\mathcal{T}\mathcal{Z}^{\le n}W|^2\, dxdy  \\
		&  \hspace{3cm} \lesssim (t-1)^{-2} \left\|(2+r-t)^\frac{\alpha}{2}\mathcal{T}\mathcal{Z}^{\le n}W \right\|^2_{L^2_{xy}(\Sext_t)} \left\| W\right\|^2_{X^{n, \alpha}_{T_0}}.
	\end{aligned}
	\]
	In both scenarios we obtain that
	\[
	\left\| (2+r-t)^{\frac{\alpha+1}{2}}\mathcal{T}W^{\gamma_1} \partial W^{\gamma_2}\right\|_{L^2_{xy}}\lesssim C_0^2\ep^2(t-1)^{-1}\sqrt{l(t)}.
	\]
	
	\smallskip
	
	\textit{2. The $u^{\gamma_1}\cdot \partial^2_y  W^{\gamma_2}$ terms}: since $|\gamma_1|\ge 1$ we can write $u^{\gamma_1} = \mathcal{Z}u^{\tilde{\gamma}_1}$ for some $|\tilde{\gamma}_1| = |\gamma_1|-1$.  We can also write $\partial^2_y W^{\gamma_2} = \partial_y W^{\tilde{\gamma}_2}$ for some other $\tilde{\gamma}_2$ such that $|\tilde{\gamma}_2| = |\gamma_2|+1$ and observe that then $|\tilde{\gamma}_1|+|\tilde{\gamma}_2| = n$. Depending on $\gamma_1 = (\alpha_1, \beta_1)$ we can distinguish two cases:
	
	\smallskip
	\textit{a. The case $|\alpha_1| >0$}: here we choose $\tilde{\gamma}_1$ so that $\mathcal{Z}u^{\tilde{\gamma}_1} = \partial u^{\tilde{\gamma}_1}$. The products $\partial u^{\tilde{\gamma}_1}\cdot\partial_y W^{\tilde{\gamma}_2}$ have the same behavior of the null terms treated in case 1.
	
	\smallskip
	\textit{b. The case $|\alpha_1|=0$}:
	here $\mathcal{Z}^{\gamma_1} = Z^{\beta_1}$ is a pure product of Klainerman vector fields and $\mathcal{Z}u^{\tilde{\gamma}_1} = Zu^{\tilde{\gamma}_1}$. We choose the exponents $(p_1, p_2)$ as follows
	\[
	(p_1, p_2) = 
	\begin{cases}
		(2, \infty)\qquad & \text{if } |\tilde{\gamma}_1| = n-1 \\
		(\infty, 2) \qquad & \text{if } |\tilde{\gamma}_2| = n \\
		(4,4) \qquad & \text{otherwise}
	\end{cases}
	\]
	and will place the two factors in $L^{p_1}(\m S^2\times \m S^1)$ and $L^{p_2}(\m S^2\times \m S^1)$ respectively. We use the Sobolev's injections $H^2(\m S^2\times \m S^1)\subset L^\infty(\m S^2\times \m S^1)$ and $H^1(\m S^2\times \m S^1)\subset L^4(\m S^2\times \m S^1)$ to derive 
	\[
	\begin{aligned}
		\left\| (2+r-t)^{\frac{\alpha+1}{2}} Zu^{\tilde{\gamma}_1} \partial_y W^{\tilde{\gamma}_2}\right\|^2_{L^2_{xy}} &  \lesssim \int_{t-1}^\infty (2+r-t)^{\alpha+1} \left\|Zu^{\tilde{\gamma}_1} \right\|^2_{L^{p_1}(\m S^2\times \m S^1)}\left\|\partial_y W^{\tilde{\gamma}_2}\right\|^2_{L^{p_2}(\m S^2\times \m S^1)}\, r^2dr \\
		&\hspace{-1.4cm}\lesssim\int_{t-1}^\infty (2+r-t)^{\alpha+1} \left\| Z\mathcal{Z}^{\le 4}u\right\|^2_{L^2(\m S^2\times \m S^1)}  \left\|\partial_y \mathcal{Z}^{\le 5}W\right\|^2_{L^2(\m S^2\times \m S^1)}\, r^2 dr.
	\end{aligned}
	\]
	Applying the inequality \eqref{wks2} to $Z\mathcal{Z}^{\le 4}u$ with $\beta=\alpha$ and successively the weighted Hardy inequality proved in corollary \ref{corhar} with $\beta = \alpha-1$ we find that
	
	\[
	\begin{aligned}
		& \sup_{r\ge t-1} (2+r-t)^{\alpha}r^2  \left\| Z\mathcal{Z}^{\le 4}u\right\|^2_{L^2(\m S^2\times \m S^1)}\\
		& \lesssim \left\|(2+r-t)^{\frac{\alpha+1}{2}}\partial_r \mathcal{Z}^{\le 5}u\right\|^2_{L^2_{xy}(\Sext_t)} + \left\|(2+r-t)^{\frac{\alpha-1}{2}} Z \mathcal{Z}^{\le 4}u\right\|^2_{L^2_{xy}(\Sext_t)}\\
		& \lesssim  \left\|(2+r-t)^{\frac{\alpha+1}{2}}\partial \mathcal{Z}^{\le 5}u\right\|^2_{L^2_{xy}(\Sext_t)}.
	\end{aligned}
	\]
	We can therefore continue the previous chain of inequalities 
	\[
	\begin{aligned}
		& \lesssim \left\|(2+r-t)^{\frac{\alpha+1}{2}}\partial_r \mathcal{Z}^{\le 5}u\right\|^2_{L^2_{xy}(\Sext_t)} \int_{t-1}^\infty (2+r-t)r^{-2}   \left\|\partial_y \mathcal{Z}^{\le 5}W\right\|^2_{L^2(\m S^2\times \m S^1)}\, r^2dr \\
		& \lesssim (t-1)^{-1-\alpha} \left\|W\right\|^2_{X^{5, \alpha}_{T_0}} \left\| (2+r-t)^{\frac{\alpha}{2}}\partial_y \mathcal{Z}^{\le 5}W \right\|^2_{L^2_{xy}(\Sext_t)}
	\end{aligned}
	\]
	and finally conclude that
	\[
	\sum_{\substack{|\gamma_1|+|\gamma_2|=n\\ |\gamma_1|\ge 1}}\left\|(2+r-t)^\frac{\alpha+1}{2} u^{\gamma_1}\cdot\partial^2_yW^{\gamma_2}\right\|_{L^2_{xy}}\lesssim C_0^2\ep^2 (t-1)^{-\frac{\alpha+1}{2}}\sqrt{l(t)}.
	\]\vspace{-10pt}
	\endproof
	
	An immediate consequence of the global energy bounds \eqref{global_strongnorm} obtained from proposition \ref{prop:boot_ext} is the following estimate on the higher order energies of the solution $W$ on the truncated exterior hyperboloids $\hex_s$.
	
	\begin{proposition}\label{prop:est_enex_hyp}
		Let $W=(u,v)^T$ be a global solution to the Cauchy problem \eqref{eq_W}-\eqref{data_W} in the exterior region $\dext$. Then
		\[
		\enexh_{5}(s, W) +\sum_{|\gamma|\le 5} \iint_{\CT}  |\mathcal{T}\q Z^\gamma W|^2 + |\partial_y \q Z^\gamma W|^2\, d\sigma dy dt\lesssim C_0^2\ep^2, \qquad s\in [2, \infty).
		\]
	\end{proposition}
	\proof
	The result follows by applying proposition \ref{prop:ext_hyp} with $\Wf = W^\gamma$ and $\Ff  =\Ff^\gamma$ for any multi-index $\gamma$ such that $|\gamma|\le 5$ and then using the global energy bound \eqref{global_strongnorm}, the estimate \eqref{est_Fgamma} of the source term $\Ff^\gamma$ and the fact that
	\[
	\left\|\Ff^\gamma\right\|_{L^1_tL^2_{xy}(\hex_{[2, s]})}\lesssim \int_2^\infty \left\|\Ff^\gamma\right\|_{L^2_{xy}(\Sext_t)}dt\lesssim C_0^2\ep^2.
	\]
	
	\endproof	
	
	We conclude this section with the derivation of a bound for the higher order exterior conformal energy of $W_0$ as well as for the higher order conformal energy on portions of the hyperboloid $\q H_s$ in the exterior region $\dext$.
	
	\begin{proposition}\label{prpconfext}
		Assume we have a solution $W = (u,v)^T$ to the Cauchy problem \eqref{eq_W}-\eqref{data_W} in the exterior region $\dext_{T_0}$ that satisfies the a-priori exterior energy bounds \eqref{boot_energy}. Then \small{
			\begin{equation}\label{est_conf_enex}
				\begin{aligned}
					\sup_{[2, T_0]}\exec_{4}(t, W_0) + \sum_{|\gamma|\le 4}\int_{\q C_{[2, T_0]}} \hspace{-10pt} \left| (t+r)(\partial_t W^\gamma_0+\partial_r W^\gamma_0) + 2W^\gamma_0\right|^2+(t-r)^2(|\nabla W^\gamma_0|^2-(\partial_r W^\gamma_0)^2)\, d\sigma dt\\ \lesssim C_0^2 \ep^2 \ln T_0 ,
				\end{aligned}
		\end{equation}}
		where the implicit constant only depends on $C_0$.
	\end{proposition}
	\proof 
	Let us fix $|\gamma|\le 4$. By integrating \eqref{systeme_vf} over the sphere $\m S^1$ we obtain that $W_0^\gamma$ is solution to the following linear inhomogeneous wave equation
	\begin{equation}\label{syst_u0}
		(-\partial^2_t + \Delta_x) W^\gamma_0 = \mathbf{F}^\gamma_0 
	\end{equation}
	with source term
	\begin{equation}\label{nonlin_0}
		\begin{aligned}
			\mathbf{F}^\gamma_0 & = \int_{\m S^1} \mathbf{F}^\gamma \, dy - \int_{\m S^1} u \cdot \partial^2_y W^\gamma \, dy \\
			&=  \sum_{\substack{|\gamma_1|+ |\gamma_2| \le |\gamma|}}\int_{\m S^1} \mathbf{N}(W^{\gamma_1}, W^{\gamma_2})\, dy + \sum_{|\gamma_1|+|\gamma_2|=|\gamma|}\int_{\m S^1} \partial_y u^{\gamma_1} \cdot \partial_y W^{\gamma_2} \, dy  .
		\end{aligned}
	\end{equation}
	When applying proposition \ref{prop:conf_ext} with $\Wf_0 = W_0^\gamma$ and $\Ff_0 = \Ff_0^\gamma$ we derive that for all $T\in [2, T_0]$
	\begin{multline}\label{conf_ext_Wgamma}
		\exec(T, W_0^\gamma) +\int_{\q C_{[2, T]}}\left| (t+r)(\partial_t W^\gamma_0+\partial_r W^\gamma_0) + 2W^\gamma_0\right|^2+(t-r)^2(|\nabla W^\gamma_0|^2-(\partial_r W^\gamma_0)^2)\, d\sigma dt\\
		\le \exec(2,  W_0^\gamma) +\int_2^T\left\|(t+r)\Ff^\gamma_0 \right\|_{L^2_x(\Sext_t)}\exec(t, W_0^\gamma)^\frac{1}{2} dt
	\end{multline}
	where
	\[
	\left\|(t+r)\Ff_0^\gamma\right\|_{L^2_x(\Sext_t)}\le \hspace{-15pt}\sum_{|\gamma_1|+|\gamma_2| \le |\gamma|}\left\|(t+r)\mathbf{N}(W^{\gamma_1}, W^{\gamma_2}) \right\|_{L^2_{xy}} + \sum_{|\gamma_1|+|\gamma_2| = |\gamma|} \left\|(t+r) \partial_y u^{\gamma_1} \cdot \partial_y W^{\gamma_2}\right\|_{L^2_{xy}}.
	\]
	We estimate the different contributions to $\Ff_0^\gamma$ separately. We start by observing that since $|\gamma_1| + |\gamma_2|\le 4$, at least one of the two multi-indices has length less or equal than 2. We call this index $\gamma_j$ and will place the factor carrying this index in $L^\infty$, the other one in $L^2$.
	
	\smallskip
	\textit{1. The $\partial_yu^{\gamma_1}\cdot\partial_yW^{\gamma_2}$ term}: we use the pointwise bound \eqref{dbarwext} and the energy bound \eqref{boot_energy} and obtain that there exists $l\in L^1([2,T_0])$ such that
	\[
	\left\|(t+r)\partial_yu^{\gamma_1}\cdot\partial_yW^{\gamma_2}\right\|_{L^2_{xy}(\Sext_t)}\lesssim C_0\ep \sqrt{l(t)}\left\| \partial_y\q Z^{\le 4}W\right\|_{L^2_{xy}(\Sext_t)}\lesssim C_0^2\ep^2l(t).
	\]
	
	\smallskip
	\textit{2. The null terms}: here we use the null form representation via the formula \eqref{rewrite} and the relation $\op=t^{-1}Z$ to write
	\begin{equation}\label{ext_null}
		\mathbf{N}(W^{\gamma_1}, W^{\gamma_2}) = \frac{1}{t}Z W^{\gamma_1}\cdot\partial W^{\gamma_2} + \frac{1}{t}\partial W^{\gamma_1}\cdot Z W^{\gamma_2}+\frac{t-r}{t}\partial W^{\gamma_1}\cdot\partial W^{\gamma_2}.
	\end{equation}
	The first two products in the above right hand side are equivalent given the range of $\gamma_1$ and $\gamma_2$ so we will just analyze the first one. In the case where $|\gamma_1|\le 2$ we deduce from the pointwise bound \eqref{zwext} that 
	\[
	\left\| (t+r)t^{-1}Z W^{\gamma_1}\cdot\partial W^{\gamma_2}\right\|_{L^2_{xy}(\Sext_t)}\lesssim C_0\ep  \left\| (t+r)r^{-1}t^{-1} \partial W^{\gamma_2} \right\|_{L^2_{xy}(\Sext_t)} \lesssim C_0\ep t^{-1}E^{\text{ex}, 0}_{4}(t, W)^\frac{1}{2}.
	\]
	In the case where $|\gamma_1|\ge 3$ we estimate $\partial W^{\gamma_2}$ with the pointwise bound \eqref{dwext} and decompose $W^{\gamma_1}$ using \eqref{dec}. On the one hand we apply the Poincar\'e inequality to obtain
	\[
	\left\|(t+r)t^{-1}Z \Wt^{\gamma_1}\cdot\partial W^{\gamma_2}\right\|_{L^2_{xy}(\Sext_t)}\lesssim C_0\ep  \left\| (t+r)r^{-1}t^{-1} \partial_yZ\Wt^{\gamma_1} \right\|_{L^2_{xy}(\Sext_t)}\lesssim C_0\ep t^{-1}E^{\text{ex}, 0}_{5}(t, W)^\frac{1}{2}
	\]
	and on the other hand we use corollary \ref{corhar} with $\beta=\alpha-1$ to get
	\[
	\begin{aligned}
		\left\|(t+r) t^{-1}Z W_0^{\gamma_1}\cdot\partial W^{\gamma_2}\right\|_{L^2_{xy}(\Sext_t)}& \lesssim 
		C_0\ep t^{-1}	\left\|(2+r-t)^{-\alpha+\frac{(\alpha-1)}{2}}Z W_0^{\gamma_1}\right\|_{L^2_{xy}(\Sext_t)} \\
		&\lesssim	C_0\ep t^{-1}\enext_{5}(t, W)^\frac{1}{2}.
	\end{aligned}
	\]
	The last quadratic term in the right hand side of \eqref{ext_null} is estimated using again \eqref{dwext} 
	\[
	\begin{aligned}
		\left\|(t+r)(t-r)/t \, \partial W^{\gamma_1}\cdot\partial W^{\gamma_2}\right\|_{L^2_{xy}(\Sext_t)}& \lesssim C_0\ep t^{-1}\left\|(2+r-t)^{\frac{1-\alpha}{2}}\partial \q Z^{\le 4}W\right\|_{L^2_{xy}(\Sext_t)} \\
		& \lesssim C_0 \ep t^{-1}E^{\text{ex},0}_{4}(t, W)^\frac{1}{2},
	\end{aligned}
	\]
	which concludes that
	\[
	\left\|\mathbf{N}(W^{\gamma_1}, W^{\gamma_2})\right\|_{L^2_{xy}(\Sext_t)}\lesssim C_0\ep t^{-1}\enext_{5}(t, W)^\frac{1}{2}.
	\]
	The combination of step 1 and step 2 with the energy bound \eqref{boot_energy} yields 
	\begin{equation}\label{est_F0_ext}
		\left\|(t+r)\Ff_0^\gamma\right\|_{L^2_{xy}(\Sext_t)}\lesssim C_0^2\ep^2 l(t) + C_0^2 \ep^2 t^{-1},
	\end{equation}
	which plugged in \eqref{conf_ext_Wgamma} for all $|\gamma|\le4 $ gives{
		\[
		\begin{aligned}
			\exec_{4}(T, W_0) &+ \sum_{|\gamma|\le 4}\int_{\q C_{[2, T]}} \hspace{-10pt} \left| (t+r)(\partial_t W^\gamma_0+\partial_r W^\gamma_0) + 2W^\gamma_0\right|^2+(t-r)^2(|\nabla W^\gamma_0|^2-(\partial_r W^\gamma_0)^2)\, d\sigma dt\\ 
			& \lesssim \exec_{4}(2,  W_0) + \int_2^T \left(C_0^2\ep^2 l(t) + C_0^2 \ep^2 t^{-1}\right) \exec_4(t, W_0)^\frac{1}{2} dt \\
			& \lesssim \exec_{4}(2,  W_0) + C_0^2\ep^2\sup_{[2, T_0]}\exec_4(t, W_0) +  C_0^2\ep^2 \ln T .
		\end{aligned}
		\]
		Therefore, assuming $\ep \ll1$ sufficiently small we get that \small{
			\[
			\begin{aligned}
				\sup_{[2, T_0]}\exec_4(t, W_0)  + \sum_{|\gamma|\le 4}\int_{\q C_{[2, T_0]}} \hspace{-10pt} \left| (t+r)(\partial_t W^\gamma_0+\partial_r W^\gamma_0) + 2W^\gamma_0\right|^2+(t-r)^2(|\nabla W^\gamma_0|^2-(\partial_r W^\gamma_0)^2)\, d\sigma dt \\
				\lesssim \exec_{4}(2,  W_0) +  C_0^2\ep^2 \ln T_0 
			\end{aligned}
			\]}
		so the result of the proposition  follows from the smallness of the conformal energy at the initial time, which in turn follows from the assumptions on the initial data.}
	\endproof
	
{	
		\begin{lemma}\label{Lem:conf_ext_hyp}
			Let $s\ge 2$ and $2<T_1<T_2$ be such that the portion of hyperboloid $\q H_s$ in the time strip $[T_1, T_2]$ is entirely contained in the exterior region $\dext$. Assume $W=(u,v)^T$ is the solution to the Cauchy problem \eqref{eq_W}-\eqref{data_W} in $\dext$ satisfying the global energy bounds \eqref{boot_energy}. Then there exists constants $C_1, C_2>0$ such that
			\[
			\sum_{|\gamma|\le 4}\int_{\q H_s\cap [T_1, T_2]} \frac{1}{t^2}\left|K W^\gamma_0 + 2t W^\gamma_0\right|^2 + \frac{s^2}{t^2}\sum_{i=0}^3\left|\Omega_{0j} W^\gamma_0\right|^2\, dx \le  C_1C_0^2\ep^2 \ln T_1 + C_2C_0^2\ep^2\ln \left(\frac{T_2}{T_1}\right)
			\]
		\end{lemma}
		\proof
		The result follows from the application of proposition \ref{prop: twisted_conformal} with $\Wf_0=\q Z^\gamma W$ with $|\gamma|\le 4$, from the observation that $\Sigma^2_t \subset \Sext_t$ for all $t \in [T_1, T_2]$
		and from the estimates \eqref{est_F0_ext} and \eqref{est_conf_enex}.
		\endproof}

	\subsection{A sharper pointwise bound for the non-zero modes}
	
	The scope of this section is to show that the pointwise bound \eqref{dbarwext} for $\partial_y\q Z^{\le 2}W$ and $\q Z^{\le 2}\Wt$  obtained by Sobolev's injections can be improved and that one can replace $\sqrt{l(t)}$ by the explicit decay in time $t^{-1/2}$. 
	
	\begin{proposition}
		We have the following pointwise estimate in the exterior region $\dext$
		\[
		\sup_{\m S^1}|\partial_y \q Z^{\leq 2}W| +  |\q Z^{\le 2}\Wt| \lesssim C_0\ep r^{-1}t^{-\frac{1}{2}}(2+r-t)^{-\frac{\alpha}{2}}.
		\]
	\end{proposition}
	\proof
	Thanks to lemma \ref{lemma_wsi2} and the Poincar\'e inequality we have that
	\[
	(2+r-t)^\alpha r^2 \sup_{\m S^1}\left(|\partial_y \q Z^{\leq 2}W|^2 + |\q Z^{\le 2}\Wt| \right)\leq \iint_{\Sext_t} (2+r-t)^\alpha    |\partial^{\leq 1}\partial_y \q Z^{\leq 4}W|^2   \, dxdy ,
	\]
	so it will be enough to prove that
	\[
	I(t):=\iint_{\Sext_t} (2+r-t)^\alpha    |\partial_y \q Z^{\leq 5}W|^2   \, dxdy \lesssim C_0^2\ep^2 t^{-1}, \quad t\in [2, \infty).
	\]
	We fix $t\in [2, \infty)$ and let $k\in \mathbb{Z}$ be such that $t\in [2^k, 2^{k+1})$.
	Since $I\in L^1([2, \infty))$ there exists $\tau_k\in [2^k, 2^{k+1}]$ such that
	\[
	I(\tau_k)\lesssim 2^{-k}\lesssim t^{-1}.
	\]
	Observe that in the above inequalities the implicit constants are independent of $k$. We can assume $\tau_k<t$ (the case $\tau_k>t$ being similar) and use the fundamental theorem of calculus to write that
	\[
	|I(t)-I(\tau_k)|\le \int_{\tau_k}^k |\partial_tI(s)|\, ds.
	\]
	By integration over $\Sext_t$ of the equality \eqref{weighted_div} with $\Wf = \partial_y\q Z^{\le 4}W$ and $\omega(z)=(1+z)^\alpha$ and the smallness of $u$ given by \eqref{wext} we derive that
	\[
	|\partial_tI(t)|\lesssim \iint_{\Sext_t} (2+r-t)^\alpha\left|\partial_yu \partial_t \Wf \partial_y \Wf - \frac{1}{2}\partial_tu (\partial_y\Wf)^2\right| + (2+r-t)^\alpha \Ff^{\le 5}\partial_t \Wf\, dxdy
	\]
	where $\Ff^{\le 5}=\sum_{|\gamma|\le 5}\Ff^{\gamma}$ and $\Ff^\gamma$ is the nonlinear term appearing in the equation satisfied by $\mathcal{Z}^{\gamma}W$ given by \eqref{nonlin_vf}. 
	On the one hand, the pointwise bounds \eqref{dwext} and \eqref{dbarwext} and the energy bound \eqref{boot_energy} imply that
	\[
	\iint_{\Sext_t} (2+r-t)^\alpha\left|\partial_yu \partial_t \Wf \partial_y \Wf - \frac{1}{2}\partial_tu (\partial_y\Wf)^2\right| \, dxdy \lesssim C^3_0\ep^3 (t-1)^{-1}l(t).
	\]
	On the other hand, from the proof of proposition \ref{prop:boot_ext} it follows that $\Ff^{\le 5} = \Ff^a + \Ff^b$, where $\Ff^a$ contains the terms analyzed in case 1 and $2.a$ such that 
	\[
	\|(2+r-t)^{\frac{\alpha+1}{2}}\Ff^a\|_{L^2_{xy}}\lesssim C_0^2\ep^2 (t-1)^{-1}\sqrt{l(t)} 
	\]
	and $\Ff^b$ contains the products treated in case $2.b$ that only satisfy
	\[
	\left\|(2+r-t)^\frac{\alpha+1}{2} \Ff^b\right\|_{L^2_{xy}}\lesssim C_0^2\ep^2(t-1)^{-\frac{\alpha+1}{2}}\sqrt{l(t)}.
	\]
	The estimate for $\Ff^b$ with the weaker weight $(2+r-t)^{\frac{\alpha}{2}}$ can actually be improved  to the following one, obtained using Sobolev injections and the pointwise bound \eqref{dbarwext}
	\[
	\begin{aligned}
		\left\| (2+r-t)^\frac{\alpha}{2} \Ff^b\right\|_{L^2_{xy}}^2 & \lesssim	  \left\|(2+r-t)^{\frac{\alpha+1}{2}}\partial \mathcal{Z}^{\le 5}u\right\|^2_{L^2_{xy}(\Sext_t)} \int_{t-1}^\infty r^{-2}   \left\|\partial_y \mathcal{Z}^{\le 5}W\right\|^2_{L^2(\m S^2\times \m S^1)}\, r^2dr \\
		& \lesssim C^4_0 \ep^4 (t-1)^{-2}l(t).
	\end{aligned}
	\]
When $\Wf = \partial_y \q Z^{\le 5}$ this yields
\begin{align*}\iint_{\Sext_t} (2+r-t)^\alpha \Ff^{\le 5}\partial_t \Wf\, dxdy
	\lesssim& C^2_0\ep^2 (t-1)^{-1}\sqrt{l(t)} \left\| (2+r-t)^\frac{\alpha}{2}\partial_y \q Z^{\leq 5}W\right\|_{L^2_{x,y}}\\
	 \lesssim & C^3_0\ep^3 (t-1)^{-1}l(t),
	\end{align*}
	which finally allows us to conclude that
	\[
	|I(t)-I(\tau_k)|\lesssim \int_{\tau_k}^t C_0^3\ep^3 (s-1)^{-1}l(s) \, ds\lesssim C_0^3\ep^3 t^{-1}.
	\]
	\endproof

	\section{Pointwise Estimates in the Interior Region}\label{PointwiseEst}
	
	The goal of this section is to recover pointwise estimates for the solution $W=(u,v)^T$ to \eqref{eq_W} in the interior hyperbolic region $\hin_{[2, s_0]}$ under the a-priori assumptions \eqref{booten1}-\eqref{bootZ} and to propagate the a-priori pointwise estimate \eqref{bootZ} on $ZW_0$.
	
	We remind the reader that we proved uniform-in-time energy bounds for the solution on the exterior truncated hyperboloids $\hex_s$ (see proposition \ref{prop:est_enex_hyp}) as well as the exterior pointwise bound \eqref{Z0_ext} for $ZW_0$, therefore if we suppose that $A, B\gg C_0$ we can think of the a-priori bounds \eqref{booten1}-\eqref{bootZ} as being valid not only on $\hin_s$ but on the whole hyperboloid $\q H_s$ for any $s\in [2, s_0]$. The pointwise estimates we will obtain in this section will then be valid in the whole hyperbolic strip between $\q H_2$ and $\q H_{s_0}$, that we denote by $\q H_{[2, s_0]}$
	\[
	\q H_{[2, s_0]} :=\left\{(t,x) : 4\le t^2-r^2\le s_0^2\right\}\times \m S^1.
	\]
	
	\subsection{Pointwise estimates from Klainerman-Sobolev inequalities}
	A first subset of pointwise estimates for $W_0$ and $\Wt$ is immediately obtained from \eqref{booten1} and \eqref{booten2} via the following Sobolev inequality on hyperboloids, whose proof can be found in \cite{lefloch-ma:global_stability_2}.
	\begin{lemma}\label{Lem:KS}
		Let $\Wf = \Wf(t,x)$ be a sufficiently regular function in the cone $\q C = \{t>r\}$. For all $(t, x)\in \q C$, let $s= \sqrt{t^2-r^2}$ and $B(x, t/3)$ be the ball centered at $x$ with radius $t/3$. Then
		\[
		|\Wf(t,x)|^2\le C t^{-3} \sum_{|\gamma|\le 2} \int_{B(x, t/3)} \left|Z^\gamma \Wf(\sqrt{s^2+|y|^2}, y)\right|^2 \, dy
		\]
		where $C$ is a positive universal constant and $Z_j = x_j\partial_t + t\partial_j$, $j=\overline{1,3}$. 
	\end{lemma}
	
	\begin{lemma} \label{Lem:KS_estimates}
		Let $\q I_{n,k}$ denote the set of multi-indices of type $(n,k)$.
		Under the a-priori energy bounds \eqref{booten1} and \eqref{booten2} we have the following pointwise estimates in $\q H_{[2, s_0]}$
		\begin{align}
			\label{ks1}
			& |\partial \mathcal{Z}^{\le 3}W_0(t,x)|\lesssim\epsilon t^{-\frac{1}{2}}s^{-1}, \\
			\label{ks2}
			&	|\op \mathcal{Z}^{\le 3}W_0(t,x)|\lesssim\ep t^{-\frac{3}{2}}, \\
			\label{ks3}
			\sum_{\q I_{3,k}}	\left\|\mathcal{Z}^\gamma \Wt(t,x,\cdot)\right\|_{L^\infty(\m S^1)} &+ \left\|\partial^{\le 1}_y\mathcal{Z}^\gamma\Wt(t,x,\cdot)\right\|_{L^2(\m S^1)}\lesssim \ep t^{-\frac{3}{2}}s^{\delta_{k+2}}, \qquad k=\overline{0,3} \\
			\label{ks3.1}
			& \hspace{-2cm}\sum_{|\gamma|=3} \left\| \partial_{tx}\q Z^\gamma \Wt(t,x,\cdot)\right\|_{L^2(\m S^1)}\lesssim \ep t^{-\frac{1}{2}}s^{-1+\delta_5},
		\end{align}
		and
		\begin{align}
			\label{esttan} \sum_{\q I_{2,k}} \left\|\bar{\partial} \partial^{\le 1}_y \mathcal{Z}^{\gamma} \Wt(t,x,\cdot) \right\|_{L^2(\m S^1)}&\lesssim \ep t^{-\frac{5}{2}}s^{\delta_{k+3}}, \qquad k=\overline{0,2} \\
			\label{esttan2}\sum_{\q I_{1,k}} \left\|\bar{\partial}^2 \partial^{\le 1}_y \mathcal{Z}^{\gamma} \Wt(t,x,\cdot) \right\|_{L^2(\m S^1)}&\lesssim \ep t^{-\frac{7}{2}}s^{\delta_{k+4}}, \qquad k=\overline{0,1}\\
			\label{esttan3} \left\|\bar{\partial}^2 \partial_{tx} \mathcal{Z} \Wt(t,x,\cdot) \right\|_{L^2(\m S^1)} & \lesssim \ep t^{-\frac{5}{2}}s^{-1+\delta_5}.
		\end{align}
		Moreover
		\begin{equation} \label{ks4}
			\sup_{\q H_{[2, s_0]}}|W|  \lesssim \ep.
		\end{equation}
	\end{lemma}	
	\proof
	The estimates \eqref{ks1} to \eqref{ks3.1} are immediate consequence of lemma \ref{Lem:KS} and \eqref{booten1}-\eqref{booten2}. The estimates on the $\op$ derivatives of $\Wt$ are deduced from \eqref{ks3} and \eqref{ks3.1} using the fact that $\op = t^{-1}Z$. Finally, if we set $\tilde{W}_0(s,x)=W_0(\sqrt{s^2+r^2},x)$ we see that
	\[
	\begin{aligned}
		|\tilde{W}_0(s,x) - \tilde{W}_0(2,x)| &\le \int_2^s \left| \partial_\tau \tilde{W}_0 (\tau, x)\right|d \tau = \int_2^s \frac{\tau}{t}\left| \partial_t W_0(\sqrt{\tau^2+x^2}, x)\right| d\tau \lesssim \ep
	\end{aligned}
	\]
	and hence derive from the smallness of the initial data that
	\[
	|W_0(t,x)|\lesssim |\tilde{W}_0(s,x) - \tilde{W}_0(2,x)| + |\tilde{W}_0(2,x)|\lesssim \ep.
	\]
	The combination of the above estimates with \eqref{ks3} yields \eqref{ks4}.
	\endproof
	
	\subsection{Improved pointwise estimates on the non-zero modes}
	The bounds for the $\Wt$ obtained in lemma \ref{Lem:KS_estimates} via Sobolev embeddings are affected by the small growth in $s$ of the energies and are not sharp but they can be improved if one studies more closely the equation satisfied by $\Wt$. Enhancing such bounds, and in particular \eqref{ks3}, will be fundamental to propagate the a-priori pointwise bound \eqref{bootZ} later in proposition \ref{prpw0}.
	We make use of the following result, which is motivated by the work of Klainerman \cite{klainerman:global_existence} and whose proof is an adaptation of a similar estimate for Klein-Gordon equations initially proved in \cite{LeFloch-Ma:global_nl_stability}, later revisited in \cite{DW20} in the case of Klein-Gordon equations with variable mass.

	\begin{proposition}\label{lmkg}
		Assume $\Wf$ is a solution of the following equation
		\begin{equation}\label{eq_Wt}
			\Box_{x,y} \Wf+ u\Delta_y \Wf= \Ff, \qquad (t,x,y)\in \m R^{1+3}\times \m S^1
		\end{equation}
		such that $\int_{\m S^1} \Wf dy=0$. For every fixed $(t,x)$ in the cone $\q C = \{t>r\}$, let $s=\sqrt{t^2-r^2}$ and $Y_{tx}, A_{tx}, B_{tx}$ be the functions defined as follows
		\begin{equation}\label{Y}
			Y^2_{tx}(\lambda) := \int_{\m S^1} \lambda \left|\frac{3}{2}\Wf_\lambda+ (\scal\Wf)_\lambda \right|^2 +\lambda^{3} (1+u_\lambda)|\partial_y \Wf_\lambda|^2\, dy
		\end{equation}
		\begin{equation} \label{A}
			A_{tx}(\lambda) := \sup_{\m S^1}\left|\frac{1}{2\lambda}(\scal u)_\lambda\right| + \sup_{\m S^1}\left| \partial_yu_\lambda\right|
		\end{equation}
		\begin{equation}\label{B}
			B^2_{tx}(\lambda)= \int_{\m S^1} \lambda^{-1}\left|(R\Wf)_\lambda\right|^2\, dy
		\end{equation}
		where $f_\lambda(t,x,y)=f\left(\frac{\lambda t}{s},\frac{\lambda r}{s},y \right)$ and
		\[
		R\Wf(t,x,y) = s^2
		\bar{\partial}^i \bar{\partial}_i \Wf + x^i x^j \bar{\partial}_i \bar{\partial}_j \Wf + \frac{3}{4} \Wf + 3x^i \bar{\partial}_i \Wf -s^2\Ff.
		\]
		Then $\Wf$ satisfies the following inequality in the hyperbolic region $\q H_{[2, \infty)}$
		\[
		s^\frac{3}{2}\left(\|\Wf\|_{L^2(\m S^1)}+\|\partial_y \Wf\|_{L^2(\m S^1)}\right) + s^\frac{1}{2}\|\scal\Wf \|_{L^2(\m S^1)}\lesssim \left(Y_{tx}(2) +  \int_{2}^s B_{tx}(\lambda)d\lambda\right) e^{\int_{2}^s A_{tx}(\lambda)\, d\lambda}.
		\]
	\end{proposition}
	
	\begin{proof}
		For every fixed $(t,x,y)\in \q H_{[2, \infty)}$ we define $\omega_{txy}(\lambda):=  \lambda^\frac{3}{2}\Wf(\frac{\lambda t}{s},\frac{\lambda x}{s},y)$ to be the evaluation of $\Wf$ on the hyperboloid $\q H_\lambda$, then dilated by factor $\lambda^{3/2}$. We have that
		\[
		\dot{\omega}_{txy}(\lambda)= \lambda^{1/2}\left(\frac{3}{2}\Wf_\lambda + (\scal\Wf)_\lambda\right)
		\]
		and
		\[
		\ddot{\omega}_{txy}(\lambda)= \frac{1}{\lambda^\frac{1}{2}} (P\Wf)_\lambda 
		\]
		where
		\begin{align*}P\Wf &= \frac{3}{4}\Wf + 3(t\partial_t \Wf + x^i\partial_i \Wf)+(t^2\partial^2_t \Wf + 2tx^i\partial_i \partial_t \Wf + x^i x^j \partial_i \partial_j \Wf).
		\end{align*}
		Using the equation \eqref{eq_Wt} we derive that $\omega_{txy}$ satisfies the following equation
		\begin{align*}
			\ddot{\omega}_{txy}-\left(1+u_\lambda\right)\Delta_y \omega_{txy} = -\lambda^\frac{3}{2} \Ff_\lambda
			+\lambda^{-\frac{1}{2}} \left( s^2
			\bar{\partial}^i \bar{\partial}_i \Wf + x^i x^j \bar{\partial}_i \bar{\partial}_j \Wf + \frac{3}{4} \Wf + 3x^i \bar{\partial}_i \Wf  \right)_\lambda.
		\end{align*}
		We drop the lower indices in $\omega_{txy}(\lambda)$ in order to have a lighter notation and simply denote it by $\omega(\lambda)$ in what follows.
		We multiply the above equation by $\partial_\lambda \omega$ and integrate over $\m S^1$
		\begin{align*}
			&\int_{\m S^1} \partial_\lambda\omega \left( \partial^2_\lambda\omega-\left(1+u_\lambda\right)\Delta_y \omega \right)dy\\
			=& \frac{d}{d\lambda} \left(\frac{1}{2}\int_{\m S^1} |\partial_\lambda\omega|^2\,  dy\right) + \int_{\m S^1} (1+u_\lambda )\partial_y \omega\, \partial_y \partial_\lambda \omega\,  dy
			+\int_{\m S^1} \partial_\lambda \omega\, \partial_y u_\lambda\,   \partial_y \omega \, dy\\
			=& \frac{d}{d\lambda} \left(\frac{1}{2}\int_{\m S^1} |\partial_\lambda \omega|^2+(1+u_\lambda)|\partial_y \omega|^2\, dy \right)-\frac{1}{2}\int_{\m S^1} \partial_\lambda u_\lambda \, |\partial_y \omega|^2dy    
			+\int_{\m S^1} \partial_\lambda\omega\, \partial_y u_\lambda \, \partial_y \omega\,  dy.
		\end{align*}
		We obtain that
		\[
		\frac{d}{d\lambda} Y^2_{tx}(\lambda)\lesssim   A_{tx}(\lambda)Y_{tx}^2(\lambda) + B_{tx}(\lambda) Y_{tx}(\lambda)
		\]
		with $A_{tx}, B_{tx}, Y_{tx}$ as in the statement and from the Gronwall lemma
		$$Y_{tx}(s) \lesssim \left(Y_{tx}(2) +  \int_{2}^s B_{tx}(\lambda)\, d\lambda\right) e^{\int_2^s A_{tx}(\lambda)\, d\lambda}.$$
		Finally, from the definition of $Y_{tx}$, the Poincar\'e inequality and the fact that $s\ge 1$ we get
		\[
		s^\frac{3}{2}\left(\|\Wf\|_{L^2(\m S^1)}+\|\partial_y \Wf\|_{L^2(\m S^1)}\right) + s^\frac{1}{2}\|\scal\Wf \|_{L^2(\m S^1)}\lesssim Y_{tx}(s).
		\]
	\end{proof}
	
	\begin{proposition}\label{prop:improved_KG}
		Under the a-priori assumptions \eqref{booten1}-\eqref{bootZ} we have
		\begin{equation}\label{est_wt2}
			\sup_{\q H_{[2, s_0]}}t^{\frac{3}{2}} \left( \left\| \partial^j \Wt \right\|_{L^2(\m S^1)} + \left\| \partial_y \partial^j\Wt\right\|_{L^2(\m S^1)}\right) + \sup_{\q H_{[2, s_0]}} t^\frac{3}{2}  \left\|  \scal\partial^{j}\Wt\right\|_{L^2(\m S^1)} \lesssim \epsilon, \quad j=\overline{0,1}
		\end{equation}
		\begin{equation}\label{est_Zwt}
			\sup_{\q H_{[2, s_0]}} t^{\frac{3}{2}} \left( \left\| Z \Wt \right\|_{L^2(\m S^1)} + \left\| \partial_y Z\Wt\right\|_{L^2(\m S^1)}\right) + \sup_{\q H_{[2, s_0]}} \frac{t^{\frac{3}{2}}}{s} \left\| \scal Z \Wt\right\|_{L^2(\m S^1)} \lesssim \epsilon s^\sigma,
		\end{equation}
		and
		\begin{equation} \label{est_w3}
			\sup_{\q H_{[2, s_0]}}	st^\frac{1}{2}\left(\|\partial^2\Wt\|_{L^2(\m S^1)} + \|\partial_y \partial^2\Wt\|_{L^2(\m S^1)}\right) + \sup_{\q H_{[2, s_0]}} t^\frac{1}{2}\|\scal \partial^2\Wt\|_{L^2(\m S^1)}\lesssim \ep ,
		\end{equation}
		\begin{equation} \label{est_Zwt2}
			\sup_{\q H_{[2, s_0]}} 	st^\frac{1}{2}\left( \left\|\partial Z\Wt\right\|_{L^2(\m S^1)} +  \left\|\partial_y\partial Z\Wt\right\|_{L^2(\m S^1)}\right) + t^\frac{1}{2}\left\| \scal \partial Z\Wt\right\|_{L^2(\m S^1)}\lesssim \ep s^\sigma.
		\end{equation}
	\end{proposition}
	\proof For any fixed $j=\overline{0,2}$ and $k=\overline{0,1}$ we compare the equation satisfied by the differentiated functions $\partial^{j}\Wt$ and $\partial^{k}Z\Wt$ respectively with the equation \eqref{eq_Wt} and apply the result of proposition \ref{lmkg}. A simple computation shows that
	\[
	\Box_{x,y} \partial^j\Wt + u \Delta_y \partial^j\Wt = F^j_1, \quad j=\overline{0,2}
	\]
	\[
	\Box_{x,y} \partial^k Z\Wt + u \Delta_y \partial^k Z\Wt = F^k_2, \quad k=\overline{0,1}
	\]
	with source terms given by
	\[
	\begin{gathered}
		F^0_1 = \mathbf{N}(W, W) - \int_{\m S^1}  \mathbf{N}(W,W)dy + \int_{\m S^1}\partial_y u\cdot \partial_y \Wt \, dy \\
		F^j_1 = \partial^j F^0_1 -\sum_{1\le h\le j} \partial^h u \cdot\partial^2_y \partial^{j-h}\Wt , \quad j=\overline{1,2}
	\end{gathered}
	\]
	and
	\[
	\begin{gathered}
		F^0_2 = ZF^0_1 - Zu\cdot \partial^2_y \Wt \\
		F^1_2 = \partial Z F^0_1 - \partial Zu \cdot \partial^2_y \Wt - \partial u \cdot \partial^2_y Z\Wt - Zu \cdot \partial^2_y \partial \Wt.
	\end{gathered}
	\]
	From proposition \ref{lmkg} we have that
	\[
	s^\frac{3}{2}\left(\|\Wf\|_{L^2(\m S^1)}+\|\partial_y \Wf\|_{L^2(\m S^1)}\right) + s^\frac{1}{2}\|\scal\Wf \|_{L^2(\m S^1)}\lesssim \left(Y_{tx}(2) +  \int_{2}^s B_{tx}(\lambda)d\lambda\right) e^{\int_{2}^s A_{tx}(\lambda)\, d\lambda}
	\]
	with $\Wf = \{\partial^j \Wt, \partial^kZ\Wt: j=\overline{0,2}, k=\overline{0,1}\}$ and corresponding source term $\Ff=\{F^j_1, F^k_2:  j=\overline{0,2}, k=\overline{0,1}\} $.
	In order to obtain the bounds in the statement we need to estimate the quantities $Y_{tx}(2)$, $A_{tx}(\lambda)$ and $B_{tx}(\lambda)$ defined in \eqref{Y}, \eqref{A}, \eqref{B} for all the different values of $\Wf$ and $\Ff$.
	
	\smallskip
	\textit{1. The $A_{tx}(\lambda)$ term}: this is the same for all values of $\Wf$. Here we decompose $u = u_0 + \ut$ and rewrite the scaling vector field as follows:
	\begin{equation}\label{dec_S}
		\scal = (t-r)\partial_t + (r-t)\partial_r + (t\partial_r + r\partial_t).
	\end{equation}
	Using the pointwise bounds \eqref{ks1}-\eqref{ks3} and the fact that $s/t\le 1$ in the interior of the light cone we derive that
	\[
	A_{tx}(\lambda)\lesssim \sup_{y \in \m S^1}  \frac{s}{t+r}\, |(\partial u)_\lambda|+\lambda^{-1}|(Zu)_\lambda| + |(\partial_yu)_\lambda| \lesssim \ep\lambda^{-\frac{3}{2}+\delta_5}
	\]
	and consequently
	\begin{equation}
		\label{estA} \int_1^s A_{tx}(\lambda)d\lambda \lesssim \ep.
	\end{equation}

	\textit{2. The $Y_{tx}(2)$ term}: the functions appearing here are evaluated on the hyperboloid $\q H_2$. From the bound \eqref{ks3} on $\Wt$, the smallness of $u$ given by \eqref{ks4} and the decomposition \eqref{dec_S} it follows that for all values of $\Wf$ under consideration
	\begin{equation}
		\label{estY1}|Y_{tx}(2)|\lesssim \|\Wf_2\|_{L^2(\m S^1)} + \|(\scal \Wf)_2\|_{L^2(\m S^1)} +\|1+u_2\|_{L^\infty(\m S^1)} \|(\partial_y \Wf)_2\|_{L^2(\m S^1)} \lesssim \ep \, \left( \frac{s}{t}\right)^{\frac{3}{2}}.
	\end{equation}
	
	The $B_{tx}(\lambda)$ terms are the only ones for which we need to distinguish between the different values of $\Wf$ and hence of $\Ff$.
	
	\smallskip
	\textit{3. The $B_{tx}(\lambda)$ term for $\Wf = \mathcal{Z}^{\le 1}\Wt$:} from the pointwise bounds \eqref{ks3}, \eqref{esttan} and \eqref{esttan2} we immediately obtain the following estimate 
	\begin{equation}\label{B1}
		\begin{aligned}
			&\left\|\lambda^{-1/2} \left(s^2
			\bar{\partial}^i \bar{\partial}_i \Wf + x^i x^j \bar{\partial}_i \bar{\partial}_j \Wf + \frac{3}{4} \Wf + 3x^i \bar{\partial}_i \Wf  \right)_\lambda\right\|_{L^2(\m S^1)}\\
			& \lesssim  \lambda^\frac{3}{2}\left(1+\frac{r^2}{s^2}\right)\|(\bar{\partial}^2\Wf)_\lambda\|_{L^2(\m S^1)}
			+ \lambda^\frac{1}{2}\,  \frac{r}{s}\|(\bar{\partial}\Wf)_\lambda\|_{L^2(\m S^1)}
			+\lambda^{-\frac{1}{2}}\| \Wf_\lambda\|_{L^2(\m S^1)} \lesssim \ep \lambda^{-2+\delta_5}\left(\frac{s}{t}\right)^{\frac{3}{2}}.
		\end{aligned}
	\end{equation}
	As concerns the estimate of the $L^2(\m S^1)$ norm of the source term $\Ff_\lambda$ we distinguish betweeen the different values $\{F^0_1, F^1_1, F^0_2\}$ it takes in this case, after observing that
	\[
	\left\| \lambda^{-\frac{1}{2}} \left( s^2 \Ff\right)_\lambda\right\|_{L^2(\m S^1)} = \lambda^\frac{3}{2} \|\Ff_\lambda\|_{L^2(\m S^1)}.
	\]

	\textit{3.a. The case $\Ff=F^0_1$}: this corresponds to $\Wf = \Wt$.
	We decompose each occurrence of $W$ in the formula for $F^0_1$  using \eqref{dec}
	\[
	\mathbf{N}(W, W)  =  \mathbf{N}(W_0, W_0) + 2\mathbf{N}(W_0, \Wt)+ \mathbf{N}(\Wt, \Wt) 
	\]
{and observe that $F^0_1$ does not make appear the zero-mode interactions $W_0\times W_0$. Simply using the pointwise bounds \eqref{ks1} and \eqref{ks3} we deduce that}
	\[
	\| F^0_1\|_{L^2(\m S^1)}\lesssim \ep^2 t^{-2}s^{-1+\delta_5}  ,
	\] and hence that
	\begin{equation}\label{estF}
		\lambda^{\frac{3}{2}}\|\Ff_\lambda
		\|_{L^2(\m S^1)} \lesssim \ep^2 \lambda^{-\frac{3}{2}+2\delta_5} \left(\frac{s}{t}\right)^2, \quad \Ff=F^0_1
	\end{equation}
	%
	%
	which summed up with \eqref{B1} yields
	\begin{equation}\label{intB}
		\int_2^s B_{tx}(\lambda)d\lambda \lesssim \ep \left(\frac{s}{t}\right)^{\frac{3}{2}}.
	\end{equation}
	The combination of \eqref{intB} with \eqref{estA} and \eqref{estY1} implies \eqref{est_wt2} for $j=0$.
	
	\smallskip
	\textit{3.b The case $\Ff=F^1_1$}: this corresponds to $\Wf = \partial \Wt$ with $\partial = \{\partial_t, \partial_x, \partial_y\}$. When $\Wf = \partial_y\Wt$ the argument used above yields \eqref{est_wt2}. When $\Wf = \partial_{tx}\Wt$ the only additional contribution that needs to be analyzed is $\partial_{tx} u_0 \cdot\partial^2_y \Wt$, which appears in $F^1_1$ after using the decomposition \eqref{dec}. This product is bounded
	using \eqref{ks1} and \eqref{ks3}
	\[
	\|\partial u_0 \cdot \partial^2_y \Wt\|_{L^2(\m S^1)} \lesssim \| \partial u_0\|_{L^\infty(\m S^1)}\| \partial^2_y \Wt\|_{L^2(\m S^1)}\lesssim \ep^2 t^{-2}s^{-1+\delta_5}.
	\]
	We obtain \eqref{estF} for $\Ff=F^1_1$ and consequently \eqref{intB}, which summed up with \eqref{estA} and \eqref{estY1} gives us \eqref{est_wt2} for $j=1$. 
	
	\smallskip
	
	\textit{3.c The case $\Ff = F^0_2$}: this corresponds to $\Wf = Z\Wt$. The term $ZF^0_1$ is estimated using the argument in case \emph{3.a} and the only contribution to analyze is $Zu\cdot\partial^2_y \Wt$.  We again decompose $u$ using \eqref{dec} and obtain from \eqref{ks3} that
	\[
	\left\| Z\ut \cdot \partial^2_y \Wt\right\|_{L^2(\m S^1)}\lesssim \ep^2 t^{-3+2\delta_5},
	\]
	while from the a-priori bound \eqref{bootZ} on $ZW_0$ and the enhanced bound \eqref{est_wt2} on $\partial^2_y \Wt$ we get 
	\[
	\left\| Zu_0\cdot \partial^2_y \Wt\right\|_{L^2(\m S^1)}\lesssim \ep^2 t^{-\frac{5}{2}}s^{\sigma}.
	\]
	Therefore from \eqref{estF} and the above bounds we derive that $F^0_2$ satisfied the following estimate
	\[
	\lambda^\frac{3}{2}\|(F^0_2)_\lambda\|_{L^2(\m S^1)}\lesssim \ep^2 \lambda^{-\frac{3}{2}+2\delta_5}\left(\frac{s}{t}\right)^2+  \ep^2 \lambda^{-1+\sigma} \left(\frac{s}{t}\right)^{\frac{5}{2}}\lesssim \ep^2 \lambda^{-1+\sigma}\left(\frac{s}{t}\right)^2
	\]
	which together with \eqref{B1} gives
	\[
	\int_2^s B_{tx}(\lambda)d\lambda \lesssim \ep s^\sigma \left(\frac{s}{t}\right)^{\frac{3}{2}}
	\]
	and finally \eqref{est_Zwt} when combined with \eqref{estA} and \eqref{estY1}.

	\medskip
	\textit{4. The $B_{tx}(\lambda)$ term for $\Wf = \{\partial^2\Wt, \partial Z\Wt\}$:} here the bounds \eqref{ks3}, \eqref{esttan} and \eqref{esttan3} give us that
	\begin{equation}\label{B2}
		\begin{aligned}
			&\left\|\lambda^{-1/2} \left(s^2
			\bar{\partial}^i \bar{\partial}_i \Wf  + \frac{3}{4} \Wf + 3x^i \bar{\partial}_i \Wf  \right)_\lambda\right\|_{L^2(\m S^1)}\\
			& \lesssim  \lambda^\frac{3}{2} \|(\bar{\partial}^2\Wf)_\lambda\|_{L^2(\m S^1)}
			+ \lambda^\frac{1}{2}\,  \frac{r}{s}\|(\bar{\partial}\Wf)_\lambda\|_{L^2(\m S^1)}
			+\lambda^{-\frac{1}{2}}\| \Wf_\lambda\|_{L^2(\m S^1)} \lesssim \ep \lambda^{-2+\delta_5}\left(\frac{s}{t}\right)^{\frac{3}{2}-\delta_5}
		\end{aligned}
	\end{equation}
	and
	\begin{equation}\label{B3}
		\left\| \lambda^{-\frac{1}{2}} (x^ix^j \overline{\partial}_i \overline{\partial}_j \Wf)_\lambda\right\|_{L^2(\m S^1)} \lesssim \ep \lambda^{-2+\delta_5}\left(\frac{s}{t}\right)^\frac{1}{2}.
	\end{equation}
	We then separately analyze the $L^2(\m S^1)$ norm of $\Ff_\lambda$ when $\Ff=\{F^2_1, F^1_2\}$.
	
	\textit{4.a The case $\Ff = F^2_1$:} this corresponds to $\Wf = \partial^2\Wt$ with $\partial = \{\partial_t, \partial_x, \partial_y\}$. The argument to perform here is analogous to the one in $3.b$ in that one should first consider the case where $\Wf = \partial^2_y\Wt$, followed by $\Wf = \partial_y \partial_{tx} \Wt$ and finally by $\Wf = \partial^2_{tx} \Wt$. It is a simple exercise to verify that
	\[
	\lambda^{\frac{3}{2}}\left\|( F^2_1)_\lambda\right\|_{L^2(\m S^1)}\lesssim \ep^2  \lambda^{-\frac{3}{2}+\delta_5}\left(\frac{s}{t}\right)^{\frac{5}{2}-\delta_5}
	\]
	which summed up with \eqref{B2} and \eqref{B3} gives
	\[
	\int_2^s B_{tx}(\lambda)d\lambda \lesssim \ep \left(\frac{s}{t}\right)^\frac{1}{2}.
	\]
	The combination of the above estimate with \eqref{estA} and \eqref{estY1} yields then \eqref{est_w3}.

	\vspace{5pt}
	\textit{4.b The case $\Ff=F^1_2$}: this corresponds to $\Wf = \partial Z\Wt$. The term $\partial ZF^0_1$ enjoys the same estimate \eqref{estF} as $F^0_1$ in case $3.a$.
	We decompose each occurrence of $u$ in the remaining contributions to $F^1_2$ using \eqref{dec} and focus on discussing the terms involving products $u_0\times \Wt$. The products $\ut \times \Wt$ are estimated using \eqref{ks3} and are bounded by $\ep^2 t^{-3+2\delta_5}$.
	
	When $\Wf = \partial_y Z\Wt$ the only term that is left to control is the $L^2(\m S^1)$ norm of $Zu_0\cdot\partial^3_y \Wt$, which is estimated using the a-priori bound \eqref{bootZ} and \eqref{est_w3} so that
	\[
	\left\|  Zu_0 \cdot \partial^2_y\partial
	\Wt\right\|_{L^2(\m S^1)}\le \|  Zu_0\|_{L^\infty(\m S^1)} \| \partial^2_y \partial\Wt\|_{L^2(\m S^1)}\lesssim \ep^2 t^{-\frac{3}{2}}s^{-1+\sigma}.
	\]
	Therefore in this case 
	\[
	\begin{aligned}
		\lambda^{\frac{3}{2}}\left\|(F^1_2)_\lambda\right\|_{L^2(\m S^1)}\lesssim \ep^2 \lambda^{-\frac{3}{2}+2\delta_5}\left(\frac{s}{t}\right)^2 + \ep^2 \lambda^{-1+\sigma} \left(\frac{s}{t}\right)^\frac{3}{2} \lesssim \ep^2 \lambda^{-1+\sigma} \left(\frac{s}{t}\right)^\frac{3}{2}
	\end{aligned}
	\]
	which together with \eqref{B2} and \eqref{B3} gives
	\begin{equation}\label{intB1}
		\int_2^s B_{txy}(\lambda)d\lambda \lesssim \ep s^{\sigma}  \left(\frac{s}{t}\right)^{\frac{1}{2}}.
	\end{equation}
	Summing this estimate up with \eqref{estA} and \eqref{estY1} we deduce \eqref{est_Zwt2} with $\partial = \partial_y$.
	
	When $\Wf = \partial_{tx}Z\Wt$ the remaining terms to control in the $L^2(\m S^1)$ norm are
	\[
	\partial Zu_0 \cdot \partial^2_y \Wt + \partial u_0 \cdot \partial^2_y Z\Wt + Zu_0 \cdot \partial^2_y \partial \Wt.
	\]
	The first term is estimated using the bounds \eqref{ks1} and \eqref{est_wt2} as follows
	\[
	\left\| \partial Zu_0 \cdot \partial^2_y \Wt\right\|_{L^2(\m S^1)}\le \| \partial Zu_0\|_{L^\infty(\m S^1)} \| \partial^2_y \Wt\|_{L^2(\m S^1)}\lesssim \ep^2 t^{-2}s^{-1},
	\]
	the second one using \eqref{ks1} and \eqref{est_Zwt2} with $\partial = \partial_y$ proved above
	\[
	\left\| \partial u_0 \cdot \partial^2_y Z\Wt\right\|_{L^2(\m S^1)}\le \| \partial u_0\|_{L^\infty(\m S^1)} \| \partial^2_y Z\Wt\|_{L^2(\m S^1)}\lesssim \ep^2 t^{-1}s^{-2+\sigma},
	\]
	the third one using the a-priori bound \eqref{bootZ} and  \eqref{est_w3}
	\[
	\left\| Z u_0 \cdot \partial^2_y\partial\Wt\right\|_{L^2(\m S^1)}\le \| Z u_0\|_{L^\infty(\m S^1)} \| \partial^2_y \partial\Wt\|_{L^2(\m S^1)}\lesssim \ep^2 t^{-\frac{3}{2}}s^{-1+\sigma}.
	\]
	We obtain in this case that
	\[
	\lambda^\frac{3}{2}\|(F^1_2)_\lambda\|_{L^2(\m S^1)}\lesssim \ep^2 \lambda^{-\frac{3}{2}+2\delta_5}\left(\frac{s}{t}\right)^2 + \ep^2 \lambda^{-\frac{3}{2}+\sigma}\left(\frac{s}{t}\right)+ \ep^2 \lambda^{-1+\sigma} \left(\frac{s}{t}\right)^\frac{3}{2}\lesssim \ep^2\lambda^{-1+\sigma}\left(\frac{s}{t}\right) 
	\]
	which together with \eqref{B2} and \eqref{B3} again gives \eqref{intB1}.
	Summing this estimate up with \eqref{estA} and \eqref{estY1} finally yields \eqref{est_Zwt2} when $\partial = \partial_tx$.
	\endproof
	
	\subsection{The propagation of the a-priori pointwise bound}
	
	In order to propagate the a-priori bound \eqref{bootZ} on the $Z$ derivative of $W_0$ in the interior region $\q H_{[2, s_0]}$ we use $L^\infty-L^\infty$ type estimates. For this we give a closer look to the wave equation satisfied by $ZW_0$ and use the enhanced pointwise bounds on the solution recovered in the previous subsection to estimate the nonlinear terms. The argument that follows will in addition provide us with a decay estimate for $W_0$ without derivatives and hence improve on \eqref{ks4}.
	
	We will make use of the following lemma, due to Alinhac \cite{Alinhac06}.

	\begin{lemma}\label{supint}
		Let $\Wf_0$ be the solution to $\Box_{tx} \Wf_0= \Ff_0$ with zero initial data and suppose that $\Ff_0$ is spatially compactly supported satisfying the following pointwise bound
		$$|\Ff_0(t,x)|\leq Ct^{-2-\nu}(t-|x|)^{-1+\mu}$$
		for some positive constant $C$. Then 
		$$|\Wf_0(t,x)|\lesssim \frac{C}{\mu \nu} (t-|x|)^{\mu -\nu}t^{-1}.$$
	\end{lemma}
	
	%
	\begin{proposition}\label{prpw0}
		There exists a constant $B>0$ sufficienly large and $\ep_0$ sufficiently small such that, for any $0<\ep<\ep_0$ if $W =(u,v)^T$ is solution to the Cauchy problem \eqref{eq_W}-\eqref{data_W} and satisfies the a-priori bounds \eqref{booten1}-\eqref{bootZ} in the interior region $\q H_{[2, s_0]}$ and the global energy bounds \eqref{boot_energy} in the exterior region $\dext$, then in $\q H_{[2, s_0]}$ it actually satisfies the enhanced pointwise bound
		\[
		|ZW_0(t,x)|\le B\ep t^{-1}s^\sigma.
		\]
	\end{proposition}
	\proof
	We are interested in propagating the a-priori pointwise bound \eqref{bootZ} and to recover improved estimates on $Z W_0$ in the interior region $\q H_{[2, s_0]}$, which is contained in the cone $\{t>r+1\}$.
{We consider a cut-off function $\chi\in C^\infty_0(\m R)$ such that $\chi(z)=1$ for $|z|\le 1/2$ and $\chi(z)=0$ for $|z|>1$, and decompose $W_0$ into the sum $W_0^1+W_0^2$, where $W_0^1$ and $W_0^2$ solve to the following Cauchy problems
		\[
		\Box W_0^1 = (1-\chi(t-r))F_0, \qquad  ( W_0^1, \partial_t W_0^1)|_{t=2} = (0,0)
		\]
		\[
		\Box_x W_0^2 =  \chi(t-r) F_0, \qquad (W_0^2, \partial_t W_0^2)|_{t=2} =  ( W_0, \partial_t  W_0)|_{t=2}.
		\]
		with $F_0$ given by
		\[
		F_0  =\int_{\m S^1}\mathbf{N}(W,W) dy +\int_{\m S^1} \partial_yu\cdot\partial_y\Wt\, dy .
		\]
		The scope of such decomposition is to estimate $ZW_0^1$ and $ZW_0^2$ separately and in particular to apply lemma \ref{supint} to $ZW_0^1$, which is solution to a wave equation with zero data and source term supported in the interior of the cone $\{t=r+1/2\}$
		\[
		\Box ZW_0^1  = (1-\chi)(t-r) ZF_0 - [Z, \chi(t-r)]F_0, \qquad (ZW_0^1, \partial_t ZW_0^1)|_{t=2} = (0,0).
		\]
		In order to do so, we need to estimate $F_0$ and $ZF_0$ in the region $t>r+1/2$. We remark that the commutator term is uniformly bounded and also supported in $\{t>r+1/2\}$.}
	
{
		We decompose each occurrence of $W$ in $\mathbf{N}(W,W)$ by means of \eqref{dec} 
		and use the null structure representation \eqref{rewrite} of $\mathbf{N}$ for the quadratic interactions $W_0\times W_0$. From \eqref{ks1},\eqref{ks2} and  \eqref{esttan} we deduce the following \small{
			\[
			\begin{aligned}
				\int_{\m S^1} |Z^{\le 1}\mathbf{N}(W_0, W_0)|\, dy & \lesssim  |\overline{\partial}Z^{\le 1}W_0| |\partial W_0| + |\partial Z^{\le 1}W_0||\op W_0| + \frac{s^2}{t^2} |\partial Z^{\le 1}W_0||\partial W_0| \lesssim \epsilon^2 t^{-2}s^{-1}
			\end{aligned}
			\]}{
			\[
			\begin{aligned}
				\int_{\m S^1} |Z^{\le 1}\mathbf{N}(W_0, \Wt)|\, dy  & \lesssim  \int_{\m S^1} |Z^{\le 1}(\partial W_0 \cdot\partial \Wt)|\, dy\lesssim  \epsilon t^{-\frac{3}{2}} \left\|\partial Z^{\le 1}\Wt \right\|_{L^2(\m S^1)} 
			\end{aligned}
			\]}
		\[
		\begin{aligned}
			\int_{\m S^1} |Z^{\le 1}\mathbf{N}(\Wt, \Wt)|\, dy + \int_{\m S^1}\left|Z^{\le 1}\left(\partial_yu\cdot\partial_y\Wt\right)\right|\, dy \lesssim \|\partial Z^{\le 1}\Wt\|_{L^2(\m S^1)} \|\partial \Wt\|_{L^2(\m S^1)}.
		\end{aligned}
		\]
		The improved bounds \eqref{est_wt2} and \eqref{est_Zwt2} on $\Wt$ coupled with the fact that $s=\sqrt{(t+r)(t-r)}$ imply that in $\{t>r+1/2\}$
		\[
		|\Ff_0(t,x)|\lesssim \ep^2 t^{-\frac{5}{2}}\langle t-r\rangle^{-\frac{1}{2}}  
		\]
		and
		\[
		| Z\Ff_0(t,x)|\lesssim \ep^2 t^{-\frac{5}{2}+\frac{\sigma}{2}}\langle t-r\rangle^{-\frac{1}{2}+\frac{\sigma}{2}} ,
		\]
		so from lemma \ref{supint} we obtain that
		\[
		|ZW_0^1(t,x)|\le C \ep^2 t^{-1}(t-r)^\sigma,
		\]
		for some constant $C=C(A,B)$ that depends quadratically on $A$ and $B$. }
	
{
		The remaining term $ZW_0^2$ is estimated via the Klainerman-Sobolev inequality of lemma \ref{Lem:KS} in terms of the higher order conformal energy of $W_0^2$ on hyperboloids. More precisely, we have that for any $(t,x)\in \q H_{[2, s_0]}$
		\[
		t^\frac{1}{2}s	|ZW_0^2(t,x)|\lesssim \ecin_2(s, W_0^2)^\frac{1}{2} +\sum_{|\gamma|\le 2}\left\| (s/t)Z \q Z^\gamma W_0^2 \right\|_{L^2_x(\q H_s \cap [T_1^s, T_2^s])}
		\]
		where $s^2=t^2-r^2$, $T_1^s = (s^2+1)/2$ is the time the hyperboloid $\q H_s$ intersects the cone $t=r+1$, and $T^s_2 = \sqrt{s^2+|y|^2}$ with $|y-x| = t/3$ (the second term in the above right hand side should be omitted when $T^s_2\le T_1^s$).
		Since the source term in the equation satisfied by $Z^\gamma W_0^2$ is supported in the exterior region $t<r+1$, we obtain from proposition \ref{Prop:Conf_En} with $\Wf_0=Z^\gamma W_0^2$, the bound \eqref{est_conf_enex} (which is valid also for $W_0^2$) with $T_0 = (s^2+1)/2$, and the smallness assumption on the initial data, that
		\[
		\ecin(s, Z^\gamma W_0^2)  \le  \ecin(2, Z^\gamma W_0^2) + C_0^2\ep^2\ln s \lesssim C_0^2\ep^2\ln s .\]
		Moreover, from lemma \ref{Lem:conf_ext_hyp} with $T_j=T_j^s$ for $j=1,2$ and the observation that $\ln (T_2^s/T_1^s)\lesssim \ln s$ we also derive that
		\[
		\sum_{|\gamma|\le 2}\left\| (s/t)Z \q Z^\gamma W_0^2 \right\|^2_{L^2_x(\q H_s \cap [T_1^s, T_2^s])} \lesssim C_0^2\ep^2\ln s.
		\]
		This concludes that
		\[
		|ZW_0^2(t,x)|\le \tilde{C}C_0\ep t^{-\frac{1}{2}}s^{-1+\sigma}
		\]
		for a universal constant $\tilde{C}$ and therefore
		\[
		|ZW_0(t,x)|\le |ZW_0^1(t,x)|+ |ZW_0^2(t,x)|\le \tilde{C} C_0\ep t^{-\frac{1}{2}}s^{-1+\sigma} + C\ep^2 t^{-1}(t-r)^\sigma \le B\ep t^{-1}s^\sigma
		\]
		if we choose $B$ sufficiently large so that $B\ge 2\tilde{C}C_0$ and $\epsilon_0$ sufficiently small so that $2C\ep \le B$.}
	
	\section{Energy Estimates in the Interior Region}
	\label{Sec: VF}

	The goal of this section is to propagate the interior energy bounds \eqref{booten1} and \eqref{booten2} on the two components $W_0$ and $\Wt$ of the solution $W$ to the Cauchy problem \eqref{eq_W}-\eqref{data_W}. We remind the reader that for any multi-index $\gamma$ the differentiated function $W^\gamma = (u^\gamma, v^\gamma)$ is solution to the equation \eqref{systeme_vf} with source term \eqref{nonlin_vf}, and that its zero-mode $W_0^\gamma$ solves the inhomogeneous wave equation \eqref{syst_u0} with source term \eqref{nonlin_0}.
	
	We first start by recovering an energy bound for the higher order conformal energy of $W_0$. Such bound follows from \eqref{booten1} and \eqref{booten2} as well as from the pointwise estimates obtained in section \ref{PointwiseEst}. It will be necessary for the propagation of \eqref{booten2}, and the computations that leads to it will be useful in the proof of proposition \ref{Prop:VF_energy0}.
	
	\begin{proposition}\label{Prop:conf_estimate}
		Assume the solution $W=(u,v)^T$ to \eqref{eq_W}-\eqref{data_W} satisfies the a-priori estimates \eqref{booten1}-\eqref{bootZ} in the interior region $\q H_{[2, s_0]}$ as well as the global exterior energy bounds \eqref{boot_energy} in the exterior region $\dext$. Then
		\begin{equation}\label{conf_bound}
			\sup_{[2, s]}\ec_k(s, W_0)\le C\ep s^{1+2\mu_k}, \qquad s\in [2, s_0], \quad k=\overline{0,4}
		\end{equation}
		where $\mu_k=\delta_k$ if $k\le 3$ and $\mu_4 = \delta_2+\delta_4$.
	\end{proposition}
	\proof
	We consider here only multi-indices $\gamma$ of type $(k,k)$, i.e. $\gamma=(0, \beta)$ with $|\beta|=k$ and $\mathcal{Z}^\gamma = Z^\beta$.
	Applying proposition \ref{Prop:Conf_En} with $\Wf_0 = W_0^\gamma$ and $\Ff_0 = \Ff_0^\gamma$ we derive that
	\begin{equation}\label{conf_ineq}
		\begin{aligned}
			\sup_{[2, s]}\ecin(\tau, W_0^\gamma) & \le  \ecin(2, W_0^\gamma) + \int_2^s \left\| \tau\Ff^\gamma_0\right\|_{L^2(\hin_\tau)}\ecin(\tau, W_0^\gamma)^\frac{1}{2}d \tau\\
			& \hspace{-10pt} + \int_{\mathcal{C}_{[2, s]}}\left|
			(t+r)(\partial_t W^\gamma_0+\partial_r W^\gamma_0) + 2W^\gamma_0\right|^2+(t-r)^2(|\nabla W^\gamma_0|^2-(\partial_r W^\gamma_0)^2)\, d\sigma dt .  
		\end{aligned}
	\end{equation}
	The integral over the boundary $\q C_{[2, s]}$ corresponds to the integral in the left hand side of \eqref{est_conf_enex} for $T_0 = (s^2+1)/2$ hence
	\begin{equation}\label{boundary term}
		\int_{\mathcal{C}_{[2, s]}}\left|
		(t+r)(\partial_t W^\gamma_0+\partial_r W^\gamma_0) + 2W^\gamma_0\right|^2+(t-r)^2(|\nabla W^\gamma_0|^2-(\partial_r W^\gamma_0)^2)\, d\sigma dt  \lesssim \ep^2  \ln s.
	\end{equation}
	The different contributions to the inhomogeneous term $\Ff_0^\gamma$ are estimated separately below:
	
	\textit{1. The $\partial_y u^{\gamma_1} \cdot\partial_y W^{\gamma_2}$ terms:} we will place the factor whose index has length smaller than $k/2$ in $L^2_yL^\infty_x$ and the remaining one in $L^2_{xy}$. Using the enhanced pointwise bounds \eqref{est_wt2} and \eqref{est_Zwt}, as well as the pointwise bound \eqref{ks3} in the case where $|\gamma_1|=|\gamma_2|=2$ (which appears only if $k=4$) we see that
	\[
	\sum_{|\gamma_1|+|\gamma_2|=k\le 4}\left\|\partial_y u^{\gamma_1} \cdot\partial_y W^{\gamma_2}\right\|_{L^1_yL^2_x(\hin_\tau)}\lesssim  \ep \tau^{-\frac{3}{2}}\left[\enint_{k}(\tau, W)^\frac{1}{2} + \tau^\sigma \enint_{k-1}(W) + \nu_4 \tau^{\delta_4}\enint_2(W)\right]
	\]
	where $\nu_4=1$ if $k=4$ and 0 otherwise.
	
	\smallskip
	\textit{2. The null terms:}
	we apply the decomposition \eqref{dec}
	to both factors $W^{\gamma_1}$ and $W^{\gamma_2}$ so that
	\[
	\mathbf{N}(W^{\gamma_1}, W^{\gamma_2}) = \mathbf{N}(W_0^{\gamma_1}, W_0^{\gamma_2})+ \mathbf{N}(\Wt^{\gamma_1}, W_0^{\gamma_2})+\mathbf{N}(W_0^{\gamma_1}, \Wt^{\gamma_2})+\mathbf{N}(\Wt^{\gamma_1}, \Wt^{\gamma_2}) .
	\]
	We express $\mathbf{N}(W_0^{\gamma_1}, W_0^{\gamma_2})$ using the representation formula \eqref{rewrite}
	\[
	\mathbf{N}(W_0^{\gamma_1}, W_0^{\gamma_2}) = \op W_0^{\gamma_1} \cdot \partial W_0^{\gamma_2} + \partial W_0^{\gamma_1}\cdot\op W_0^{\gamma_2} + \frac{t-r}{t}\partial W^{\gamma_1}_0 \cdot\partial W^{\gamma_2}_0
	\]
	and estimate all terms in the above right hand side using \eqref{ks1} and \eqref{ks2} as follows
	\[
	\begin{aligned}
		\left\|\int_{\m S^1} \mathbf{N}(W_0^{\gamma_1}, W_0^{\gamma_2})\, dy \right\|_{L^2_x(\hin_\tau)}& \lesssim \left\|(t/\tau)\op \mathcal{Z}^{\le 3}W_0\right\|_{L^\infty_x(\hin_\tau)}\left\|(\tau/t)\partial \mathcal{Z}^{\le k}W_0\right\|_{L^2_x(\hin_\tau)} \\
		& + \left\|\op \mathcal{Z}^{\le k}W_0\right\|_{L^2_x(\hin_\tau)}\left\|\partial\mathcal{Z}^{\le 3}W_0\right\|_{L^\infty_x(\hin_\tau)} \\
		& \lesssim \ep \tau^{-\frac{3}{2}}\enint_{k}(\tau, W)^\frac{1}{2}.
	\end{aligned}
	\]
	We neglect the null structure for all the remaining quadratic forms.
	The products involving $W_0\times\Wt$ are equivalent given the range of $\gamma_1$ and $\gamma_2$ and are estimated using \eqref{ks1} whenever $|\gamma_1|\le 3$ and \eqref{est_wt2} otherwise
	\[
	\begin{aligned}
		\left\|\int_{\m S^1} \mathbf{N}(W_0^{\gamma_1}, \Wt^{\gamma_2})\, dy\right\|_{L^2_x(\hin_\tau)} & \le \|(t/\tau)\partial \mathcal{Z}^{\le 3}W_0\|_{L^\infty_x(\hin_\tau)}\left\|(\tau/t)\partial \mathcal{Z}^{\le k}\Wt\right\|_{L^2_{xy}(\hin_\tau)} \\
		&+ \left\|(\tau/t)\partial \mathcal{Z}^{\le 4}W_0\right\|_{L^2_{x}(\hin_\tau)}\left\|(t/\tau)\partial \Wt\right\|_{L^\infty_{xy}(\hin_\tau)} \\
		& \lesssim \ep \tau^{-\frac{3}{2}}\enint_{k}(\tau, W)^\frac{1}{2}.
	\end{aligned}
	\]
	The quadratic term involving products $\Wt\times\Wt$ is estimated as in step 1.
	
	Summing up the estimates in step 1 and 2, using the a-priori energy bounds \eqref{booten1}-\eqref{booten2}, and choosing $\sigma, \delta_{k-1}$ small so that $\sigma+\delta_{k-1}\le \delta_k$ we obtain that
	\begin{equation}\label{source_F0}
		\left\|\Ff^\gamma_0\right\|_{L^2_x(\hin_\tau)}\lesssim \ep^2 \tau^{-\frac{3}{2}+\mu_k} \qquad \text{with } 
		\mu_k = \begin{cases}
			\delta_k \quad & \text{if } k\le 3 \\
			\delta_4+\delta_2 \quad &\text{if } k=4.
		\end{cases}
	\end{equation}
	Plugging the above estimate together with \eqref{boundary term} into \eqref{conf_ineq} and using the smallness assumptions on the initial data we finally obtain
	\[
	\begin{aligned}
		\sup_{[2, s]}\ecin_{k}(\tau, W_0) & \lesssim  \ecin_{k}(2, W_0) + \ep^2 s^{2C\ep}\ln s  + \int_2^s \ep^2 \tau^{-\frac{1}{2}+\mu_k}\ecin_{k}(\tau, W_0)^\frac{1}{2}d \tau \\
		& \lesssim \ep^2 s^{1+2\mu_k} + \ep^2\sup_{[2, s]}\ecin_{k}(\tau, W_0)
	\end{aligned}
	\]
	and the result of the proposition follows choosing $\ep$ sufficiently small.
	\endproof
	
	We observe that from lemma \ref{Lem:KS}, the conformal energy bound \eqref{conf_bound} and lemma \ref{Lem:conf_ext_hyp} with $T_1 = (s^2+1)/2$ and $T_2 = \sqrt{s^2+|y|^2}$ for $|y-x|=t/3$, we have the following additional pointwise bound for the $Z$ derivatives of $W_0$ in $\q H_{[2, s_0]}$
	\begin{equation}\label{ks5}
		|Z Z^jW_0(t,x)|\lesssim \ep t^{-\frac{1}{2}}s^{-\frac{1}{2}+\mu_{j+2}}\, \quad \text{for } j=\overline{0,2}, \qquad \text{where } \mu_k = \begin{cases}
			\delta_k  \quad & \text{if } k\le 3 \\
			\delta_4+\delta_2 \quad &\text{if } k=4.
		\end{cases}
	\end{equation}
	We now have all the ingredients to propagate the a-priori energy bounds \eqref{booten1} and \eqref{booten2}.
	
	\begin{proposition}\label{Prop:VF_energy0}
		There exists a constant $A>0$ sufficiently large, some small parameters $\sigma\ll \delta_k\ll \delta_{k+1}$ for $k=\overline{1,4}$ and $\ep_0>0$ sufficiently small such that, if for any $0<\ep<\ep_0$ the solution $W=(u,v)^T$ to the Cauchy problem \eqref{eq_W}-\eqref{data_W} satisfies the a-priori bounds \eqref{booten1}-\eqref{bootZ} in the hyperbolic region $\hin_{[2, s_0]}$ and the global exterior energy bounds \eqref{boot_energy} in the exterior region $\dext$, then it also satisfies the enhanced energy bound
		\begin{equation} \label{improved_en}
			\enint_{5}(s, W_0) \le A^2\ep^2, \qquad s\in [2, s_0].
		\end{equation}
	\end{proposition}
	\proof
	For any $|\gamma|\le 5$, the equation \eqref{syst_u0} satisfied by $W_0^\gamma$ has the same structure as the inhomogeneous wave equation \eqref{linear_wave0}, therefore proposition \ref{prop:linear_wave0} with $\Wf_0=W_0^\gamma$ and $\Ff_0 = \Ff_0^\gamma$ implies that
	\[
	\enint(s, W_0^\gamma) \lesssim \enint(2, W^\gamma_0) + \int_{\CT} |\mathcal{T}W^\gamma_0|^2  d\sigma dt + \int_{2}^{s} \|\Ff_0^\gamma\|_{L^2_x(\hin_\tau)}\enint(\tau, W_0^\gamma)^\frac{1}{2}\, d\tau
	\]
	for all $s\in [2, s_0]$. The implicit constant in the above right hand side is independent of $s_0$. 
	
	The integral over the boundary $\CT$ has already been proved to be finite and small in proposition \ref{prop:est_enex_hyp}. The source term $\Ff_0^\gamma$ has also already been estimated in the case $|\gamma|\le 4$ (see \eqref{source_F0} in the previous proposition). When $|\gamma|=5$ we simply use the estimate \eqref{ks3} to deduce
	\[
	\sum_{\substack{|\gamma_1|+|\gamma_2|=5 \\ |\gamma_1|\ge 1}}\left\|\partial_y u^{\gamma_1}\cdot\partial_yW^{\gamma_2}\right\|_{L^1_yL^2_x(\hin_\tau)} + \left\|\mathbf{N}(\Wt^{\gamma_1}, \Wt^{\gamma_2})\right\|_{L^1_yL^2_x(\hin_\tau)} \lesssim \ep \tau^{-\frac{3}{2}+\delta_4}\enint_{5}(t, W)^\frac{1}{2}
	\]
	and \eqref{ks1}, \eqref{ks2} and \eqref{est_wt2} 
	\[
	\sum_{\substack{|\gamma_1|+|\gamma_2|\le 5}}	\left\|\mathbf{N}(W_0^{\gamma_1}, W_0^{\gamma_2})\right\|_{L^1_yL^2_x(\hin_\tau)} + \left\|\mathbf{N}(W_0^{\gamma_1}, \Wt^{\gamma_2})\right\|_{L^1_yL^2_x(\hin_\tau)}\lesssim \ep \tau^{-\frac{3}{2}}\enint_{5}(s, W)^\frac{1}{2}.
	\]
	Therefore
	\begin{equation}\label{estF0}
		\sum_{|\gamma|\le 5}	\|\Ff_0^\gamma\|_{L^2_x(\hin_\tau)}\lesssim \ep \tau^{-\frac{3}{2}+\delta_4}\enint_{5}(\tau, W)^\frac{1}{2}
	\end{equation}
	and from the a-priori energy bound \eqref{booten2} and the exterior energy bound \eqref{en_ext1} we obtain that, for some fixed $K>1$ and some universal constant $C>0$
	\[
	\begin{aligned}
		\enint_{5}(s, W_0)& \le \enint_{5}(2, W_0) + \frac{C^2_0\ep^2}{K} + \int_{2}^{s} A^2C\ep^3 \tau^{-\frac{3}{2}+\delta_4 + \delta_5}\, d\tau \\
		& \le \enint_{5}(2, W_0)  + \frac{C^2_0\ep^2}{K} + A^2C\ep^3.
	\end{aligned}
	\]
	The wished improved energy bound then follows choosing $K=3$, $\ep_0$ small such that $3C \ep^2_0<1$ and $A\ge C_0$ sufficiently large so that
	\[
	\enint_{5}(2, W_0)\le \frac{A^2\ep_0^2}{3}.
	\]
	
	%
	%
	%
	%
	\endproof

	\begin{proposition} \label{Prop:energy_KG}
		There exists a constant $A>0$ sufficiently large, some small parameters $\sigma\ll \delta_k\ll \delta_{k+1}$ for $k=\overline{1,4}$ and $\ep_0>0$ sufficiently small such that, if for any $0<\ep<\ep_0$ the solution $W=(u,v)^T$ to the Cauchy problem \eqref{eq_W}-\eqref{data_W} satisfies the a-priori bounds \eqref{booten1}-\eqref{bootZ} in the hyperbolic region $\hin_{[2, s_0]}$ and the global exterior energy bounds \eqref{boot_energy} in the exterior region $\dext$, then it also satisfies the enhanced energy bound
		\begin{equation}\label{improv_enWt}
			\enint_{5, k}(s, \Wt) \le A^2\ep^2
			\, s^{2\delta_k}, \qquad s \in [2, s_0].
		\end{equation}
	\end{proposition}
	\proof
	We start by considering a multi-index $\gamma$ of type $(n,k)$ with $k\le n\le 5$ and compare the equation satisfied by $\Wt^\gamma= (\mathtt{u}^\gamma, \mathtt{v}^\gamma)^T$ with the linear inhomogeneous equation \eqref{linearized}. We have that 
	\[
	\Box_{x,y} \Wt^\gamma + u\, \partial^2_y \Wt^\gamma = \mathbf{F}^\gamma -  \mathbf{F}^\gamma_0, \qquad (t,x,y)\in \m R^{1+3}\times \m S^1
	\]
	so applying proposition \ref{Prop:en_lin} with $\Wf = \Wt^\gamma$ and $\Ff$ replaced by $\Ff^\gamma-\Ff^\gamma_0$ we derive the following inequality
	\begin{equation}\label{ineq_WtX}
		\begin{aligned}
			\enint(s, \Wt^\gamma)& \lesssim \enint(2, \Wt^\gamma) + \iint_{\CT} |\mathcal{T}\Wt^\gamma|^2 + |\partial_y \Wt^\gamma|^2 d\sigma dydt \\
			&+ \int_{2}^{s}\|\Ff^\gamma- \Ff^\gamma_0\|_{L^2_{xy}(\hin_\tau)} \enint(\tau, \Wt^\gamma)^\frac{1}{2}\, d\tau.
		\end{aligned}
	\end{equation}
	All quadratic terms in $\Ff^\gamma$ can be estimated similarly to those in $\Ff_0^\gamma$ -- which satisfy the bound \eqref{estF0} -- except for the products $u^{\gamma_1}_0\cdot\partial^2_y W^{\gamma_2}$ when $\mathcal{Z}^{\gamma_1} = Z^{\beta_1}$ is a pure product of Klainerman vector fields, obtained after using the decomposition \eqref{dec} on $u^{\gamma_1}$. The reason is that the $L^2_x$ norm of $Z^{\beta_1}u_0$ is not controlled by the higher order energies of $W_0$ but rather by the higher order conformal energy of $W_0$ (up to multiplication with a factor $(s/t)$).
	Those are the only terms that need to be analyzed here since the argument already used in proposition \ref{Prop:VF_energy0} coupled with the Poincar\'e's inequality shows that 
	\[
	\sum_{\substack{|\gamma_1|+|\gamma_2|=|\gamma| \\ |\gamma_2|<|\gamma|}} \left\| \ut^{\gamma_1}\cdot\partial^2_y \Wt^{\gamma_2}\right\|_{L^2_{xy}(\hin_\tau)} + \sum_{|\gamma_1|+|\gamma_2|\le |\gamma|}\left\|\mathbf{N}(W^{\gamma_1}, W^{\gamma_2})\right\|_{L^2_{xy}(\hin_\tau)} \lesssim  \ep \tau^{-\frac{3}{2}+\delta_4} \enint_{5}(\tau, W)^\frac{1}{2}.
	\] 
	Depending on the starting multi-index $\gamma_1 = (\alpha_1, \beta_1)$ we distinguish between two cases:
	
	\smallskip
	\textit{1. The case $|\alpha_1|>0$}: here we can
	write $u_0^{\gamma_1}=\partial u_0^{\tilde{\gamma}_1}$ for some other index $\tilde{\gamma}_1$ with $|\tilde{\gamma}_1| = |\gamma_1|-1$ and estimate the product $\partial u_0^{\tilde{\gamma}_1}\cdot\partial^2_y \Wt^{\gamma_2}$ as done for $\mathbf{N}(W_0^{\gamma_1}, \Wt^{\gamma_2})$ in the proof of proposition \ref{Prop:VF_energy0}.
	
	\smallskip
	\textit{2. The case $|\alpha_1|=0$}: in this case $\gamma_1$ is of type $(k_1, k_1)$ with $k_1\le k\le n$ and $\gamma_2$ is of type $(n-k_1, k_2)$ with $k_1+k_2=k$. We
	write $u_0^{\gamma_1}=Zu_0^{\tilde{\gamma}_1}$ for some other index $\tilde{\gamma}_1$ of type $(k_1-1, k_1-1)$ and distinguish between the different values $k_1$ can take.
	
	If $k_1  = n$ then $|\gamma_2|=0$ and we get from the pointwise bound \eqref{est_wt2} and the conformal energy bound \eqref{conf_bound} that
	\[
	\left\| Zu_0^{\tilde{\gamma}_1}\cdot\partial^2_y \Wt \right\|_{L^2_{xy}(\hin_\tau)}\lesssim \left\|\partial^2_y \Wt \right\|_{L^2_yL^\infty_x(\hin_\tau)}\ec_{n-1}(\tau, W_0)^\frac{1}{2}\lesssim \ep^2 \tau^{-1 + \mu_{n-1}}.
	\]
	
	If $k_1 = n-1$ we have $|\gamma_2|=1$, so from the pointwise bounds \eqref{est_w3}, \eqref{est_Zwt2} and the conformal energy bound \eqref{conf_bound}
	\[
	\left\| Zu_0^{\tilde{\gamma}_1}\cdot\partial^2_y \Wt^{\gamma_2} \right\|_{L^2_{xy}(\hin_\tau)}\lesssim \left\|\partial^2_y \Wt^{\gamma_2} \right\|_{L^2_yL^\infty_x(\hin_\tau)}\ec_{n-2}(\tau, W_0)^\frac{1}{2}\lesssim \ep^2 \tau^{-1 + \mu_{n-2}+\sigma}.
	\]
	
	If $k_1 = n-2$ then $|\gamma_2|=2$ and $n=\overline{3,5}$. We distinguish the following cases depending on if $\mathcal{Z}^{\gamma_2}$ is of the form $\partial^2, \partial Z$ or $Z^2$.
	
	\textit{a. The case $\q Z^{\gamma_2} = \partial^2$}: here we use \eqref{ks3} and \eqref{conf_bound} to derive that
	\[
	\left\| Zu_0^{\tilde{\gamma}_1}\cdot \partial^2_y\Wt ^{\gamma_2}\right\|_{L^2_{xy}(\hin_\tau)}\lesssim \left\|\partial^2_y \partial^2 \Wt\right\|_{L^2_yL^\infty_x(\hin_\tau)}\ec_{n-3}(\tau, W_0)^\frac{1}{2}\lesssim \ep^2 \tau^{-1 + \mu_{n-3}+\delta_2}.
	\]
	
	\textit{b. The case $\q Z^{\gamma_2} = \partial Z$}: when $n=3$ we have $k_1=1$ and hence $|\tilde{\gamma}_1|=0$. From the pointwise bound \eqref{bootZ} and the a-priori energy bound \eqref{booten2} we find
	\[
	\left\| Zu_0 \cdot \partial^2_y \Wt^{\gamma_2}\right\|_{L^2_{xy}(\hin_\tau)}\lesssim \|Zu_0\|_{L^\infty(\hin_\tau)}\enint_{3,1}(\tau, W)^\frac{1}{2}\lesssim \ep^2 \tau^{-1+\sigma + \delta_1}.
	\]
	
	When $n=\overline{4,5}$ we use \eqref{ks3} and \eqref{conf_bound} to derive that
	\[
	\left\| Zu_0^{\tilde{\gamma}_1}\cdot \partial^2\Wt^{\gamma_2} \right\|_{L^2_{xy}(\hin_\tau)}\lesssim \left\|\partial^2_y \partial Z \Wt\right\|_{L^2_yL^\infty_x(\hin_\tau)}\ec_{n-3}(\tau, W_0)^\frac{1}{2}\lesssim \ep^2 \tau^{-1 + \mu_{n-3}+\delta_3}.
	\]
	
	\textit{c. The case $\q Z^{\gamma_2}=Z^2$}: when $n=3$ we estimate $Z u_0^{\tilde{\gamma}_1}$ using \eqref{bootZ} while for $n=\overline{4,5}$ we use \eqref{ks5}. Together with the a-priori bound \eqref{booten2} these give that
	\[
	\left\| Zu^{\tilde{\gamma_1}}_0 \cdot \partial^2_y \Wt^{\gamma_2}\right\|_{L^2_{xy}(\hin_\tau)}\lesssim \|Zu^{\tilde{\gamma_1}}_0\|_{L^\infty(\hin_\tau)}\enint_{3,2}(\tau, W)^\frac{1}{2}\lesssim \ep^2 \tau^{-1+\nu_{n} + \delta_2}.
	\]
	where $\nu_3=\sigma$, $\nu_4 = \mu_3$ and $\nu_5 = \mu_4$.
	
	The final cases to treat correspond to $k_1 = n-3$ and $|\gamma_2|=3$, which appears for $n=\overline{4,5}$, and $k=n-4$ and $|\gamma_2|=4$ which only appears when $n=5$. If $k_1=1$ then $|\gamma_2|=n-1$ and $n=\overline{4,5}$, so using \eqref{booten2} and \eqref{bootZ} we obtain
	\[
	\left\|Zu_0 \cdot \partial^2_y \Wt^{\gamma_2}\right\|_{L^2_{xy}(\hin_\tau)}\lesssim \|Zu_0\|_{L^\infty(\hin_\tau)}\enint_{n,n-1}(\tau, W)^\frac{1}{2}\lesssim \ep^2 \tau^{-1+\sigma+\delta_{n-1}}, \quad n =\overline{4,5}.
	\]
	If $k_1=2$ then $|\gamma_2|=3$ and $n=4$, so the bounds \eqref{booten2} and \eqref{ks5} yield
	\[
	\left\|Z^2u_0 \cdot \partial^2_y \Wt^{\gamma_2}\right\|_{L^2_{xy}(\hin_\tau)}\lesssim \|Z^2u_0\|_{L^\infty(\hin_\tau)}\enint_{4,3}(\tau, W)^\frac{1}{2}\lesssim \ep^2 \tau^{-1+2\delta_3}.
	\]
	Choosing appropriately $\sigma\ll \delta_k\ll \delta_{k+1}$ for $k=\overline{1,4}$ so that $\delta_{k+1}$ is bigger than some linear combination of $\sigma$ and $\delta_j$ with $j\le k$, we finally obtain that
	\[
	\left\|u_0^{\gamma_1}\cdot \partial^2_y \Wt^{\gamma_2}\right\|_{L^2_{xy}(\hin_\tau)}\lesssim \ep^2 \tau^{-1+\delta_k}
	\]
	with an implicit constant that depends on $A$ and $B$. The same bound holds then true for $\Ff^\gamma$ so plugging it into \eqref{ineq_WtX} together with the estimate for $\Ff_0^\gamma$, the exterior energy estimate \eqref{en_ext1}, and the a-priori energy bounds \eqref{booten2}, gives us that
	\[
	\enint_{5,k}(s, \Wt)\le  C\enint_{5,k}(2, \Wt) + \frac{CC_0^2\ep^2}{2} + \int_2^s  A^2C\ep^3 \tau^{-1+2\delta_k}\, d\tau.
	\]
	The end of the proof follows finally by choosing appropriately the constants $A$ and $\ep_0$.
	\endproof

	\bibliographystyle{abbrv}
	\bibliography{Biblio_Anna}{}

\begin{thebibliography}{10}

\bibitem{Alinhac06}
S.~Alinhac.
\newblock Semilinear hyperbolic systems with blowup at infinity.
\newblock {\em Indiana Univ. Math. J.}, 55(3):1209--1232, 2006.

\bibitem{ABWY}
L.~Andersson, P.~Blue, Z.~Wyatt, and S.-T. Yau.
\newblock Global stability of spacetimes with supersymmetric compactifications.
\newblock Preprint, arXiv:2006.00824v1, 2020.

\bibitem{Bachelot88}
A.~Bachelot.
\newblock Probl\`eme de {C}auchy global pour des syst\`emes de
  {D}irac-{K}lein-{G}ordon.
\newblock {\em Ann. Inst. H. Poincar\'{e} Phys. Th\'{e}or.}, 48(4):387--422,
  1988.

\bibitem{CK93}
D.~Christodoulou and S.~Klainerman.
\newblock {\em The global nonlinear stability of the {M}inkowski space},
  volume~41 of {\em Princeton Mathematical Series}.
\newblock Princeton University Press, Princeton, NJ, 1993.

\bibitem{DLFW}
S.~Dong, P.~G. LeFloch, and Z.~Wyatt.
\newblock Global evolution of the {$\rm U(1)$} {H}iggs {B}oson: nonlinear
  stability and uniform energy bounds.
\newblock {\em Ann. Henri Poincar\'{e}}, 22(3):677--713, 2021.

\bibitem{DW20}
S.~{Dong} and Z.~{Wyatt}.
\newblock {Stability of a coupled wave-Klein-Gordon system with quadratic
  nonlinearities.}
\newblock {\em {J. Differ. Equations}}, 269(9):7470--7497, 2020.

\bibitem{DW20_2d}
S.~Dong and Z.~Wyatt.
\newblock Two dimensional wave–{K}lein-{G}ordon equations with semilinear
  nonlinearities.
\newblock Preprint, 2020.

\bibitem{ettinger}
B.~Ettinger.
\newblock Well-posedness of the equation for the three-form field in
  eleven-dimensional supergravity.
\newblock {\em Trans. Amer. Math. Soc.}, 367(2):887--910, 2015.

\bibitem{Georgiev:system}
V.~Georgiev.
\newblock Global solution of the system of wave and {K}lein-{G}ordon equations.
\newblock {\em Math. Z.}, 203(4):683--698, 1990.

\bibitem{IS2019}
M.~Ifrim and A.~Stingo.
\newblock Almost global well-posedness for quasilinear strongly coupled
  wav-{K}lein-{G}ordon systems in two space dimensions.
\newblock Preprint, arXiv:1910.12673v1, 2019.

\bibitem{IP2}
A.~Ionescu and B.~Pausader.
\newblock The {E}instein-{K}lein-{G}ordon coupled system: global stability of
  the {M}inkowski solution.
\newblock To appear in \textit{Annals of Math Studies}, 2020.

\bibitem{IoPa}
A.~D. Ionescu and B.~Pausader.
\newblock On the global regularity for a wave-{K}lein-{G}ordon coupled system.
\newblock {\em Acta Math. Sin. (Engl. Ser.)}, 35(6):933--986, 2019.

\bibitem{Kaluza21}
T.~Kaluza.
\newblock {Zum Unit\"atsproblem der Physik}.
\newblock {\em Int. J. Mod. Phys. D}, 27(14):1870001, 2018.

\bibitem{katayama:null_condition}
S.~Katayama, T.~Ozawa, and H.~Sunagawa.
\newblock A note on the null condition for quadratic nonlinear {K}lein-{G}ordon
  systems in two space dimensions.
\newblock {\em Comm. Pure Appl. Math.}, 65(9):1285--1302, 2012.

\bibitem{klainerman:global_existence}
S.~Klainerman.
\newblock Global existence of small amplitude solutions to nonlinear
  {K}lein-{G}ordon equations in four space-time dimensions.
\newblock {\em Comm. Pure Appl. Math.}, 38(5):631--641, 1985.

\bibitem{klainerman:null_condition}
S.~Klainerman.
\newblock The null condition and global existence to nonlinear wave equations.
\newblock In {\em Nonlinear systems of partial differential equations in
  applied mathematics, {P}art 1 ({S}anta {F}e, {N}.{M}., 1984)}, volume~23 of
  {\em Lectures in Appl. Math.}, pages 293--326. Amer. Math. Soc., Providence,
  RI, 1986.

\bibitem{klainerman_wang_yang}
S.~Klainerman, Q.~Wang, and S.~Yang.
\newblock Global solution for massive {M}axwell-{K}lein-{G}ordon equations.
\newblock {\em Comm. Pure Appl. Math.}, 73(1):63--109, 2020.

\bibitem{Klein26}
O.~Klein.
\newblock {Quantum Theory and Five-Dimensional Theory of Relativity. (In German
  and English)}.
\newblock {\em Z. Phys.}, 37:895--906, 1926.

\bibitem{LeFloch_Ma}
P.~G. LeFloch and Y.~Ma.
\newblock {\em The hyperboloidal foliation method}, volume~2 of {\em Series in
  Applied and Computational Mathematics}.
\newblock World Scientific Publishing Co. Pte. Ltd., Hackensack, NJ, 2014.

\bibitem{LeFloch-Ma:global_nl_stability}
P.~G. LeFloch and Y.~Ma.
\newblock The global nonlinear stability of {M}inkowski space for
  self-gravitating massive fields.
\newblock {\em Comm. Math. Phys.}, 346(2):603--665, 2016.

\bibitem{lefloch-ma:global_stability_2}
P.~G. LeFloch and Y.~Ma.
\newblock {\em The global nonlinear stability of {M}inkowski space for
  self-gravitating massive fields}, volume~3 of {\em Series in Applied and
  Computational Mathematics}.
\newblock World Scientific Publishing Co. Pte. Ltd., Hackensack, NJ, 2018.

\bibitem{LeskyRacke}
P.~H. Lesky and R.~Racke.
\newblock Nonlinear wave equations in infinite waveguides.
\newblock {\em Comm. Partial Differential Equations}, 28(7-8):1265--1301, 2003.

\bibitem{LR10}
H.~Lindblad and I.~Rodnianski.
\newblock The global stability of {M}inkowski space-time in harmonic gauge.
\newblock {\em Ann. of Math. (2)}, 171(3):1401--1477, 2010.

\bibitem{ma:2D_semilinear}
Y.~Ma.
\newblock Global solutions of non-linear wave-{K}lein-{G}ordon system in two
  space dimension: semi-linear interactions.
\newblock Preprint, arXiv:1712.05315, 2017.

\bibitem{ma:2D_quasilinear}
Y.~Ma.
\newblock Global solutions of quasilinear wave-{K}lein-{G}ordon system in
  two-space dimension: completion of the proof.
\newblock {\em J. Hyperbolic Differ. Equ.}, 14(4):627--670, 2017.

\bibitem{ma:2D_tools}
Y.~Ma.
\newblock Global solutions of quasilinear wave-{K}lein-{G}ordon system in
  two-space dimension: technical tools.
\newblock {\em J. Hyperbolic Differ. Equ.}, 14(4):591--625, 2017.

\bibitem{ma:1D_semilinear}
Y.~Ma.
\newblock Global solutions of nonlinear wave-{K}lein-{G}ordon system in one
  space dimension.
\newblock {\em Nonlinear Anal.}, 191:111641, 57, 2020.

\bibitem{Ma2020}
Y.~Ma.
\newblock Global solutions of nonlinear wave-{K}lein-{G}ordon system in two
  spatial dimensions: a prototype of strong coupling case.
\newblock {\em J. Differential Equations}, 287:236--294, 2021.

\bibitem{MSS05}
J.~Metcalfe, C.~D. Sogge, and A.~Stewart.
\newblock Nonlinear hyperbolic equations in infinite homogeneous waveguides.
\newblock {\em Comm. Partial Differential Equations}, 30(4-6):643--661, 2005.

\bibitem{MS2008}
J.~Metcalfe and A.~Stewart.
\newblock Almost global existence for quasilinear wave equations in waveguides
  with {N}eumann boundary conditions.
\newblock {\em Trans. Amer. Math. Soc.}, 360(1):171--188, 2008.

\bibitem{OTT_KGZ}
T.~Ozawa, K.~Tsutaya, and Y.~Tsutsumi.
\newblock Normal form and global solutions for the {K}lein-{G}ordon-{Z}akharov
  equations.
\newblock {\em Ann. Inst. H. Poincar\'{e} Anal. Non Lin\'{e}aire},
  12(4):459--503, 1995.

\bibitem{stingo_WKG}
A.~Stingo.
\newblock Global existence of small amplitude solutions for a model quadratic
  quasi-linear coupled wave-{K}lein-{G}ordon system in two space dimension,
  with mildly decaying {C}auchy data.
\newblock To appear in Memoirs of the AMS.

\bibitem{Tataru2002}
D.~Tataru.
\newblock Strichartz estimates for second order hyperbolic operators with
  nonsmooth coefficients. {III}.
\newblock {\em J. Amer. Math. Soc.}, 15(2):419--442, 2002.

\bibitem{T96_KGZ}
K.~Tsutaya.
\newblock Global existence of small amplitude solutions for the
  {K}lein-{G}ordon-{Z}akharov equations.
\newblock {\em Nonlinear Anal.}, 27(12):1373--1380, 1996.

\bibitem{T03_DP}
Y.~Tsutsumi.
\newblock Global solutions for the {D}irac-{P}roca equations with small initial
  data in {$3+1$} space time dimensions.
\newblock {\em J. Math. Anal. Appl.}, 278(2):485--499, 2003.

\bibitem{T03_MH}
Y.~Tsutsumi.
\newblock Stability of constant equilibrium for the {M}axwell-{H}iggs
  equations.
\newblock {\em Funkcial. Ekvac.}, 46(1):41--62, 2003.

\bibitem{Q.Wang}
Q.~Wang.
\newblock Global existence for the {E}instein equations with massive scalar
  fields.
\newblock 2015.
\newblock Lecture at the workshop \textit{{M}athematical {P}roblems in
  {G}eneral {R}elativity}.

\bibitem{Q.Wang:E-KG}
Q.~Wang.
\newblock An intrinsic hyperboloid approach for {E}instein {K}lein-{G}ordon
  equations.
\newblock {\em J. Differential Geom.}, 115(1):27--109, 2020.

\bibitem{WITTEN}
E.~Witten.
\newblock Instability of the {K}aluza-{K}lein vacuum.
\newblock {\em Nuclear Physics B}, 195(3):481 -- 492, 1982.

\bibitem{Wyatt18}
Z.~Wyatt.
\newblock The weak null condition and {K}aluza-{K}lein spacetimes.
\newblock {\em J. Hyperbolic Differ. Equ.}, 15(2):219--258, 2018.

\end{thebibliography}
	
\end{document}